\pdfoutput=1

\documentclass[11pt]{amsart}
\usepackage{hyperref,amsmath,amssymb,amscd,amsfonts,latexsym,epsfig,psfrag,graphicx,url,color,xr,mathtools,float}
\usepackage{fullpage}
\usepackage[all]{xy}

\newtheorem{theorem}{Theorem}[section]

\newtheorem{conjecture}[theorem]{Conjecture}
\newtheorem{proposition}[theorem]{Proposition}
\newtheorem{lemma}[theorem]{Lemma}

\theoremstyle{definition}
\newtheorem{definition}[theorem]{Definition}
\theoremstyle{remark}
\newtheorem{remark}[theorem]{Remark}

\newtheorem{example}[theorem]{Example}

\newcommand\cA{\mathcal{A}}

\newcommand{\cB}{\mathcal{B}}
\newcommand\cC{\mathcal{C}}

\newcommand{\J}{\mathcal{J}}

\newcommand\M{\mathcal{M}}

\newcommand\cS{\mathcal{S}}

\newcommand{\N}{\mathbb{N}}
\newcommand{\R}{\mathbb{R}}
\renewcommand{\H}{\mathbb{H}}
\newcommand{\C}{\mathbb{C}}
\newcommand{\Z}{\mathbb{Z}}
\newcommand{\Q}{\mathbb{Q}}

\newcommand{\RP}{\mathbb{RP}}
\newcommand{\CP}{\mathbb{CP}}

\newcommand{\bn}{\mathbf{n}}

\newcommand{\fJ}{\mathfrak{J}}

\newcommand{\on}{\operatorname}

\renewcommand\Re{\on{Re}}
\renewcommand\Im{\on{Im}}

\newcommand\pt{\on{pt}}

\newcommand\loc{{\on{loc}}}

\newcommand{\re}{\on{re}}
\newcommand{\im}{\on{im}}

\newcommand{\Fuk}{\on{Fuk}}
\newcommand{\comp}{C^2}

\newcommand{\Jac}{\on{Jac}}
\newcommand{\Id}{\on{Id}}
\newcommand{\compnt}{{\on{comp}}}
\newcommand{\marked}{{\on{marked}}}
\newcommand{\seam}{{\on{seam}}}
\newcommand{\node}{{\on{node}}}
\newcommand{\ext}{{\on{ext}}}

\newcommand\qu{/\kern-.7ex/} 
\newcommand\lqu{\backslash \kern-.7ex \backslash}

\def\tint{{\textstyle\int}}

\def\half{\mbox{\fontsize{10}{10}\selectfont $\frac 12$}}

\newcommand{\ol}{\overline}
\newcommand{\ul}{\underline}

\newcommand{\sr}{\stackrel}
\newcommand{\wh}{\widehat}
\newcommand{\wt}{\widetilde}

\newcommand{\lan}{\langle}
\newcommand{\ran}{\rangle}

\newcommand{\nm}[1]{\left|#1\right|}
\newcommand{\const}{\mathrm{const}}

\def\d{\delta}

\newcommand{\eps}{\epsilon}

\def\i{\iota}

\def\n{\nu}

\def\om{\omega}

\def\Si{\Sigma}

\def\rT{{\rm T}}
\def\rd{{\rm d}}

\renewcommand{\d}{{\rm d}}

\def\lra{\longrightarrow}

\newcommand{\less}{{\smallsetminus}}

\newenvironment{itemlist}
   {
      \begin{list}
         {$\bullet$}
         {
            \setlength{\leftmargin}{1em}
            \setlength{\itemsep}{.5ex}
         }
   }
   {
      \end{list}
   }

\newcounter{qcounter}

\newenvironment{enumlist}
   {
      \begin{list}
         {\bf\Alph{qcounter})} 
         {
         \usecounter{qcounter}
                     \setlength{\itemsep}{.5ex}
            \setlength{\leftmargin}{1em}
         }
   }
   {
      \end{list}
   }
   
\newcommand\quotient[2]{
        \mathchoice
            {
                \text{\raise1ex\hbox{$#1$}\Big/\lower1ex\hbox{$#2$}}%
            }
            {
                #1\,/\,#2
            }
            {
                #1\,/\,#2
            }
            {
                #1\,/\,#2
            }
    }

\newcommand\quoti[2]{
                \text{\raise1ex\hbox{$#1$}/\lower1ex\hbox{$\scriptstyle#2$}}
  }

\newcommand\quot[2]{
                \text{\raise1ex\hbox{$#1\!\!$}/\lower1ex\hbox{$\!\scriptstyle#2$}}
  }

\newcommand\quo[2]{
                \text{\raise.8ex\hbox{$\scriptstyle#1\!$}/\lower.8ex\hbox{$\!\scriptstyle#2$}}
  }

\newcommand\qq[2]{
                \text{\raise.8ex\hbox{$#1\!$}/\lower.8ex\hbox{$#2$}}
}

\begin{document}

\title{Gromov compactness for squiggly strip shrinking in pseudoholomorphic quilts}

\author{Nathaniel Bottman}
\address{School of Mathematics, Institute for Advanced Study,
1 Einstein Dr, Princeton, NJ 08540}
\email{\href{mailto:nbottman@math.ias.edu}{nbottman@math.ias.edu}}

\author{Katrin Wehrheim}
\address{Department of Mathematics, University of California,
Berkeley, CA 94720}
\email{\href{mailto:katrin@math.berkeley.edu}{katrin@math.berkeley.edu}}

\begin{abstract}  
We establish a Gromov compactness theorem for strip shrinking in pseudoholomorphic quilts when composition of Lagrangian correspondences is immersed.
In particular, we show that figure eight bubbling occurs in the limit, argue that this is a codimension-$0$ effect, and predict its algebraic consequences -- geometric composition extends to a curved $A_\infty$-bifunctor, in particular the associated Floer complexes are isomorphic after a figure eight correction of the bounding cochain.
An appendix with Felix Schm\"{a}schke provides examples of nontrivial figure eight bubbles.
\end{abstract} 

\maketitle

\section{Introduction}

We consider compact Lagrangian correspondences ${L_{01} \subset M_0^- \times M_1}$ and $L_{12} \subset M_1^- \times M_2$, where $M_\ell=(M_\ell,\om_{M_\ell})$ are symplectic manifolds that are either compact or satisfy appropriate boundedness conditions (see Remark~\ref{rmk:noncompact}), and where \(M_\ell^- \coloneqq (M_\ell, -\om_{M_\ell})\).
The {\bf geometric composition} of such Lagrangian correspondences is 
$L_{01} \circ L_{12} \coloneqq \pi_{02}( L_{01} \times_{M_1} L_{12}) $, the image under the projection $ \pi_{02}\colon M_0^- \times
M_1 \times M_1^- \times M_2 \to M_0^- \times M_2 $ of the fiber product
\begin{align*}
L_{01} \times_{M_1} L_{12} \coloneqq (L_{01} \times L_{12})
 \cap (M_0^- \times \Delta_1 \times M_2) .
\end{align*}
Here \(\Delta_1 \subset M_1 \times M_1^-\) denotes the diagonal.
If \(L_{01}\times L_{12}$ intersects $M_0^- \times \Delta_1 \times M_2\) transversely then \( \pi_{02}\colon L_{01} \times_{M_1} L_{12} \to M_0^-\times M_2\) is a Lagrangian immersion, in which case we call \(L_{01}\circ L_{12}\) an {\bf immersed composition}.
In the case of {\bf embedded composition}, where the projection is injective and hence a Lagrangian embedding, some strict monotonicity and Maslov index assumptions allowed Wehrheim--Woodward \cite{isom} to establish an isomorphism of quilted Floer cohomologies
\begin{equation} \label{eq:HFiso} 
HF(\ldots, L_{01},L_{12}, \ldots ) \cong HF(\ldots , L_{01} \circ L_{12}, \ldots) .
\end{equation}
The analytic core of the proof was a {\bf{strip-shrinking degeneration}}, in which a triple of pseudoholomorphic strips coupled by Lagrangian seam conditions degenerates to a pair of strips, via the width of the middle strip shrinking to zero\footnote{
Earlier (ca.\ 2004), Matthias Schwarz had proposed a continuation map approach to \eqref{eq:HFiso}, via jumping boundary conditions \cite{ab:jumping}.
This strategy was recast as a Morse--Bott setup by \cite{ll}.}.
The monotonicity and embeddedness assumptions allowed for an implicit exclusion of all bubbling, which Wehrheim--Woodward \cite{isom} conjectured to include a novel {\bf{figure eight bubbling}} that (unlike disk or sphere bubbling) could be an algebraic obstruction to \eqref{eq:HFiso}.

The aim of this paper is to achieve a geometric understanding of figure eight bubbling toward encoding its effect algebraically.
In particular, in \S\ref{ss:algebra1} we propose a direct generalization of \eqref{eq:HFiso} to the nonmonotone case as an isomorphism of quilted Floer homologies with twisted differentials,\footnote{
Twisted differentials are obtained by adding to the Floer differential (that may not square to zero) further contributions arising from an $A_\infty$-structure applied to repetitions of a fixed chain or cochain. 
(Throughout this paper, we will not distinguish between chains and cochains as this only affects the grading, which we do not discuss.)
}
\begin{align*}
HF\bigl( \ldots , (L_{01},b_{01}), (L_{12},b_{12}), \ldots \bigr)
\; \simeq\; 
HF\bigl( \ldots , (L_{01}\circ L_{12}, 8(b_{01},b_{12}) ), \ldots \bigr),
\end{align*}
in which the bounding cochain $8(b_{01},b_{12})$ for the composed Lagrangian is obtained from moduli spaces of figure eight bubble trees with inputs $b_{01}$ and $b_{12}$.
Here even $8(0,0)$ is a generally nonzero count of figure eight bubbles.
Our proposal extends to geometric compositions $L_{01} \circ L_{12}$ that may only be 
\textbf{cleanly immersed}, i.e.\ immersed in such a way that the local branches of \(L_{01} \circ L_{12}\) intersect cleanly, and leads us to formulate Conjecture~\ref{conj:bifunctor}:
Geometric composition defines an {\bf \(\mathbf{A_\infty}\)-bifunctor} \(\comp\colon (\Fuk(M_1^- \times M_2), \Fuk(M_0^- \times M_1)) \to \Fuk(M_0^- \times M_2)\) between Fukaya categories of immersed Lagrangians with clean self-intersections\footnote{
See Remark~\ref{rem:immfuk} for a brief discussion of immersed Fukaya categories.
For now, note that while others -- chiefly \cite{a:immersed, aj} -- have proposed setups for Floer theory of immersed Lagrangians, the notion of the figure eight bubble giving rise to an \(A_\infty\)-bifunctor is entirely new and also yields a new notion of Floer theory for immersed compositions.},
by adding input marked points to the seams of a figure eight bubble and viewing the quilt singularity as an output marked point.
In the special case $M_0=\pt$, such a bifunctor in particular induces an \(A_\infty\)-functor
\(\Fuk(M_1^- \times M_2)  \to A_\infty{\rm Fun}\bigl(\Fuk(M_1) , \Fuk(M_2)\bigr)\) that would provide a more general version of the functor developed in \cite{mww} between extended Fukaya categories of monotone Lagrangians.
More generally, our proposal naturally extends to a symplectic $A_\infty$ 2-category, as we describe in Remark~\ref{rmk:2cat}.

\begin{figure}
\centering
\def\svgwidth{\columnwidth}
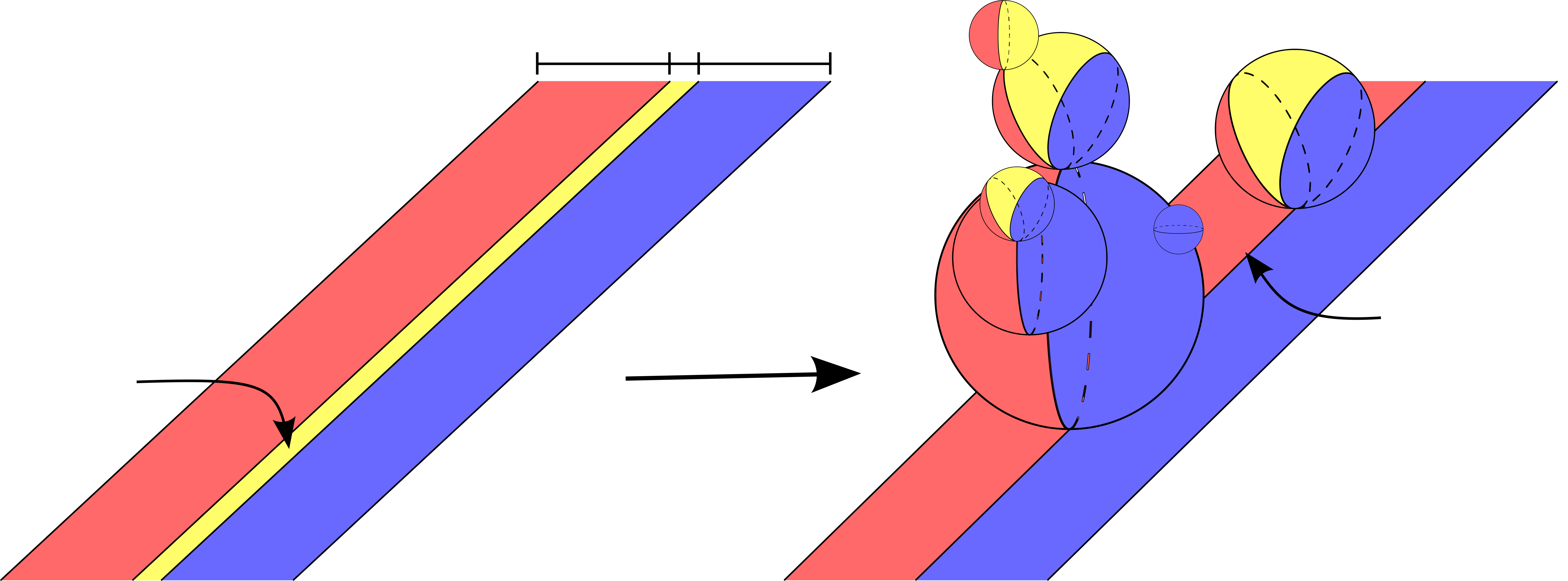
\caption{
A prototypical configuration of bubbles arising in the limit of strip shrinking.
The colors red/yellow/blue of patches indicate the target manifolds $M_0$/$M_1$/$M_2$ of the corresponding pseudoholomorphic maps.
More generally, we allow the boundary seams of the yellow strip to be squiggly as specified in \S\ref{sec:rescale}.
\label{fig:tree}}
\end{figure}

Toward these proposals we study families of pseudoholomorphic quilts in which the width of a strip or annulus shrinks to zero, where it is replaced by the composition of the two correspondences $L_{01},L_{12}$ associated to the adjacent seams (see Figure~\ref{fig:tree}).
The main results of this paper, described in more detail in \S\ref{ss:overview}, are a Gromov compactness theorem for this strip-shrinking, a precise description of the singular quilt bubbling phenomena -- figure eights and squashed eights -- and a lower bound on their bubbling energy.
The proofs rely on a collection of width-independent elliptic estimates established in \cite{b:singularity}.  In turn, our Gromov compactness theorem is a key ingredient for \cite[Thm.~2.2]{b:singularity}, a removal of singularity for figure eight bubbles.
In the latter application to a singular quilt, the shrinking strips arising from reparametrization to cylindrical coordinates near the singularity have nonstandard complex structure:
The quilted surface is locally biholomorphic to a part of the complex plane with nonstraight seams \(\Im z = \pm \arcsin(\tfrac 1 2\Re z)\).
Since essentially only \(\cC^0\)-convergence and \(\cC^k\)-bounds on the width functions \(f^\nu(s) \coloneqq \arcsin(\tfrac 1 2s)|_{[-\nu,-\nu+1]}\) for \(\nu \in \N\) affect the analysis, we formulate our results for general ``squiggly strip shrinking''.

\medskip
\noindent
{\bf Philosophical Remark:}
{\it 
In the early days of pseudoholomorphic quilts, the relevance of figure eight bubbles was doubted.
We hope that this paper puts those doubts to rest.
As we summarize in \S\ref{ss:algconsequences} and explain in \S\ref{s:propaganda}, figure eight bubbling cannot be \emph{a priori} excluded for dimension reasons, and will contribute to the algebra.
So even if e.g.\ the isomorphism \eqref{eq:HFiso} of Floer homology under geometric composition was to be proven with methods other than strip shrinking, one must in general expect figure eight type obstructions.
When all symplectic manifolds and Lagrangians are monotone and all geometric compositions are embedded, then for energy reasons strip shrinking will not lead to figure eight bubbling; however, in only slightly more general situations, this exclusion argument may fail: as is explained in \S\ref{ss:HF},
\cite{w:chekanov}
gives an example in which all symplectic manifolds and Lagrangians are monotone and geometric composition is multiply covered, and figure eight bubbles contribute nontrivially to the algebra resulting from strip shrinking.
However, figure eight bubbles should be viewed as a tool rather than an inconvenience: For instance, they provide a bounding cochain for \(L_{01} \circ L_{12}\) so that \eqref{eq:HFiso} continues to hold in more general situations than \cite{isom} considered.
More generally, fully embracing figure eight bubbles will yield a natural 2-categorical structure on the collection of all compact symplectic manifolds, which will unify and extend a wide variety of currently-known algebraic structures.

While this paper only provides substantial evidence for these algebraic results (or ``proofs up to technical details" depending on ones standards of rigour), it does demonstrate in full detail that figure eight bubbling is in fact analytically manageable:
\S\ref{sec:bubbles} introduces the novel bubble types -- figure eight bubbles and squashed eight bubbles -- and establishes lower bounds as well as topological controls on their energy.
A full removal of singularities for the new bubbles is established in \cite{b:singularity}.
Section~\ref{sec:rescale} gives rigorous definitions of ``squiggly strip shrinking'' and shows how the full diversity of bubble types appears in the Gromov compactification for squiggly strip shrinking. 
Moreover, Appendix~\ref{app:ex} in collaboration with Felix Schm\"{a}schke provides the first nontrivial example of a general figure eight bubble.
}

\subsection{Analytic Results} \label{ss:overview}

\noindent 
Given a sequence of pseudoholomorphic spheres of bounded energy, ``Gromov compactness'' may be approached in two ways.
{\bf Hard rescaling} refers to the process of finding finitely many points of energy concentration, and rescaling at each of these points proportionally to the rate of gradient blowup.
The result is a nontrivial bubble at each energy concentration point.
{\bf Soft rescaling} is a more refined version of this process, where the rescalings are chosen in such a way that the limiting object is a tree of pseudoholomorphic spheres which captures all energy of the limiting subsequence.
The ingredients in the soft rescaling version of Gromov compactness are: \begin{itemize}
\item[(1)] The hard rescaling version of Gromov compactness (\cite[\S4.2]{ms:jh});

\item[(2)] A removal of singularity result (\cite[\S4]{ms:jh});

\item[(3)] An exponential decay result for long cylinders of small energy (\cite[\S4.7]{ms:jh});

\item[(4)] A significant amount of bookkeeping of trees (\cite[\S5]{ms:jh}).
\end{itemize}
The contribution of \S\ref{sec:rescale} (summarized below) of this paper is the analogue of (1) for a sequence of squiggly strip quilts with width function converging obediently to zero.
The analogue of (2) is proven in \cite{b:singularity}, and the techniques developed in that paper can also be used to establish the analogue of (3).
While the analytic pieces are now in place to prove full Gromov compactness in our setting, we defer the analogue of (4) to \cite{bw:bigkahuna} because the proof will be significantly more efficient in the polyfold setting of that paper.

The compactness analysis in \S\ref{sec:rescale} will for ease of notation be performed in the special case of quilted squares with seam conditions in \(L_{01}, L_{12}\) and the width of the strip mapping to $M_1$ converging to zero.
However, it generalizes directly to the following result for strip or annulus shrinking in pseudoholomorphic quilts.
(For an introduction to quilts see \cite{quilts}.)
\medskip

\noindent{\bf Theorem~\ref{thm:rescale} (Gromov compactness for strip-shrinking, hard rescaling version):}
Let $\ul{Q}^\nu$ be a sequence of quilted surfaces containing a patch \(Q_1^\nu\) diffeomorphic to an annulus or strip, equipped with complex structures in which the width of \(Q_1^\nu\) tends to zero as \(\nu \to \infty\).
(For allowable squiggliness -- i.e.\ variation of width -- see Definition~\ref{def:obedience}.)
Label the patches of $\ul{Q}^\nu$ with a tuple $\ul{M}$ of closed symplectic manifolds (or noncompact ones without boundary which come with a priori \(\cC^0\)-bounds as discussed in Remark~\ref{rmk:noncompact}), let $M_1$ and $M_0,M_2$ be the labels of $Q_1^\nu$ and the adjacent patches, and fix compatible almost complex structures over each patch. 
Fix compact Lagrangian seam conditions for each seam of $\ul{Q}^\nu$ so that the Lagrangian correspondences $L_{01}, L_{12}$ associated to the seams bordering $Q_1^\nu$ have immersed composition \(L_{01} \circ L_{12}\).
Now suppose that $(\ul{v}^\nu)_{\nu\in\N}\colon \ul{Q}^\nu \to \ul{M}$ is a sequence of pseudoholomorphic quilts of bounded energy with the given seam conditions.

Then there is a subsequence (still denoted $(\ul{v}^\nu)_{\nu\in\N}$) that converges up to bubbling to a punctured pseudoholomorphic quilt $\ul{v}^\infty\colon \ul{Q}^\infty  \less Z \to (\ul{M}\less M_1)$.
Here $\ul{Q}^\infty$ is the quilted surface obtained as the limit of $\bigl(\ul{Q}^\nu\bigr)$
by replacing $Q_1^\nu$ with a seam labeled by \(L_{01} \circ L_{12}\), $Z$ is a finite set of bubbling points, $\ul{v}^\infty$ satisfies seam conditions in the fixed Lagrangian correspondences and for the new seam in \(L_{01} \circ L_{12}\) (in the generalized sense of \eqref{gen bc}), and convergence holds in the following sense:

\begin{itemlist}
\item
The energy densities $|\d\ul{v}^{\nu}|^2$ are uniformly bounded on every compact subset of $\ul{Q}^\infty \less Z$, and at each point in $Z$ there is energy concentration of at least $\hbar>0$, given by the minimal bubbling energy from Definition~\ref{def:bubble energy}.
\item
The quilt maps $\ul{v}^{\nu}|_{\ul{Q}^\nu \less (Q^\nu_1\cup Z)}$ on the complement of $Z$ in the patches other than $Q^\nu_1$ converge with all derivatives on every compact set to $\ul{v}^\infty$.
\item
At least one type of bubble forms at each point $z\in Z$ in the following sense: There is a sequence of (tuples of) maps obtained by rescaling the maps defined on the various patches near $z$, which converges $\cC^\infty_\loc$ to one of the following: 
\begin{itemize}  
\item[--]
a nonconstant, finite-energy pseudoholomorphic map $\R^2 \to M_\ell$ to one of the symplectic manifolds in $\ul{M}$ (this can be completed to a nonconstant pseudoholomorphic sphere in $M_\ell$);
\item[--] 
a nonconstant, finite-energy pseudoholomorphic map $\H \to M_k^-\times M_\ell$ to a product of symplectic manifolds associated to the patches on either side of a seam in $\ul{Q}^\nu$, that satisfies the corresponding Lagrangian seam condition (this can be extended to a nonconstant pseudoholomorphic disk in $M_k^-\times M_\ell$, in particular including the cases of disks with boundary on $L_{01}\subset M_0^-\times M_1$ or $L_{12}\subset M_1^-\times M_2$);
\item[--]
a nonconstant, finite-energy figure eight bubble in the sense of \eqref{eq:8} below;
\item[--]
a nonconstant, finite-energy squashed eight bubble in the sense of \eqref{gen bc} below, 
with generalized seam conditions in $L_{01}\circ L_{12}$.
\end{itemize}
\end{itemlist}

\medskip
\noindent
{\bf Singular quilt bubbling phenomena:}
Beyond the standard bubbling phenomena (holomorphic spheres and disks) our Gromov compactification for quilts with strip-shrinking involves two new types of bubbles:
A {\bf figure eight bubble} is a tuple of finite-energy pseudoholomorphic maps
\begin{equation} \label{eq:8}
w_0\colon \R\times(-\infty,-\tfrac 1 2]\to M_0, \qquad
w_1\colon \R\times[-\tfrac 1 2,\tfrac 1 2]\to M_1, \qquad
w_2\colon \R\times[\tfrac 1 2, \infty)\to M_2
\end{equation}
satisfying the seam conditions
\begin{equation*}
(w_0(s,-\tfrac 1 2),w_1(s,-\tfrac 1 2))\in L_{01} , \quad
(w_1(s,\tfrac 1 2),w_2(s,\tfrac 1 2))\in L_{12} 
\qquad\forall \: s\in \R ,
\end{equation*}
while a {\bf squashed eight bubble} is a triple of finite-energy pseudoholomorphic maps \begin{align*}
w_0\colon \R\times(-\infty,0]\to M_0, \qquad w_1\colon \R \to M_1, \qquad w_2\colon \R\times[0, \infty)\to M_2
\end{align*}
satisfying the seam condition
\begin{equation}\label{gen bc}
(w_0(s,0), w_1(s), w_1(s), w_2(s,0)) \in L_{01} \times_{M_1} L_{12} \quad \forall \: s \in \R.
\end{equation}

\begin{remark}
Both of these bubbles are of singular quilt type in the sense that we cannot generally expect a smooth extension to a quilted sphere.
For the squashed eight bubble this results from the boundary condition $L_{01} \circ L_{12}$ generally just being a Lagrangian immersion.
For the figure eight bubble this is due to the way in which the two seams intersect at infinity: After stereographic compactification to a quilted sphere, they touch tangentially rather than intersect transversely, which would allow a description of the singularity in terms of striplike ends. 
Nevertheless, if the composition $L_{01}\circ L_{12}$ is cleanly immersed, then the removable singularity result in \cite{b:singularity} shows that $w_0(s,t)\to p_0\in M_0$ and $w_2(s,t)\to p_2\in M_2$ have uniform limits as $s^2+t^2\to\infty$, whereas $w_1(s,t)\to p_1^{\pm}\in M_1$ has two possibly different limits as $s\to \pm \infty$, both of which are lifts of $(p_0,p_2)\in L_{01}\circ L_{12}$, that is $(p_0,p_1^{\pm},p_1^{\pm},p_2)\in L_{01}\times_{M_1} L_{12}$. 
\end{remark}

Recall that even smooth removal of singularity generally does not provide lower bounds on the energy of bubbles except in situations with simple topology; see Remark~\ref{rmk:htop}.
In \S\ref{sec:bubbles} we establish this lower energy bound for figure eight and squashed eight bubbles by purely analytic means.
\medskip

\noindent{\bf Lower Energy Bound Lemma~\ref{lem:hbar}:}
For fixed almost complex structures and Lagrangians with immersed composition \(L_{01} \circ L_{12}\), the energy of nontrivial figure eight and squashed eight bubbles is bounded below by a positive quantity.

\begin{remark}[{\bf Gromov compactification of strip-shrinking moduli spaces -- soft rescaling version}]
\label{rmk:8bubbletree}
Note that the partial Gromov compactness statement in Theorem~\ref{thm:rescale}
only requires the composition $L_{01}\circ L_{12}$ to be immersed.
If the self-intersections of this immersion are locally clean, then the results of \cite{b:singularity} allow us to remove the singularities in the limits of the main component as well as the figure eight and squashed eight bubbles.
Moreover, the techniques of \cite{b:singularity} also provide ``bubbles connect'' results for long cylinders, so that a ``soft rescaling iteration'' (guided by capturing all energy) will yield the following full Gromov compactification of a moduli space with squiggly strip- (or annulus-) shrinking in terms of bubble trees: 
On the complement of the new seam, these trees are made up of trees of disk bubbles\footnote{\label{foot:fold}
Here we identify spheres with a single circle as seam and two patches labeled by \(M_k\) and \(M_\ell\) with disk bubbles in \(M_k^- \times M_\ell\) by ``folding'' across the seam as in the portion of the proof of Theorem~\ref{thm:rescale} treating the (D01) case.}
attached to the seams, with additional trees of sphere bubbles attached to the disks, seams, or interior of the patches.
On the new seam, as indicated in Figure~\ref{fig:tree}, starting from the root, every bubble tree starts with a (possibly empty or containing constant vertices) tree of squashed eight bubbles.
Attached to this are figure eight bubbles (possibly constant) in such a way that between any leaf and the root of the complete tree there is at most one figure eight.
Trees of disk bubbles with boundary on $L_{01}$ resp.\ $L_{12}$ can then be attached to the corresponding seams of the figure eight bubbles.
Finally, trees of sphere bubbles can be attached to the interior of each patch or the seams in this bubble tree.
The hierarchy in this compactification is illustrated in Figure~\ref{fig:prediction} (including the additional complication of Morse flow lines).
We will not provide a detailed construction of this compactification in the present paper, whose point is to establish the foundational analysis.
The full compactification (including Morse flow lines) will be part of the polyfold setup in \cite{bw:bigkahuna}.
\end{remark}

\medskip

Finally, Appendix~\ref{app:ex} in collaboration with Felix Schm\"{a}schke explains how pseudoholomorphic disks and strips can be viewed as special cases of figure eight bubbles, and we provide an example of a nontrivial figure eight bubble with embedded composition $L_{01}\circ L_{12}$ and target spaces $M_0=M_2=\CP^3$, $M_1=\CP^1 \times \CP^1$.

\subsection{Algebraic consequences of figure eight bubbling} \label{ss:algconsequences}
We establish our analytic results in settings that will allow us to describe compactified moduli spaces of pseudoholomorphic quilts with shrinking strips as zero sets of Fredholm sections in polyfold bundles.
This will put the universal regularization theory of \cite{hwz:fred2} at our disposal.
In particular, there will be no need to prove a separate gluing theorem for exhibiting configurations with figure eight bubbles as boundaries of the compactified moduli space: 
Pre-gluing constructions outlined in \S\ref{ss:boundary} will provide polyfold charts with boundary for bubble trees of (not necessarily pseudoholomorphic) quilted maps, and the proof of the nonlinear Fredholm property of the quilted Cauchy--Riemann operator in this polyfold setup will be essentially a version of the quadratic estimates in the classical gluing analysis (which in our case should follow from combining the results of \cite{isom} and \cite{b:singularity}).
Moreover, the boundary stratification of the ambient polyfold will directly induce the boundary stratification of the (regularized) moduli spaces.
This offers a {\bf method for predicting algebraic consequences:} If the considered moduli spaces can be cut out from ambient polyfolds, then the algebraic identities are given by summing over the top boundary strata of the polyfold.
Furthermore, if the local charts for the polyfold arise from pre-gluing constructions (as has been the case in all known examples), then the boundary stratification can be read off from the gluing parameters. 
We thus analyze in \S\ref{ss:HF} the boundary strata predicted by our Gromov-compactification in the case of strip-shrinking used to prove \eqref{eq:HFiso}.
Based on that, \S\ref{ss:Morse} gives a fair amount of detail on the extension of these moduli spaces by Morse trajectories.
This is desirable for easy polyfold implementation as well as reducing algebraic headaches by working with finitely generated chain complexes.
Finally, we use this analysis of boundary strata to predict in \S\ref{ss:algebra1} a generalization of the isomorphism \eqref{eq:HFiso} of Floer homology under geometric composition.

Besides a Fredholm description of figure eight moduli spaces and generalization of \eqref{eq:HFiso}, another motivation for developing the analysis described in \S\ref{ss:overview} is to obtain a new approach to the construction of the $A_\infty$-functors associated to monotone Lagrangian correspondences in \cite{mww}.
Whereas the latter requires a technically cumbersome construction of quilted surfaces with striplike ends and regular Hamiltonian perturbations, and is heavily restricted by monotonicity requirements, a polyfold setup for figure eight moduli spaces will provide a direct construction of curved\footnote{
Disk bubbling gives rise to $\mu_0$ terms in all $A_\infty$-relations, so that in particular squares of differentials can be nonzero.
This can be thought of as allowing curvature, hence we follow e.g.\ \cite{a:intro} and denote this generalized type of $A_\infty$-relations by the prefix ``curved''.
}
 $A_\infty$-functors for general Lagrangian correspondences.
 We explain in \S\ref{ss:algebra2} that we will be able to work directly with the singular quilted surfaces that realize the multiplihedra in \cite{mw}, since these are special cases of figure eights with $M_0=\pt$ and marked points on boundary and seams.
Analyzing moreover the boundary stratification of general figure eight moduli spaces, we arrive at the conjecture that geometric composition is encoded in terms of a curved \(A_\infty\)-bifunctor, which in turn specializes to the desired generalization of the $A_\infty$-functors in \cite{mww}.
In fact, these methods can be extended to generalizations of figure eight moduli spaces, leading to a conjectural symplectic $(\infty,2)$-category that we will further investigate in future work.
While the complete polyfold construction of these new algebraic structures in \cite{b:thesis,bw:bigkahuna} will be lengthy since we aim to provide a technically sound and easily portable basis for all future use of quilt moduli spaces, its rough form and algebraic consequences are already so apparent from our current understanding that it seems timely to give this detailed outline.
We do so in order to motivate the development of this theory and enable investigations of its future applications.

\subsection{Acknowledgements} 
The first ideas of studying figure eight bubbling, and preliminary results toward Theorem~\ref{thm:rescale}, were obtained by the second author during her collaboration with Chris Woodward.
We would like to thank Mohammed Abouzaid, Sheel Ganatra, Tim Perutz, Anatoly Preygel, and Zachary Sylvan for helpful conversations about \S\ref{ss:algebra2}, and the referees for constructive feedback on the exposition.
We gratefully acknowledge support by an NSF Graduate Research Fellowship, a Davidson Fellowship, and an NSF Career Grant, and would like to thank the Institute for Advanced Study, Princeton University, and the University of California, Berkeley for their hospitality.

\section{Squiggly strip quilts and figure eight bubbles}
\label{sec:bubbles}

The purpose of this section is to introduce the new bubbling phenomena with their basic properties.
Besides sphere and disk bubbling, two novel sorts of bubbles may appear: figure eight bubbles and squashed eight bubbles, both of which are introduced in Definition~\ref{def:8}.
In Lemma~\ref{lem:hbar}, we show that the energy of the figure eight and squashed eight bubbles is bounded below, which will be a key ingredient in our proof of the Gromov Compactness Theorem~\ref{thm:rescale}.
The proof of Lemma~\ref{lem:hbar}
-- which we defer to Appendix~\ref{sec:squiggly} --
relies on a \(\cC^\infty\)-compactness statement for squiggly strip shrinking from \cite{b:singularity},
quoted as Theorem~\ref{thm:nonfoldedstripshrink}.

\medskip

In this section and the next we will be working with symplectic manifolds and with pseudoholomorphic curves with seam conditions defined by compact Lagrangian correspondences
\begin{equation}\label{eq:lag}
L_{01} \subset M_0^- \times M_1,  \qquad 
L_{12} \subset M_1^- \times M_2.
\end{equation}
When the following intersection in \(M_0^- \times M_1 \times M_1^- \times M_2\) is transverse, we follow \cite{quiltfloer} and say that \(L_{01}\) and \(L_{12}\) have {\bf immersed composition}:
\begin{equation} \label{eq:comptrans}
(L_{01} \times L_{12}) \pitchfork (M_0^- \times \Delta_1 \times M_2) \; \eqqcolon \, 
L_{01}\times_{M_1} L_{12}.
\end{equation}
Indeed, the transversality implies that \(L_{01} \times_{M_1} L_{12} \subset M_0^- \times M_1 \times M_1^- \times M_2\) is a compact submanifold, and the projection \(\pi_{02}\colon L_{01} \times_{M_1} L_{12} \to M_0^- \times M_2\) is a Lagrangian immersion by e.g.\ \cite[Lemma~2.0.5]{quiltfloer} (which builds on \cite[\S4.1]{gu:rev}).
We will denote its image, the {\bf geometric composition} of $L_{01}$ and $L_{12}$, by
\begin{align*}
L_{01} \circ L_{12} \,\coloneqq \; \pi_{02}( L_{01} \times_{M_1} L_{12} )  \;\subset\;  M_0^- \times M_2 .
\end{align*}
In some contexts we will assume \textbf{cleanly-immersed composition}, that is an immersed composition such that any two local branches of the immersed Lagrangian \(L_{01} \circ L_{12}\) intersect cleanly.

\smallskip

\begin{center}
\fbox{\parbox{0.9\columnwidth}{
Throughout
\S\ref{sec:bubbles}, \S\ref{sec:rescale}, and Appendix~\ref{sec:squiggly}
we will work with fixed symplectic manifolds \(M_0, M_1, M_2\) without boundary which are either compact or satisfy boundedness assumptions as detailed in Remark~\ref{rmk:noncompact}, and compact Lagrangians \(L_{01}, L_{12}\) as in \eqref{eq:lag} with immersed composition.}}
\end{center}

\medskip

\noindent We will consider pseudoholomorphic quilts with respect to compatible almost complex
structures:
\begin{equation}\label{eq:J}
J_\ell\colon [-\rho,\rho]^2 \to\J(M_\ell,\omega_\ell)  \qquad\text{for}\; \ell = 0,1,2 .
\end{equation}
These are allowed to be domain-dependent\footnote{
Note in particular that we do not require \(J_\ell\) to be constant near the seam as in \cite{isom}.
This is so that our results are maximally applicable: For instance, the corrected proof of transversality in \cite{striptrans} does not guarantee regular almost complex structures that are constant near the seam.}
but are \(\cC^k\) as maps \([-\rho,\rho]^2\times \rT M_\ell\to \rT M_\ell\), where \(k\) will be either a positive integer or infinity.
Then compatibility means that
\begin{align} \label{eq:metrics}
g_\ell(s,t)\coloneqq \omega_\ell (- , J_\ell(s,t) - )
\end{align}
are metrics on \(M_\ell\) that are \(\cC^k\) in \((s,t)\in[-\rho,\rho]^2\).

In the Gromov Compactness Theorem~\ref{thm:rescale},
we will see that four types of bubbles may occur at points of energy concentration: the familiar sphere and disk bubbles, and the novel figure eight and squashed eight bubbles.
These novel types of bubbles result from energy concentrating on the limit seam \((-\rho,\rho)\times\{0\}\) in such a way that after rescaling (to achieve uniform gradient bounds), the middle squiggly strip converges to a straight strip of constant width, or zero width in the case of a squashed eight bubble.
Note here that limit maps of this rescaling will be pseudoholomorphic with respect to the almost complex structures at the point of energy concentration.

\begin{definition} \label{def:8}
Fix domain-independent almost complex structures \(J_\ell \in \J(M_\ell, \om_\ell)\) for \(\ell = 0,1,2\).
\vspace{-4mm}
\begin{itemlist}
\item
A {\bf figure eight bubble between \(\mathbf{L_{01}}\) and \(\mathbf{L_{12}}\)} is a triple of smooth maps \begin{align*}
 \ul w = \left( \begin{aligned}
 w_0&\colon\R \times (-\infty, -\tfrac 1 2] \to M_0 \\
 w_1&\colon\R \times [-\tfrac 1 2, \tfrac 1 2] \to M_1 \\
 w_2&\colon\R \times [\tfrac 1 2, \infty) \to M_2
 \end{aligned} \right)
 \end{align*}
that satisfy the Cauchy--Riemann equations \(\partial_s w_\ell + J_\ell(w_\ell)\partial_t w_\ell = 0\) for \(\ell= 0,1,2 \), fulfill the seam conditions 
\begin{gather*}
 (w_0(s, -\tfrac 1 2), w_1(s, -\tfrac 1 2)) \in L_{01}, \quad
(w_1(s, \tfrac 1 2), w_2(s, \tfrac 1 2)) \in L_{12} \qquad \forall \: s\in\R,
\end{gather*}
and have finite energy
\begin{align*}
\tint w_0^* \om_0 + \tint w_1^* \om_1 + \tint w_2^* \om_2 
\;= \tfrac 12 \Bigl( \tint |\d w_0|^2 + \tint |\d w_1|^2 + \tint |\d w_2|^2 \Bigr)
\;<\; \infty .
\end{align*}
\item
A {\bf squashed eight bubble with seam in \(\mathbf{L_{01} \times_{M_1} L_{12}}\)} is a triple of smooth maps 
\begin{align*}
 \ul w = \left( \begin{aligned}
 w_0&\colon\R \times (-\infty, 0] \to M_0 \\
 w_1&\colon\R \to M_1 \\
 w_2&\colon\R \times [0, \infty) \to M_2
 \end{aligned} \right)
 \end{align*}
that satisfy the Cauchy--Riemann equations 
\(\partial_s w_\ell + J_\ell(w_\ell)\partial_t w_\ell = 0\) for
\(\ell\in\{0,2\}\),
fulfill the generalized seam condition
 \begin{equation*}
(w_0(s,0), w_1(s), w_1(s), w_2(s,0))\in L_{01} \times_{M_1} L_{12}
\qquad\forall \: s\in \R ,
\end{equation*}
and have finite energy 
\begin{align*}
\tint w_0^* \om_0 + \tint w_2^* \om_2 
\;= \tfrac 12 \Bigl( \tint |\d w_0|^2 +\tint |\d w_2|^2 \Bigr)
\;<\; \infty .
 \end{align*}
\end{itemlist}
\end{definition}

The name ``figure eight" for the first type of pseudoholomorphic quilt comes from an equivalent description via stereographic projection (as explained in the following remark), while the name ``squashed eight'' indicates that the second type of quilt can occur as limits of figure eights whose entire energy concentrates at infinity, corresponding to shrinking the middle strip.
Alternatively, squashed eights can be viewed as punctured disk bubbles $D\less\{1\}\to M_0^-\times M_2$ with boundary mapping to the immersed Lagrangian $L_{01}\circ L_{12}$ in such a way that it has a smooth lift to \(L_{01} \times_{M_1} L_{12}\).
As explained below, the singularity cannot necessarily be removed.

\begin{remark}\rm  \label{rmk:stereographic}
\begin{itemlist}
\item 
Recall that a pseudoholomorphic map $\R^2 \to M$ gives rise to a punctured pseudoholomorphic sphere $w:S^2\less\{(0,0,1)\} \to M$ via stereographic projection $S^2\less\{(0,0,1)\}\to\R^2$, where we identify 
$S^2=\{(x,y,z)\in\R^3 \,|\, x^2+y^2+z^2 = 1\}$ with the unit sphere in \(\R^3\).
If the energy $\int w^*\omega_M$ is finite, then $w$ extends smoothly to the puncture $(0,0,1)$ by the standard removal of singularity theorem.

\item  One can view a pseudoholomorphic disk in $M_0^- \times M_1$
with boundary on $L_{01}$ as a quilt on $S^2$ arising from a quilt on $\R^2$, given by a $J_0$-holomorphic patch $w_0: \R\times(-\infty,0]\to M_0$ and a $J_1$-holomorphic patch $w_1: \R\times[0,\infty)\to M_1$ satisfying the seam conditions $(w_0(s,0), w_1(s,0)\bigr) \in L_{01}$, as follows:
Stereographic projection lifts these to pseudoholomorphic maps 
$w_0: S^2\less\{(0,0,1)\} \cap \{ y \leq 0 \}\to M_0$ and $w_1: S^2\less\{(0,0,1)\} \cap \{ y \geq 0 \} \to M_1$ defined on the two punctured hemispheres, which map the common boundary to $L_{01}$.
The standard removal of singularity can be interpreted to say that $w_0$ and $w_1$ extend smoothly to the puncture $(0,0,1)$, thus forming a pseudoholomorphic quilted sphere with one seam -- the equator $\{y=0\}$.

The two hemispheres are conformal to disks, so that the extended maps $w_0, w_1$ can be combined to a single pseudoholomorphic map from the disk to $M_0^- \times M_1$ w.r.t.\ the almost complex structure $(-J_0)\times J_1$, with boundary values in $L_{01}$.

\item
A squashed eight bubble gives rise to a quilt on \(S^2\) as in the previous item, but due to the generalized nature of the seam condition, the removal of singularity is less standard.
Under the hypothesis that \(L_{01}\) and \(L_{12}\) have cleanly-immersed composition, \cite[Appendix A]{b:singularity}  yields continuous extensions of \(w_0\) and \(w_2\) across \(\{(0,0,1)\}\), thus giving rise to a continuous but not necessarily smooth map from the disk to \(M_0^- \times M_2\) with boundary values in \(L_{01}\circ L_{12}\).

\begin{figure}
\centering
\def\svgwidth{0.7\columnwidth}
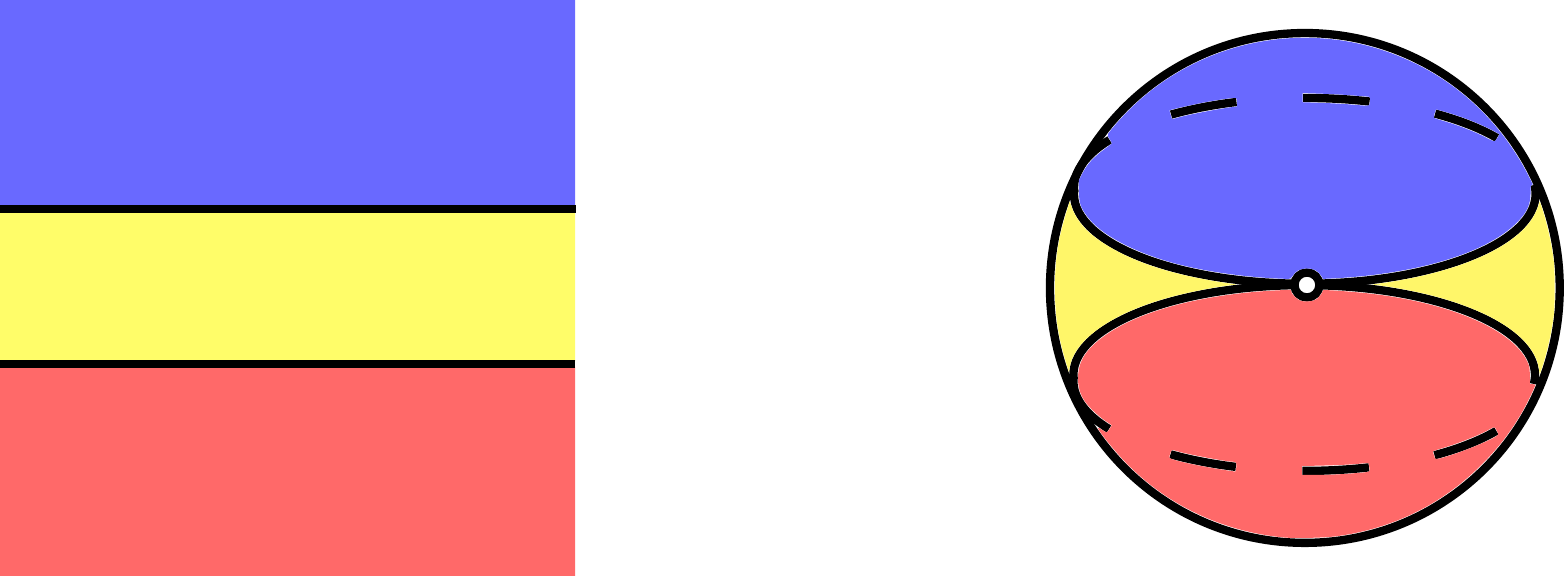
\caption{The left figure illustrates a figure eight bubble, while the right figure illustrates its reparametrization as a pseudoholomorphic quilt whose domain is the punctured sphere.
The domain of the left figure is \(\C\), and the point at infinity corresponds to the puncture in the right figure.
\label{fig:viewsof8}}
\end{figure}

\item
In the case of the figure eight bubble, \((w_0,w_1,w_2)\) is a pseudoholomorphic quilt with total domain \(\R^2\), which maps the seam \(\R\times\{-\half\}\) to \(L_{01}\) and the seam \(\R\times\{\half\}\) to \(L_{12}\).
Pulling these maps back to the sphere by stereographic projection, we obtain a pseudoholomorphic quilt as shown in Figure~\ref{fig:viewsof8} whose domain is the punctured sphere, and which consists of the following patches:
\begin{align*}
w_0\colon& S^2\less\{(0,0,1)\} \cap \{ y \leq -\half (1-z) \}\;\to\; M_0 , \\
w_1\colon&  S^2 \less\{(0,0,1)\} \cap \{ -\half (1-z) \leq y \leq \half (1-z)\} \;\to\; M_1 ,\\
w_2\colon&  S^2\less\{(0,0,1)\} \cap \{ \half (1-z) \leq y  \} \;\to\; M_2 .
\end{align*}
This quilt maps the seam \(\{y = -\half (1-z)\}$ to \(L_{01}\) and the seam \(\{y = \half (1-z)\}\) to \(L_{12}\).
The union of these two seams \(y=\pm\half(1-z)\) on the sphere looks like the figure eight when viewed from the positive \(z\)-axis: two circles that intersect tangentially at \((0,0,1)\).
\cite{b:singularity} establishes a continuous removal of singularity for \((w_0, w_1, w_2)\) at this tangential intersection when \(L_{01}\) and \(L_{12}\) have cleanly-immersed composition.
\end{itemlist}
\end{remark}

We now turn to the definition of, and lower bounds on, the minimal bubbling energy \(\hbar\), which we will need to control the number of bubbling points in the proof of Theorem~\ref{thm:rescale}.

\begin{definition} \label{def:bubble energy}
The {\bf minimal bubbling energy for almost complex structures $\mathbf{J_0,J_1,J_2}$} as in \eqref{eq:J} is the minimum $\hbar \coloneqq \min\{\hbar_{S^2},\hbar_{D^2},\hbar_{L_{01}\circ L_{12}},\hbar_{8} \}$ of the following types of bubble energies.\footnote{\label{foot:bubblesinstrata}
For noncompact manifolds as in Remark~\ref{rmk:noncompact}, spheres, disks, and figure eights touching or contained in the boundary strata of a compactification have to be considered here.}
 
\begin{itemlist}
\item
The \textbf{minimal sphere energy \(\hbar_{S^2}\)} is the minimal energy of a nonconstant\footnote{
If there are no nonconstant pseudoholomorphic spheres, e.g.\ because the symplectic manifolds are exact, then we set $\hbar_{S^2}=\inf\emptyset\coloneqq\infty$; and similarly in the following.}
$J_\ell(s_0,t_0)$-holomorphic sphere in $M_\ell$ for any $\ell = 0,1,2 $ and $(s_0,t_0)\in[-\rho,\rho]^2$.
\item
The \textbf{minimal disk energy \(\hbar_{D^2}\)} is the minimal energy of a nonconstant pseudoholomorphic disk in $(M_0\times M_1, (-J_0(s_0,0) )\times J_1(s_0,0))$ with boundary on $L_{01}$ or in $(M_1\times M_2, (-J_1(s_0,0) )\times J_2(s_0,0))$ with boundary on $L_{12}$ for any $s_0\in[-\rho,\rho]$.
\item
The {\bf minimal figure eight energy \(\hbar_8\)} is the minimal energy of a nonconstant \\
\((J_0(s_0,0), J_1(s_0,0), J_2(s_0,0))\)-holomorphic figure eight bubble between \(L_{01}\) and \(L_{12}\) for any $s_0\in[-\rho,\rho]$.
\item
The {\bf minimal squashed eight bubble energy \(\hbar_{L_{01} \circ L_{12}}\)} is the minimal energy of a nonconstant \((J_0(s_0,0), J_2(s_0,0))\)-holomorphic squashed eight with seam in \(L_{01} \times_{M_1} L_{12}\) for any $s_0\in[-\rho,\rho]$.
\end{itemlist}
\end{definition}

In the remainder of the section, we prove two results related to the minimal figure eight energy.
We begin by establishing positivity \(\hbar_8>0\) in Lemma~\ref{lem:hbar}, which we will need in \S\ref{sec:rescale} to bound the number of bubbles during strip shrinking.
This considerably strengthens the bubbling analysis in \cite{isom}, which merely proves that the number of bubbling points must be finite.

The final result, Proposition~\ref{prop:weak-remsing}, is a weak removal of singularity for any figure eight and squashed eight bubble.
It applies even when the geometric composition \(L_{01}\circ L_{12}\) is not immersed and yields a tuple of smooth maps with compact quilted domain that approximately capture the energy of the bubble and thus can be used in Remark~\ref{rmk:htop} to give a topological understanding of the possible bubble energies.

\begin{lemma} \label{lem:hbar}
Fix \(\rho>0\) and sequences \((J_0^\nu,J_1^\nu,J_2^\nu)_{\nu \in \N_0}\) of \(\cC^3\) almost complex structures on \([-\rho,\rho]^2\) as in \eqref{eq:J}, such that \(J_\ell^\nu\) is locally bounded in \(\cC^3\) and such that the \(\cC^2_\loc\)-limit of \(J_\ell^\nu\) is a \(\cC^\infty\) almost complex structure.
Then \(\inf_\nu \hbar(J_0^\nu,J_1^\nu,J_2^\nu)\) is positive, where \(\hbar\) is the minimum bubbling energy as in Definition~\ref{def:bubble energy}.
\end{lemma}

\noindent We defer the proof to p.~\pageref{proof:hbar} of Appendix~\ref{sec:squiggly}.

\begin{remark} \label{rmk:htop}
The minimal bubbling energies
\(\hbar_{S^2}, \hbar_{D^2}, \hbar_8, \hbar_{L_{01} \circ L_{12}}\) in Definition~\ref{def:bubble energy} can also be bounded below by concrete topological quantities
\begin{align*}
\hbar_{S^2} \geq \hbar_{S^2}^{\rm top}, \qquad \hbar_{D^2} \geq \hbar_{D^2}^{\rm top}, \qquad \min\{\hbar_8, \hbar_{L_{01}\circ L_{12}}\} \geq \hbar_8^{\rm top}
\end{align*}
rather than the abstract analytic lower bound from Lemma~\ref{lem:hbar}.
For sphere and disk bubbles, this topological quantity is the minimal positive symplectic area of spherical or disk (relative) homotopy classes.
Proposition~\ref{prop:weak-remsing} bounds the minimal energy of squashed eight and figure eight bubbles by the minimal positive symplectic area of ``quilted homotopy classes''
\begin{equation*}
\hbar^{\rm top}_8 \coloneqq \inf\left\{ \sum_{\ell \in \{0,1,2\}} \langle [\omega_\ell], [u_\ell] \rangle \,\left|\,
\begin{matrix}
u_0:D^2\to M_0, \\ 
u_1:[0,2\pi]\times [-\frac12, \frac12]\to M_1,\\
u_2:D^2\to M_2 ,
\end{matrix}
\quad\eqref{eq: htop 8 seam},
\eqref{eq: htop 8 w1 ends}
\right.\right\}
\end{equation*}
 for which $u_0, u_1, u_2$ are continuous and satisfy the seam conditions
\begin{equation}\label{eq: htop 8 seam}
(u_0(e^{-i\theta}), u_1(\theta,-\tfrac12))\in L_{01} , \quad
(u_1(\theta,\tfrac12), u_2(e^{i\theta}))\in L_{12} 
\qquad\forall \: \theta\in[0,2\pi]
\end{equation}
and the ``constant limit conditions'' 
\begin{equation}\label{eq: htop 8 w1 ends}
u_1(0, t_1) = u_1(0, t_2), \quad u_1(2\pi, t_1) = u_1(2\pi, t_2) \quad \forall \: t_1, t_2 \in [-\tfrac12, \tfrac12].
\end{equation}
We differentiate such quilt maps by the relation between the two constants $u_1(0,-)$ and $u_1(2\pi,-)$:
\begin{itemize}
\item A {\bf (non-switching) homotopy figure eight} is a tuple of maps $(u_0, u_1, u_2)$ as above with $u_1(2\pi, -)= u_1(0, -)$. 
That is, $u_1: S^1\times [-\frac12, \frac12]\to M_1$ is in fact defined on an annulus with seam conditions that identify $S^1\times \{\pm\frac12\}$ with the boundaries of the two disk patches.
\item
A {\bf sheet-switching homotopy figure eight} is a tuple of maps $(u_0, u_1, u_2)$ as above with \(u_1(0, -) \eqqcolon u_1^- \neq u_1^+ \coloneqq u_1(2\pi, -)\), where \(u_1^-, u_1^+\) represent two different lifts of $\bigl( u_0(1),  u_2(1)\bigr)\in L_{01}\circ L_{12}$ to \(L_{01} \times_{M_1} L_{12}\).
\end{itemize}
Note that sphere homotopy classes as well as disk homotopy classes for \(L_{01}\) and \(L_{12}\) can be represented by non-switching homotopy figure eights with one or two constant patches.\footnote{
Given a sphere or disk bubble, we can attach it to a constant homotopy figure eight under mild hypotheses (e.g.\ that \(L_{01}, L_{12}\) are nonempty and \(M_1\) is connected): Given any two points on \(L_{01}\) and \(L_{12}\), make a zero-energy homotopy figure eight with these values on the seams, \(u_0\) and \(u_2\) constant, and \(u_1\) a path between the two projections.}
However, $\hbar^{\rm top} \coloneqq \min\{\hbar^{\rm top}_{S^2}, \hbar^{\rm top}_{D^2}, \hbar^{\rm top}_8\}$ is not generally positive unless the symplectic and Lagrangian manifolds have very simple topology. 
For example, we have $\hbar^{\rm top}_{S^2} = 0$ as soon as $\lan [\om_\ell] , \pi_2(M_\ell) \ran\subset\R$ contains two incommensurate values for some $\ell\in\{0,1,2\}$.
\end{remark}

The possible homotopy classes of figure eight bubbles in the above remark can be deduced from the removal of singularity theorem in \cite{b:singularity}.
However, this also follows from the following weaker result which requires fewer estimates.
It yields not a pseudoholomorphic quilt on \(S^2\) but a smooth quilt map with domain \(S^2 \cong (D^2)^- \cup (S^1 \times [0, 1]) \cup D^2\) that approximately captures the energy of the bubble.
This result was first announced in \cite{w:chekanov}, but we include it here for convenience.
As in Lemma~\ref{lem:hbar}, we do not assume $L_{01}$ and $L_{12}$ to have immersed composition.

\begin{proposition} \label{prop:weak-remsing}
Let \(L_{01} \subset M_0^- \times M_1\), \(L_{12} \subset M_1^- \times M_2\) be compact Lagrangian correspondences\footnote{
This result does not require the geometric composition \(L_{01}\circ L_{12}\) to be embedded or immersed.
}, and let \((w_0,w_1,w_2)\) be either (1) a figure eight bubble between \(L_{01}\) and \(L_{12}\) or (2) a squashed eight bubble with seam in \(L_{01} \times_{M_1} L_{12}\), where \(w_1\) is the pullback \(w_1(s,t) \coloneqq \ol w_1(s)\).
Then for any \(\eps > 0\) there exist smooth maps \(u_0:D^2\to M_0\), \(\wh u_1:[0,2\pi]\times [-\frac12,\frac12]\to M_1\), \(u_2:D^2\to M_2\) satisfying the seam conditions
\begin{align*}
\bigl(u_0(e^{-i\theta}),\wh u_1(\theta,-\tfrac12)\bigr)\in L_{01} , \quad
\bigl(\wh u_1(\theta,\tfrac12),u_2(e^{i\theta})\bigr)\in L_{12} 
\qquad\forall \: e^{i\theta}\in\partial D^2\cong \R/2\pi\Z ,
\end{align*}
and whose energy is \(\eps\)-close to that of \((w_0,w_1,w_2)\),
\begin{align*}
\Bigl|  \Bigl( \tint u_0^*\om_0 + \tint \wh u_1^*\om_1 + \tint u_2^*\om_2 \Bigr)
- \Bigl( \tint w_0^*\om_0 + \tint w_1^*\om_1 + \tint w_2^*\om_2 \Bigr) \Bigr|
\leq \eps.
\end{align*}
Moreover, \(\wh u_1\) can be chosen to be constant on the two lifts of the line \(\{[0]=[2\pi]\}\times[-\frac12,\frac12]\subset S^1\times [-\frac12,\frac12]\), so that \(\wh u_1|_{\{0\}\times[-\frac12,\frac12]}\equiv p_1^+\), \(\wh u_1|_{\{2\pi\}\times[-\frac12,\frac12]}\equiv p_1^-\) form together with \(p_0\coloneqq u_0(e^{i0})\), \(p_2\coloneqq u_2(e^{i0})\) two lifts \((p_0,p_1^\pm,p_1^\pm,p_2)\in L_{01}\times_{M_1}L_{12}\) of the same point \((p_0,p_2)\in L_{01}\circ L_{12}\).

In particular, if \(\pi_{02} : L_{01} \times_{M_1} L_{12} \to L_{01}\circ L_{12}\) is injective, then \(\wh u_1\) can be chosen such that it induces a smooth map \(u_1: S^1\times[-\frac12,\frac12] \to M_1\). 
\end{proposition}

\noindent We defer the proof to p.~\pageref{proof:weak-remsing} of Appendix~\ref{sec:squiggly}.

\begin{remark}
Under the hypothesis that \(L_{01}, L_{12}\) have immersed composition, Proposition~\ref{prop:weak-remsing} can be modified to show that a squashed eight can be approximated by a homotopy squashed eight, rather than a homotopy figure eight.
In this situation, the minimum squashed eight energy \(\hbar_{L_{01} \circ L_{12}}\) can be bounded below by the minimum positive symplectic area of ``homotopy squashed eights'': \begin{align*}
\hbar_{L_{01} \circ L_{12}}^{\rm top} \coloneqq \inf\bigl\{ \langle [(-\omega_0) \oplus \omega_2] , [u] \rangle >0 \,\big|\, u\in\cC^0(D,M_0^-\times M_2) , u(\partial D) \subset L_{01}\circ L_{12}   \bigr\} .
\end{align*}
\end{remark}

\section{Toward Gromov compactness for strip shrinking}
\label{sec:rescale}

In this section we state and prove the Gromov Compactness Theorem~\ref{thm:rescale}, which is the main result of this paper.
In order to focus on the relevant effects, rather than deal with complicated notation, Theorem~\ref{thm:rescale} is stated in the setting of squiggly strip quilts, with the width of the middle strip shrinking obediently to zero.
However, the results of this section directly generalize to a sequence of pseudoholomorphic quilt maps whose domains are quilted surfaces which vary only by the width of one patch -- diffeomorphic to a strip or annulus -- going to zero.

Theorem~\ref{thm:rescale} is a refinement and generalization of \cite[Theorem 3.3.1 and Lemma 3.3.2]{isom}, where compactness up to energy concentration is proven for strip shrinking in the special case of embedded composition, though only in an $H^2\cap W^{1,4}$-topology and with a lower bound on the energy concentration that has no geometric interpretation but arises by contradiction from mean value inequalities.  
(In fact, the $H^2\cap W^{1,4}$-convergence does not even suffice to deduce nontriviality of the weak limit of rescaled solutions near a bubbling point.)
We establish full $\cC^\infty_{\rm loc}$-convergence in the most general natural case, with the further generalization to noncompact manifolds being discussed in Remark~\ref{rmk:noncompact}.
The proof will moreover illuminate the origin of the different bubbling phenomena.
Analytically, it relies on Theorem~\ref{thm:nonfoldedstripshrink}, a result from \cite{b:singularity}.
A further generalization is that we will allow the two seams bordering the middle strip to not be straight, so that Theorem~\ref{thm:rescale} allows the first author to establish a removal of singularity theorem for figure eight bubbles in \cite{b:singularity}.

The underlying quilted domains of our pseudoholomorphic quilt maps will be open squares with two seams.
One should imagine these as part of the domain of a larger pseudoholomorphic quilt with compact domain or with quilted cylindrical ends.
The basic (localized and rescaled) examples studied in \cite{isom} are squares $(-1,1)^2$ with seams $(-1,1)\times\{\pm\delta\}$, whose main feature is a middle strip $(-1,1)\times[-\delta,\delta]$ of constant width $2\delta>0$.
The following definition generalizes the underlying quilted surfaces to allow middle domains $\{(s,t)\in (-1,1)^2 \,|\, |t|\leq f(s) \}$ of local widths $2f(s)>0$ varying with $s$.
Since diffeomorphically such domains are still strips, we call them ``squiggly strips''.

\begin{definition} \label{def:squigglystrip}
Fix \(\rho>0\), a real-analytic function \(f\colon [-\rho,\rho] \to (0, \rho/2]\), almost complex structures \(J_\ell\), \(\ell = 0,1,2\) as in \eqref{eq:J}, and a complex structure \(j\) on \([-\rho,\rho]^2\).
A {\bf \(\mathbf{(J_0,J_1,J_2,j)}\)-holomorphic size-\(\mathbf{(f,\rho)}\) squiggly strip quilt for \(\mathbf{(L_{01},L_{12})}\)} is a triple of smooth maps
\begin{align} \label{eq:squigglymaps}
 \ul{v}=\left(
\begin{aligned}
v_0\colon \{(s,t) \in (-\rho,\rho)^2 \: | & \; t \leq -f(s) \} \to M_0 \\ 
v_1\colon \{(s,t) \in (-\rho,\rho)^2\: | &  \;  |t | \leq f(s) \} \to M_1 \\ 
v_2\colon \{(s,t) \in (-\rho,\rho)^2\: | & \; t \geq f(s) \} \to M_2 
\end{aligned}
\right) 
\end{align}
that fulfill the seam conditions
\begin{align} \label{eq:squigglyseams}
\bigl(v_0(s,-f(s)),v_1(s,-f(s))\bigr)\in L_{01}, \qquad \bigl(v_1(s,f(s)), v_2(s,f(s))\bigr)\in L_{12} \qquad \forall \: s \in (-\rho,\rho),
\end{align}
satisfy the Cauchy--Riemann equations
\begin{align} \label{eq:squigglyCR}
\d v_\ell(s,t) \circ j(s,t) - J_\ell(s,t,v_\ell(s,t)) \circ \d v_\ell(s,t) = 0 \qquad \forall\: \ell = 0,1,2
\end{align}
for \((s,t)\) in the relevant domains, 
and have finite energy\footnote{
Throughout we will make use of the standard energy identity \cite[Lemma 2.2.1]{ms:jh} for pseudoholomorphic maps.
}
\begin{align*}
E(\ul{v})
\,\coloneqq \; \tint v_0^*\om_0 + \tint v_1^*\om_1 + \tint v_2^*\om_2
\;= \tfrac 12 \Bigl( \tint |\d v_0|^2 +\tint |\d v_1|^2 +\tint |\d v_2|^2 \Bigr)
\;<\; \infty .
\end{align*}
\end{definition}
\noindent When \(j\) is the standard complex structure \(i\), \eqref{eq:squigglyCR} reduces to the equation
\begin{align*}
\partial_s v_\ell(s,t) + J_\ell(s,t, v_\ell(s,t))\partial_t v_\ell(s,t)=0.
\end{align*}

\noindent 
When considering a \((J_0,J_1,J_2)\)-holomorphic squiggly strip quilt \(\ul v^\nu\),
it will be useful to consider the {\bf energy density functions}
\begin{align} \label{eq:density}
\bigl| \d\ul{v} \bigr|:(-\rho,\rho)^2\to[0,\infty), \qquad
\bigl|\d\ul{v}(s,t)\bigr| \coloneqq \left( \bigl|\d v_0(s,t )\bigr|_{J_0}^2 
+ \bigl|\d v_1(s,t)\bigr|_{J_1}^2 + \bigl|\d v_2(s, t)\bigr|_{J_2}^2 \right)^{\frac 12},
\end{align}
where the norms \(|\d v_\ell(s,t)|\) are induced by the metrics defined by \(\om_\ell\) and \(J_\ell\) as in \eqref{eq:metrics} and set to be zero on the complement of the domain of \(v_\ell\).
If \(\ul v\) is a squiggly strip quilt, then \(|\d \ul v|\) is upper semi-continuous at points in
\(\{ (s, \pm f(s)) \}\), continuous elsewhere, and satisfies \(E(\ul{v})= \tfrac 12\int |\d\ul{v} |^2\).

The goal of this paper is to generalize and strengthen the strip shrinking analysis in \cite{isom}, which considers sequences of \((J_0^\nu,J_1^\nu,J_2^\nu)\)-holomorphic squiggly strip quilts of width \(f^\nu\equiv\delta^\nu\to 0\).
For that purpose we consider varying width functions $f^\nu$ that converge to zero in the following way.

\begin{definition} \label{def:obedience}
Fix \(\rho>0\).  A sequence \(\bigl( f^\nu \bigr)_{\nu\in\N}\) of real-analytic functions \(f^\nu\colon [-\rho,\rho]\to (0,\rho/2]\)
\textbf{obediently shrinks to zero}, \(\mathbf{f^\nu \Rightarrow 0}\),
if 
\begin{align}
\label{eq:obedient1}
\max_{s\in[-\rho,\rho]} f^\nu(s) \underset{\nu\to\infty}{\longrightarrow} 0
\end{align}
and
\begin{align}
\label{eq:obedient2}
\sup_{\nu\in\N}\; \frac{\max_{s\in[-\rho,\rho]} \bigl| \tfrac{\rd^k}{\rd s^k}f^\nu(s) \bigr|}{\min_{s\in[-\rho,\rho]} f^\nu(s)} \eqqcolon C_k <\infty  \qquad\forall \: k\in \N_0,
\end{align}
and in addition there are holomorphic extensions \(F^\nu\colon [-\rho,\rho]^2 \to \C\) of \(f^\nu(s) = F^\nu(s,0)\) such that \((F^\nu)\) converges \(\cC^\infty\) to zero\footnote{To see why the last condition is
not redundant, consider the sequence of functions \((F^\nu\colon [-1,1]^2 \to \C)\) defined by
$F^\nu(z) \coloneqq \exp(-2\nu)\sin(4\nu z) + \exp(-\nu)$.
For $x \in \R$, we have the formulas
$F^\nu(x) = \exp(-2\nu)\sin(4\nu x) + \exp(-\nu)$ and $F^\nu(ix) = i\bigl(\exp(2\nu(2x-1))-\exp(-2\nu(2x+1))\bigr)/2 + \exp(-\nu)$,
so the restrictions to  $[-1,1] \times \{0\}$ satisfy \eqref{eq:obedient1}, \eqref{eq:obedient2}
but $F^\nu(3i / 4)$ diverges to $i\infty$.}.
\end{definition}

We will see in Theorem~\ref{thm:rescale} that any sequence \((\ul v^\nu)\) of pseudoholomorphic squiggly strip quilts of bounded energy and obediently shrinking widths \(f^\nu\Rightarrow 0\) has a subsequence that -- up to finitely many points where energy concentrates -- converges to a degenerate strip quilt, in which the middle domain mapping to \(M_1\) is replaced by a single straight seam mapping to the immersed Lagrangian \(L_{01} \circ L_{12}\).
Here bubbling near the middle squiggly strip may lead to limit maps whose seam values switch between the sheets of \(L_{01}\circ L_{12}\).
Thus we need to allow for singularities in the degenerate strip quilts as follows.

\begin{definition}
Fix \(\rho>0\), almost complex structures \(J_\ell\), \(\ell \in \{0,2\}\) as in \eqref{eq:J}, and a complex structure \(j\) on \([-\rho,\rho]^2\).
A {\bf \(\mathbf{(J_0,J_2,j)}\)-holomorphic size-\(\mathbf{\rho}\) degenerate strip quilt for \(\mathbf{(L_{01} \times_{M_1} L_{12})}\) with singularities} is a triple of smooth maps
\begin{align} \label{eq:degenmaps}
\ul{v}=\left(
\begin{aligned}
u_0\colon & \, (-\rho,\rho)\times(-\rho,0] \;\less\; S\times\{0\} \to M_0 \\ 
u_1\colon & \, (-\rho,\rho) \;\less\; S \to M_1 \\
u_2\colon & \, (-\rho,\rho)\times[0,\rho) \;\less\; S\times\{0\} \to M_2 
\end{aligned}
\right)
\end{align}
defined on the complement of a finite set \(S\subset\R\) that fulfill the lifted seam condition
\begin{align} \label{eq:degenseams}
\bigl( u_0(s,0), u_1(s),
u_1(s),
u_2(s,0) \bigr) \in  
L_{01}\times_{M_1}  L_{12}
\qquad\forall \: s\in(-\rho,\rho)\less S,
\end{align}
satisfy the Cauchy--Riemann equation \eqref{eq:squigglyCR} for \(\ell \in \{0,2\}\) and
\((s,t)\) in the relevant domains, and have finite energy 
\begin{align*}
E(\ul{u}) \, \coloneqq \; \tint u_0^*\om_0 + \tint u_2^*\om_2
\;= \tfrac 12 \Bigl( \tint |\d u_0|^2 + \tint |\d u_2|^2 \Bigr)
\;<\; \infty.
\end{align*}
\end{definition}

\begin{remark} \label{rem:singset}
If \(u_1\) in the above definition continuously extends to a point $p \in S$,
then -- by the standard removal of singularity result with embedded Lagrangian boundary conditions --  all \(u_i\) extend smoothly to this point.
Indeed, in this case we may restrict $u_0,u_2$ to a small $B_\eps(p)\setminus\{p\}$, then fold across $(p-\eps,p)\cup(p,p+\eps)$ to obtain a punctured pseudoholomorphic half-disk with boundary condition in a single local branch of $L_{01}\circ L_{12}$.
Hence one can prescribe \(S\) to be the set of discontinuities of \(u_1\).

In fact, the removal of singularity for squashed eights established in \cite[Appendix A]{b:singularity} shows that \(u_0\) and \(u_2\) extend continuously to any point in \(S\) under the hypothesis that \(L_{01}\) and \(L_{12}\) have cleanly-immersed composition.
In this case, the only map with any discontinuities is \(u_1\).
\end{remark}

\begin{theorem} \label{thm:rescale} 
Fix \(\rho > 0\), sequences \(J_0^\nu ,J_1^\nu ,J_2^\nu\) of smooth almost complex structures on \([-\rho,\rho]^2\) as in \eqref{eq:J} that converge \(\cC^\infty_\loc\) to \(J_\ell^\infty\colon [-\rho,\rho]^2 \to \J(M_\ell,\om_\ell)\) for \(\ell = 0,1,2 \), and a sequence \((f^\nu\colon[-\rho,\rho] \to (0,\rho/2])\) of real-analytic functions shrinking obediently to zero as in Definition~\ref{def:obedience}.
Then for any sequence \((\ul{v}^\nu)_{\nu\in\N}\) of \((J_0^\nu,J_1^\nu,J_2^\nu)\)-holomorphic size-\((f^\nu,\rho)\) squiggly strip quilts for \((L_{01},L_{12})\) as in Definition~\ref{def:squigglystrip} with bounded energy \(E \coloneqq \sup_{\nu \in \N} E(\ul v^\nu) < \infty\) there exist finitely many blow-up points \(Z=\{z_1,\ldots,z_{N}\}\subset (-\rho,\rho)^2\) and a subsequence that Gromov-converges in the following sense:
\begin{enumerate}
\item
There exists a \((J_0^\infty, J_2^\infty)\)-holomorphic degenerate strip quilt \(\ul v^\infty\) for \(L_{01} \times_{M_1} L_{12}\) with singularities, whose singular set \(S \subset (-\rho,\rho) \cong (-\rho,\rho)\times\{0\}\) is contained in \(\{z_1, \ldots, z_N\}\cap (-\rho,\rho)\times\{0\}\), such that \((v_0^\nu(s,t - f^\nu(s)))\) resp.\ \((v_1^\nu(s,0))\) resp.\ \((v_2^\nu(s,t + f^\nu(s)))\) converge \(\cC^\infty_\loc\) on the domains \((-\rho,\rho) \times (-\rho,0] \less Z\) resp.\ \((-\rho,\rho) \times \{0\} \less Z\) resp.\ \((-\rho,\rho) \times [0, \rho) \less Z\) to \(v_0^\infty\) resp.\ \(v_1^\infty\) resp.\ \(v_2^\infty\).
\item
There is a concentration of energy $\hbar>0$, given by the minimal bubbling energy from Definition~\ref{def:bubble energy}, at each $z_j$ in the sense that there is a sequence of radii $r^\nu\to 0$ such that
\begin{align*}
\liminf_{\nu\to\infty} \int_{B_{r^\nu}(z_j)} \tfrac 1 2|\d\ul{v}^\nu|^2 \;\geq \; \hbar > 0,
\end{align*}
where the energy densities $|\d\ul{v}^\nu|$ are defined as in \eqref{eq:density}.
\item At least one type of bubble forms near each blow-up point \(z_j=(s_j,t_j)\):
There is a sequence \((w^\nu)\) of (tuples of) maps obtained by rescaling the maps defined on the intersection of the respective domain with \(B_{r^\nu}(z_j)\), which converges in $\cC^\infty_\loc$ to one of the following: 
\begin{enumlist}  
\item[\bf (S0),(S1),(S2):] 
a \(J_\ell^\infty(z_j)\)-holomorphic map \(w_\ell^\infty\colon \R^2 \to M_\ell\) for \(\ell = 0,1,2\)
via the rescaling \eqref{eq:S1_rescale} for $\ell=1$ and analogously for $\ell = 0,2$,
which can be completed to a nonconstant pseudoholomorphic sphere \({\overline{w}_\ell^\infty\colon S^2 \to M_\ell}\);
\item[\bf (D01):] 
a \((-J_0^\infty(s_j,0))\times J_1^\infty(s_j,0)\)-holomorphic map \(w_{01}^\infty\colon \H  \to M_0^-\times M_1\) with \(w_{01}^\infty(\partial\H )\subset L_{01}\) via the rescaling \eqref{eq:D01_rescale},
which can be extended to a nonconstant pseudoholomorphic disk \(\overline{w}_{01}^\infty\colon (D,\partial D) \to (M_0^-\times M_1,L_{01})\);
\item[\bf (D12):] 
a \((-J_1^\infty(s_j,0))\times J_2^\infty(s_j,0)\)-holomorphic map \(w_{12}^\infty\colon \H  \to M_1^-\times M_2\) with \(w_{12}^\infty(\partial\H )\subset L_{12}\) via a rescaling analogous to \eqref{eq:D01_rescale},
which can be extended to a nonconstant pseudoholomorphic disk \(\overline{w}_{12}^\infty\colon (D,\partial D) \to (M_1^-\times M_2,L_{12})\);
\item[\bf (E012):] 
a nonconstant
\((J_0^\infty(s_j,0), J_1^\infty(s_j,0), J_2^\infty(s_j,0))\)-holomorphic figure eight bubble between \(L_{01}\) and \(L_{12}\), as in Definition~\ref{def:8}, via the rescaling \eqref{eq:E012_rescale};
\item[\bf (D02):] 
a nonconstant \((J_0^\infty(s_j,0),J_2^\infty(s_j,0))\)-holomorphic squashed eight bubble with seam in $L_{01} \times_{M_1} L_{12}$, as in Definition~\ref{def:8}, via the rescaling \eqref{eq:D02_rescale}.
\end{enumlist}
\end{enumerate}
\end{theorem}

\begin{remark}
If the composition \(L_{01} \circ L_{12}\) is cleanly immersed, then \cite[Thm.~2.2]{b:singularity} guarantees a continuous removal of singularity for figure eight and squashed eight bubbles, in particular for the bubbles produced in cases (E012) and (D02) of Theorem~\ref{thm:rescale}.
This allows us to partially characterize the singular set $S\subset\R$ in (1) at which $\ol v_1$ does not extend continuously, and hence to which $v_0,v_2$ may not extend smoothly: A necessary condition for a bubbling point $z_j$ to lie in $S$ is that a sheet-switching bubble can be found by rescaling near $z_j$, i.e.\ a squashed eight bubble whose boundary arc on \(L_{01} \circ L_{12}\) does not lift to \(L_{01} \times_{M_1} L_{12}\), or a figure eight bubble with $\lim_{s\to-\infty} w^\infty_1(s,-) \neq \lim_{s\to+\infty} w^\infty_1(s,-)$.
However, this is not a sufficient condition, since a tree involving several sheet-switching bubbles at $z_j$ could allow continuous extension of $\ol v_1$ to $z_j$.
\end{remark}

The proof of Theorem~\ref{thm:rescale} will take up the rest of this section.
Our first goal will be to find a subsequence and blow-up points so that (2) and (3) hold together with the following bound on energy densities:
\begin{enumerate}
\item[(0)]
The energy densities \(|\d\ul{v}^\nu|\) are uniformly bounded away from the bubbling points, that is for each compact subset \(K\subset (-\rho,\rho)^2 \less\{z_1,\ldots z_N\}\) we have
$
\sup_{\nu\in\N} \; \bigl\| \d\ul{v}^{\nu} \bigr\|_{L^\infty( K \cap \, (-\rho,\rho)^2 ) }
< \infty $.
\end{enumerate}
Then we will show that (1) follows from Theorem~\ref{thm:nonfoldedstripshrink}.

Suppose that we have already found a subsequence (for convenience again indexed by \(\nu\in\N\)) and some blow-up points \(z_1,\ldots,z_N\in(-\rho,\rho)^2\) such that (2) holds and we have established (3) at each such point.
Now either (0) holds, too, or we can pass to a further subsequence and find another blow-up point \(z_{N+1}=(s_{N+1},t_{N+1})=\lim_{\nu\to\infty}(s^\nu, t^\nu)\in (-\rho,\rho)^2\less\{z_1,\ldots z_N\}\) such that \(\lim_{\nu \to \infty} |\d \ul{v}^\nu(s^\nu,t^\nu)| = \infty\).
We can apply the Hofer trick \cite[Lemma 4.3.4]{ms:jh}\footnote{Note that the Hofer trick applies directly to each function $f(x)=|\d\ul{v}^\nu(x)|$ for $x=(s,t)$, although it is only upper semi-continuous.
In the proof, continuity is used only to exclude $f(x_n)\to\infty$ for a convergent sequence $x_n\to x_\infty$.
For a bounded upper semi-continuous function $f$, we still have $\limsup f(x_n) \leq f(x_\infty) <\infty$, excluding this divergence.
} 
to vary the points \(z^\nu=(s^\nu,t^\nu)\) slightly (not changing their limit) and find \(\eps^\nu\to 0\) such that we have
\begin{align} \label{eq:vbounds}
\sup_{(s,t)\in B_{\eps^\nu}(s^\nu,t^\nu)} |\d \ul{v}^\nu(s,t)| \leq 2  |\d \ul{v}^\nu(s^\nu,t^\nu)| \eqqcolon 2 R^\nu, \qquad 
R^\nu \eps^\nu \to\infty .
\end{align}
We will essentially rescale by \(R^\nu\) around \((s^\nu,t^\nu)\) to obtain different types of bubbles, depending on where the lines \(\{t= \pm f^\nu(s)\}\) get mapped under the rescaling.
We denote by \({\tau_\pm^\nu \coloneqq R^\nu(\pm f^\nu(s^\nu)-t^\nu)}\) the \(t\)-coordinate of the preimage of the point \((s^\nu, \pm f^\nu(s^\nu))\) under the rescaling
$z \mapsto z^\nu + z/R^\nu$.
After passing to a subsequence, we may assume that \(\tau_\pm^\nu\) converges to \(\tau_\pm^\infty \in \R \cup \{\pm \infty\}\) with \(\tau_-^\infty \leq \tau_+^\infty\).
Then exactly one of the following cases holds:
\begin{enumlist}  
\item[\bf (S0)]
$\tau_-^\infty = \tau_+^\infty= \infty$
\item[\bf (S1)]
$\tau_-^\infty = -\infty$ and $\tau_+^\infty = \infty$
\item[\bf (S2)]
$\tau_-^\infty = \tau_+^\infty = -\infty$
\item[\bf (D01)]
$\tau_-^\infty \in \R$ and \(\tau_+^\infty = \infty\)
\item[\bf (D12)]
$\tau_-^\nu = -\infty$ and $\tau_+^\nu \in \R$
\item[\bf (E012)] $\tau_\pm^\infty \in \R$ and \(\tau_-^\infty < \tau_+^\infty\)
\item[\bf (D02)] $\tau_-^\infty = \tau_+^\infty \in \R$
\end{enumlist}

Below, we will for each case specify the rescaled maps and establish their convergence to one of the bubble types in (3) as well as prove the energy concentration in (2).
Thus in all cases we will have proven (2) and (3) for the new blow-up point $z_{N+1}$, and after adding this point we will either have (0) satisfied or be able to find another blow-up point.
Since $\hbar>0$ by Lemma~\ref{lem:hbar}, we will find at most $E/\hbar$ such blow-up points in this iteration before (0) holds.

Out of the seven blow-up scenarios just listed, the only case where our rescaling argument will be significantly different from standard rescaling arguments is (D02), in which we will need to appeal to the new analysis of Theorem~\ref{thm:nonfoldedstripshrink}.
The rescaling argument in the cases (D01), (D12), (E012) is essentially the same as the standard process of ``bubbling off a disk'', since locally we can fold across the seam to obtain a pseudoholomorphic map to a product manifold.

Before we rescale to obtain bubbles, we record the key properties of the rescaled width function.

\begin{lemma} \label{lem:rescaled-width}
Given a sequence $\bigl(f^\nu\colon [-\rho,\rho]\to(0,\rho/2]\bigr)$
of real-analytic functions shrinking obediently to zero, shifts $s^\nu{\to}s^\infty$, and rescaling factors $\alpha^\nu{\to}\infty$, the rescaled width functions
 \(\wt f^\nu(s) \coloneqq \alpha^\nu f^\nu(s^\nu + s/\alpha^\nu)\)
satisfy \(\cC^\infty_\loc(\R)\) convergence
$$
\wt f^\nu - \wt f^\nu(0) \;\underset{\nu\to\infty}{\longrightarrow} \; 0 , \qquad
\wt f^\nu / \wt f^\nu(0)  \;\underset{\nu\to\infty}{\longrightarrow} \; 1.
$$
Moreover, let \(F^\nu\) be the extension of \(f^\nu\) from Definition~\ref{def:obedience}, identify \((s,t)\in \R^2\) with \(z=s+it \in \C \), and set
\begin{align*}
\phi^\nu(s,t) &\coloneqq \bigl(s^\nu + s/\alpha^\nu, 2f^\nu\bigl(s^\nu + s/\alpha^\nu\bigr)\, t\bigr), \\
\psi^\nu(z) &\coloneqq s^\nu + z/\alpha^\nu - i F^\nu\bigl( s^\nu + z/\alpha^\nu \bigr).
\end{align*}
Then for any \(R>0\) and \(\nu\) sufficiently large, the maps \((\psi^\nu)^{-1} \circ \phi^\nu\) are well defined on \(B_R(0)\).
In the special case \(\alpha^\nu \coloneqq (2f^\nu(s^\nu))^{-1}\), the maps \((\psi^\nu)^{-1} \circ \phi^\nu\) converge \(\cC^\infty_\loc(\R^2,\R^2)\) to \((s,t) \mapsto (s, t+1/2)\).
\end{lemma}

\noindent We defer the proof to p.~\pageref{proof:rescaled-width} of Appendix~\ref{sec:squiggly}.

Continuing with the proof of Theorem~\ref{thm:rescale},
the nontrivial bubbles claimed in (3) are now obtained as follows:
\medskip

\begin{itemlist}
\item[\bf (S1):]
We will obtain a {\bf sphere bubble in \(\mathbf{M_1}\)} by rescaling
\begin{align}
\label{eq:S1_rescale}
w_1^\nu(s,t)\coloneqq v^\nu_1(s^\nu + \tfrac s {R^\nu} ,  t^\nu + \tfrac t {R^\nu}  )
\end{align}
to define maps \begin{align*}
w_1^\nu\colon U_1^\nu \coloneqq \bigl\{ (s,t) \: | \: -R^\nu(\rho + s^\nu) \leq s \leq R^\nu(\rho - s^\nu), \: -\wt f^\nu(s) - R^\nu t^\nu \leq t \leq \wt f^\nu(s) - R^\nu t^\nu \bigr\} \;\longrightarrow\; M_1,
\end{align*}
where $\wt f^\nu$ is the rescaled width function $\wt f^\nu(s) \coloneqq R^\nu f^\nu(s^\nu + s/R^\nu)$.
The map \(w_1^\nu\) is pseudoholomorphic with respect to \(\wt J_1^\nu(t)\coloneqq J_1^\nu(s^\nu + s/R^\nu , t^\nu +  t/R^\nu)\); due to the convergence \(R^\nu \to \infty\), these almost complex structures converge in \(\cC^\infty_\loc(\R^2)\) as \(\nu \to \infty\) to the constant almost complex structure \(\wt J_1^\infty\coloneqq J_1^\infty(s^\infty,t^\infty)\).
By construction, the maps \(w_1^\nu\) satisfy both upper and lower gradient bounds: \begin{gather}
|\d w_1^\nu(0) | = \tfrac1{R^\nu} |\d  v_1^\nu(s^\nu, t^\nu)| = \tfrac 1 {R^\nu}|\d \ul v(s^\nu, t^\nu)| = 1 , \label{eq:w1gradbound} \\ 
 \sup_{(s,t)\in B_{R^\nu\eps^\nu}(0)} |\d w_1^\nu(s,t)| 
\leq \sup_{(s,t)\in B_{\eps^\nu}(s^\nu,t^\nu)} \tfrac1{R^\nu} |\d \ul{v}^\nu(s,t)|  
 \leq 2 , \nonumber
\end{gather}
where the second equality in the top line follows for large \(\nu\) from the assumption \(\tau_\pm^\nu \to \pm\infty\).
The containment \(s^\nu \to s^\infty \in (-\rho,\rho)\) implies that the left resp.\ right bounds \(R^\nu(\mp \rho - s^\nu)\) of \(U_1^\nu\) have limits \(-\infty\) resp.\ \(\infty\); furthermore, the assumption \(\tau_\pm^\infty = \pm\infty\) and Lemma~\ref{lem:rescaled-width} implies that the lower resp.\ upper bounds \(-\wt f^\nu(s) - R^\nu t^\nu\) resp.\ \(\wt f^\nu(s) - R^\nu t^\nu\) of \(U_1^\nu\) converge \(\cC^\infty_\loc\) to \(-\infty\) resp.\ \(\infty\).
Hence the maps \(w_1^\nu\) are defined with uniformly bounded differential on balls centered at \(0\) of radii tending to infinity. 
Standard compactness for pseudoholomorphic maps (e.g.\ \cite[Appendix B]{ms:jh}\footnote{\label{foot:bounds}
If noncompact symplectic manifolds are involved, then one needs to establish \(\cC^0\)-bounds on the maps before ``standard Gromov compactness'' can be quoted.
}) implies that a subsequence, still denoted by \((w_1^\nu)\), converges \(\cC^\infty_\loc\) to a \(\wt J_1^\infty\)-holomorphic map \(w_1^\infty\) defined on \(\R^2\).
Its energy is bounded by \(E\), so after removing the singularity (using \cite[Theorem~4.1.2(i)]{ms:jh}) we obtain a \(\wt J_1^\infty\)-holomorphic sphere \(\ol w_1^\infty\colon S^2 \to M_1\), which is nonconstant by \eqref{eq:w1gradbound}.
Rescaling invariance of the energy and $\cC^\infty_\loc$-convergence imply energy concentration:
\begin{align*}
\liminf_{\nu\to\infty} \int_{B_{\eps^\nu}(s^\nu,t^\nu)} \tfrac 1 2\bigl| \d\ul{v}^\nu\bigr|^2
\geq \liminf_{\nu\to\infty} \int_{B_{\eps^\nu}(s^\nu, t^\nu)} v_1^{\nu\;*} \om_1 =\liminf_{\nu\to\infty} \int_{B_{R^\nu\eps^\nu}(0) \cap U_1^\nu} w_1^{\nu\;*} \om_1
\;\;&\geq\; \int_{\R^2} w_1^{\infty\;*} \om_1 \\
&\geq\; \hbar_{S^2} > 0.
\end{align*}

\item[\bf (S0,S2):] In complete analogy to (S1), rescaling by \(w_\ell^\nu(s,t) \coloneqq v_\ell^\nu(s^\nu + s/R^\nu, t^\nu + t/R^\nu)\) yields a nonconstant pseudoholomorphic {\bf sphere in \(\mathbf{M_\ell}\)} and with energy concentration of at least \(\hbar_{S^2}\).

\item[\bf (D01):]
We will obtain a {\bf disk bubble in \(\mathbf{ M_0^-\times M_1}\) with boundary on \(\mathbf{L_{01}}\)} by rescaling
\begin{align}
\label{eq:D01_rescale}
w_\ell^\nu(z) \coloneqq v_\ell^\nu \Bigl( s^\nu + \tfrac s {R^\nu}, -f^\nu\bigl(s^\nu + \tfrac s {R^\nu}\bigr) + \tfrac t {R^\nu} \Bigr)
\end{align}
to define maps \begin{align*}
w_0^\nu&\colon \bigl\{(s,t) \:|\: -R^\nu(\rho+s^\nu) \leq R^\nu(\rho - s^\nu), \: R^\nu\bigl(-\rho + f^\nu(s^\nu + \tfrac s {R^\nu})\bigr) \leq t \leq 0\bigr\} \;\longrightarrow\; M_0, \\
w_1^\nu&\colon \bigl\{(s,t) \:|\: -R^\nu(\rho+s^\nu) \leq R^\nu(\rho - s^\nu), \: 0 \leq t \leq 2R^\nu f^\nu(s^\nu + \tfrac s {R^\nu})\bigr\} \;\longrightarrow\; M_1.
\end{align*}
In the case of straight seams $t=\pm f^\nu(s) = \pm \delta^\nu$ the rescaled maps easily pair to maps $w_{01}^{\nu}: (s,t)\mapsto \bigl(w_0^\nu(-s,t),w_1^\nu(s,t)\bigr)\in  M_0^-\times M_1$ defined on increasing domains \(B_{r^\nu}(0) \cap \H\) with $r^\nu\to\infty$ in half space, with boundary values in $L_{01}$, which by standard arguments converge and extend to a pseudoholomorphic disk.
The squiggly seams require an easier version of the arguments in (E012) to establish convergence \(w_0^\nu \to w_0^\infty\), \(w_1^\nu \to w_1^\infty\) in \(\cC_\loc^\infty(-\H)\) resp.\ \(\cC_\loc^\infty(\H)\) to nonconstant \(J_\ell^\infty(s^\infty,0)\)-holomorphic maps satisfying the Lagrangian seam condition \begin{align*}
(w_0^\infty(s,0), w_1^\infty(s,0)) \in L_{01} \quad \forall s \in \R.
\end{align*}
Then \(w_{01}^\infty(s,t) \coloneqq (w_0^\infty(s,-t), w_1^\infty(s,t))\colon \H \to M_0^-\times M_1\) is a nonconstant \(\wt J_{01}^\infty \coloneqq (-J_0^\infty(s^\infty,0)) \times J_1^\infty(s^\infty,0)\)-holomorphic map with \(w_{01}^\infty(\partial\H) \subset L_{01}\).
Its energy is bounded by \(E\), so after removing the singularity (using e.g.\ \cite[Theorem~4.1.2(ii)]{ms:jh}) we obtain a nonconstant \(\wt J_{01}^\infty\)-holomorphic disk \(\ol w_{01}^\infty\colon D^2 \to M_0^-\times M_1\).
Energy quantization in the case of straight seams is given by
\begin{align*}
\liminf_{\nu\to\infty} \int_{B_{\eps^\nu}(s^\nu,t^\nu)} \tfrac 1 2\bigl| \d\ul{v}^\nu\bigr|^2
\geq \liminf_{\nu\to\infty} \sum_{\ell \in \{0,1\}} \int_{B_{\eps^\nu}(s^\nu,t^\nu)} v_\ell^{\nu\;*} \om_\ell &= \liminf_{\nu\to\infty} \int_{B_{r^\nu}(0)\cap\H} w_{01}^{\nu\;*} \bigl( (-\om_0) \oplus \om_1 \bigr) \\
&\hspace{0.5in} \geq \int w_{01}^{\infty\;*} \bigl( (-\om_0) \oplus \om_1 \bigr) \geq \hbar_{D^2},
\end{align*}
and in the general case just requires more refined choices of domains as in (E012).

\item[\bf (D12):] In complete analogy to (D01), rescaling yields a nonconstant pseudoholomorphic {\bf disk in \\\(\mathbf{ M_1^-\times M_2}\) with boundary on \(\mathbf{ L_{12}}\)} and energy concentration of at least \(\hbar_{D^2}\).

\item[\bf (E012):] We will obtain a {\bf figure eight bubble between \(\mathbf{L_{01}}\) and \(\mathbf{L_{12}}\)} by a rescaling that in the case of straight middle strips of width $2 f^\nu \equiv \delta^\nu \to 0$ amounts to $w_\ell^\nu(s,t) = v_\ell^\nu \bigl( s^\nu + \delta^\nu s, \delta^\nu t \bigr)$ for $\ell=0,1,2$.
In the general case of squiggly strips, we straighten out the strip by using an $s$-dependent rescaling factor in the $t$ variable:
\begin{align}\label{eq:E012_rescale}
w_\ell^\nu(s,t) \coloneqq v_\ell^\nu \bigl(\phi^\nu(s,t) \bigr), \qquad \phi^\nu(s,t) \coloneqq \Bigl(s^\nu + 2f^\nu(s^\nu)s, 2f^\nu\bigl(s^\nu + 2f^\nu(s^\nu)s\bigr)\, t\Bigr).
\end{align}
Note that $\phi^\nu$ is a diffeomorphism between open subsets of $\R^2$ since $f^\nu>0$, and that it pulls back the seams $\cS_\pm =\bigl\{(x,y)\in\R^2 \,|\, y=\pm f^\nu(x)\bigr\}$ to straight seams $(\phi^\nu)^{-1}(\cS_\pm) = \bigl\{(s,t) \in \R^2\,|\, t=\pm\tfrac 12 \bigr\}$ as in Definition~\ref{def:8} of the figure eight bubble.
Moreover, the rescaled quilt maps $(w_\ell^\nu)_{\ell=0,1,2}$ have the total domain
\begin{align*}
(\phi^\nu)^{-1}\bigl( (-\rho,\rho)^2 \bigr)
= \Bigl\{ (s,t)\in\R^2 \,\Big| \, \tfrac{-\rho-s^\nu}{2f^\nu(s^\nu)} < s < \tfrac{\rho - s^\nu}{2f^\nu(s^\nu)}, \: |t| < \tfrac{\rho}{2f^\nu(s^\nu + 2f^\nu(s^\nu)s)} \Bigr\}
\end{align*}
Since \(s^\nu \to s^\infty \in (-\rho,\rho)\) and \(f^\nu \sr{\cC^\infty}{\lra} 0\), we can find a sequence of radii \(r^\nu \to \infty\) so that these domains contain the balls $B_{r^\nu}(0)$.
In the case of straight seams the maximal radii are $r^\nu =  (\rho-|s^\nu|)/\delta^\nu$, but we may also choose smaller radii \(r^\nu \to \infty\) so that in addition $r^\nu\delta^\nu\to 0$.
In general we choose \(r^\nu \to \infty\) so that $B_{r^\nu}(0)\subset (\phi^\nu)^{-1}\bigl( (-\rho,\rho)^2 \bigr)$ and
\begin{equation} \label{a}
 r^\nu \max_{s \in [-\rho,\rho]} \bigl( f^\nu(s) + (f^\nu)'(s)\bigr) \underset{\nu\to\infty}{\longrightarrow} 0 .
\end{equation}
Finally, we wish to choose $r^\nu\to\infty$ such that in addition the inclusion \(\phi^\nu\bigl(B_{r^\nu}(0)\bigr) \subset B_{\eps^\nu}(s^\nu,t^\nu)\) ensures that the gradient bounds \eqref{eq:vbounds} transfer to the rescaled maps. 
For that purpose first note that for sufficiently large $\nu\in\N$ from \eqref{a} we also obtain the estimate
\begin{align}\label{c}
\max_{s \in [-r^\nu,r^\nu]} 2f^\nu(s^\nu + 2f^\nu(s^\nu)s) \;\leq\; 3f^\nu(s^\nu) .
\end{align}
Indeed, for \(s \in [-r^\nu, r^\nu]\) and $\nu$ sufficiently large such that $r^\nu \|(f^\nu)'\|_{\cC^0([-\rho,\rho])}\leq \frac 14$ we have
\begin{align*}
f^\nu(s^\nu + 2f^\nu(s^\nu)s) &\leq f^\nu(s^\nu) + \int_0^s 2f^\nu(s^\nu)(f^\nu)'(s^\nu + 2f^\nu(s^\nu)s) \,ds \\
&\leq f^\nu(s^\nu)\bigl(1 + 2r^\nu \|(f^\nu)'\|_{\cC^0([-\rho,\rho])}\bigr) 
\;\leq\; \tfrac 32 f^\nu(s^\nu).
\end{align*}
Next, for \((s,t) \in B_{r^\nu}(0)\) we obtain
\begin{align*}
\left|\bigl(s^\nu + 2f^\nu(s^\nu)s, 2f^\nu(s^\nu + 2f^\nu(s^\nu)s)t\bigr) - (s^\nu, t^\nu)\right| &\leq \left|(2f^\nu(s^\nu)s, 2f^\nu(s^\nu + 2f^\nu(s^\nu)s)t)\right| + |(0,t^\nu)| \\
&\leq 3r^\nu f^\nu(s^\nu) + |t^\nu| \\
&= \frac {3 r^\nu(\tau_+^\nu - \tau_-^\nu) + |\tau_+^\nu + \tau_-^\nu|} {2 R^\nu\eps^\nu} \; \eps^\nu
\end{align*}
from \eqref{c} and the identities
\begin{equation}\label{eq:tau}
 f^\nu(s^\nu) = \frac{\tau_+^\nu - \tau_-^\nu}{2 R^\nu},
\qquad
t^\nu = - \frac{\tau_+^\nu + \tau_-^\nu}{2 R^\nu}.
\end{equation}
Thus to obtain the inclusion \(\phi^\nu\bigl(B_{r^\nu}(0)\bigr) \subset B_{\eps^\nu}(s^\nu,t^\nu)\) for large $\nu$ it suffices to replace the above $r^\nu$ by a possibly smaller sequence $r^\nu\to\infty$ so that $3 r^\nu(\tau_+^\nu - \tau_-^\nu) + |\tau_+^\nu + \tau_-^\nu| < 2 R^\nu\eps^\nu$ for large $\nu$.
This is possible since \(R^\nu\eps^\nu \to \infty\) and \(\tau_\pm^\nu \to \tau_\pm^\infty \in \R\).
So from now on, after dropping finitely many terms from the sequence, we are working with a sequence of rescaled quilt maps \eqref{eq:E012_rescale} together with a sequence \(r^\nu \to \infty\) so that we have the inequalities \eqref{a}, \eqref{c}, and the inclusions
\begin{equation}\label{bd}
\phi^\nu\bigl(B_{r^\nu}(0)\bigr) \subset B_{\eps^\nu}(s^\nu,t^\nu) \cap (-\rho,\rho)^2 .
\end{equation}
Thus, restricting the maps from \eqref{eq:E012_rescale} to the balls $B_{r^\nu}(0)$ of increasing radius amounts to considering the quilt map  \(w_\ell^\nu\colon W_\ell^\nu \to M_\ell\) with the domains 
\begin{align*}
W_0^\nu \coloneqq B_{r^\nu}(0) \cap \bigl(\R \times (-\infty, -\tfrac 1 2]\bigr), \:\:\: W_1^\nu \coloneqq B_{r^\nu}(0) \cap \bigl(\R \times [-\tfrac 1 2, \tfrac 1 2]\bigr), \:\:\: W_2^\nu \coloneqq B_{r^\nu}(0) \cap \bigl(\R \times [\tfrac 1 2, \infty)\bigr).
\end{align*}
These maps are \((\wt J_\ell^\nu, j^\nu)\)-holomorphic, where \(\wt J_\ell^\nu =  J_\ell^\nu \circ \phi^\nu\) are the almost complex structures on $M_\ell$ with appropriately rescaled domain dependence, and \(j^\nu\) is the complex structure 
\begin{align*}
j^\nu(s,t) &\coloneqq \bigl(\d\phi^\nu(s,t)\bigr)^{-1}\circ j_0 \circ \d\phi^\nu(s,t) \\
&= \left(\begin{array}{ll}
2f^\nu(s^\nu) & 0 \\
4f^\nu(s^\nu)(f^\nu)'(s^\nu + 2f^\nu(s^\nu)s)\,t & 2f^\nu(s^\nu + 2f^\nu(s^\nu)s)
\end{array}\right)^{-1}\left(\begin{array}{ll}
0 & -1 \\
1 & 0
\end{array}\right)\times \\
&\hspace{2in} \times\left(\begin{array}{ll}
2f^\nu(s^\nu) & 0 \\
4f^\nu(s^\nu)(f^\nu)'(s^\nu + 2f^\nu(s^\nu)s)\,t & 2f^\nu(s^\nu + 2f^\nu(s^\nu)s)
\end{array}\right) \\
&= \left(\!\!\begin{array}{ll}
(2f^\nu(s^\nu))^{-1}& 0 \\
- \frac{(f^\nu)'(\ldots)\,t}{f^\nu(s^\nu + \ldots)} & \bigl(2 f^\nu(s^\nu + \ldots)\bigr)^{-1}
\end{array}\!\!\right)\!\!
\left(\begin{array}{ll}
- 4f^\nu(s^\nu)(f^\nu)'(s^\nu + \ldots)\,t & - 2f^\nu(s^\nu + 2f^\nu(s^\nu)s) \\
 2f^\nu(s^\nu) & 0
\end{array}\!\!\right)
\\
&= \left(\begin{array}{cc}
-2t(f^\nu)'(s^\nu + 2f^\nu(s^\nu)s) & -\tfrac {f(s^\nu+2f^\nu(s^\nu)s)} {f^\nu(s^\nu)} \\
\tfrac {f^\nu(s^\nu)\bigl( 4t^2(f^\nu)'(s^\nu+2f^\nu(s^\nu)s)^2 + 1\bigr)} {f^\nu(s^\nu+2f^\nu(s^\nu)s)} & 2t(f^\nu)'(s^\nu+2f^\nu(s^\nu)s)
\end{array}\right) \\
&= \left( \begin{array}{cc}
-\tfrac {t(\wt f^\nu)'(s)} {\wt f^\nu(0)} & -\tfrac {\wt f^\nu(s)} {\wt f^\nu(0)} \\
\tfrac {t^2(\wt f^\nu)'(s)^2 + \wt f^\nu(0)^2} {\wt f^\nu(0)\wt f^\nu(s)} & \tfrac {t(\wt f^\nu)'(s)} {\wt f^\nu(0)}
\end{array} \right),
\end{align*}
where we abbreviate
$\wt f^\nu(s) \coloneqq f^\nu(s^\nu + 2f^\nu(s^\nu)s) / 2f^\nu(s^\nu)$.
Note that Lemma~\ref{lem:rescaled-width} with \(\alpha^\nu \coloneqq (2f^\nu(s^\nu))^{-1}\) implies \(j^\nu \to i\) in \(\cC^\infty_\loc\), and the almost complex structures also converge \(\wt J_\ell^\nu \to J_\ell^\infty(s^\infty,0)\) in \(\cC^\infty_\loc\) since $\phi^\nu(s,t) \to 0$ for any fixed $(s,t)$.
Moreover, the maps \(w_\ell^\nu\) satisfy the Lagrangian seam conditions \begin{align*}
(w_0^\nu(s, -\tfrac 1 2 ), w_1^\nu(s, -\tfrac 1 2 )) \in L_{01}, \quad (w_1^\nu(s,\tfrac 1 2), w_2^\nu(s, \tfrac 1 2)) \in L_{12} \qquad \forall \: s \in (-r^\nu, r^\nu).
\end{align*}
The gradient blowup at $(s^\nu,t^\nu)$ in \eqref{eq:vbounds} translates into lower bounds on the gradient \(|\d w^\nu| \coloneqq \bigl(|\d w_0^\nu| + |\d w_1^\nu| + |\d w_2^\nu|\bigr)^{1/2}\) at $t=\wt t^\nu \coloneqq \tfrac{t^\nu}{2f^\nu(s^\nu)} = \tfrac {-\tau_+^\nu - \tau_-^\nu} {2(\tau_+^\nu - \tau_-^\nu)} \to \tfrac {-\tau_+^\infty-\tau_-^\infty}{2(\tau_+^\infty-\tau_-^\infty)} \in \R$, since $\phi^\nu(0,\tfrac{t^\nu}{2f^\nu(s^\nu)})=(s^\nu,t^\nu)$ and hence
\begin{align*} 
\bigl|\d \ul w^\nu( 0, \wt t^\nu) \bigr|^2 
&= 
\sum_\ell \bigl|2f^\nu(s^\nu)\partial_sv^\nu_\ell(s^\nu,t^\nu) + 4\wt t^\nu f^\nu(s^\nu)(f^\nu)'(s^\nu)\partial_tv^\nu_\ell(s^\nu,t^\nu)\bigr|^2 + \bigl|2f^\nu(s^\nu)\partial_tv_\ell^\nu(s^\nu,t^\nu)\bigr|^2 \nonumber \\
&\geq 
3\bigl|\d\ul v^\nu(s^\nu,t^\nu)\bigr|^2, \quad \nu \gg 0.
\end{align*}
Now by $t^\nu\to t^\infty$, the obedient convergence $f^\nu \Rightarrow 0$ in Definition~\ref{def:obedience}, and \eqref{eq:tau} we obtain a nonzero lower bound for sufficiently large $\nu$,
\begin{equation} \label{eq:E012bounds}
\bigl|\d \ul w^\nu( 0, \wt t^\nu) \bigr|^2 
\;\geq \;
2 f^\nu(s^\nu)^2|\d\ul v^\nu(s^\nu,t^\nu)|^2 
\;=\; 
\tfrac 1 2 (\tau^\nu_+ - \tau^\nu_-)^2
\;\geq\;  
\tfrac 1 4(\tau^\infty_+ - \tau^\infty_-)^2 > 0.
\end{equation}
Next we use \eqref{a}--\eqref{bd} to transfer the upper bound in \eqref{eq:vbounds}
to the gradient of the rescaled maps for sufficiently large $\nu$,
\begin{align}
&\sup_{(s,t) \in B_{r^\nu}(0)} |\d\ul w^\nu(s,t)|^2  \nonumber\\
&\qquad= 
\sup_{(s,t) \in B_{r^\nu}(0)} \sum_\ell 
\begin{array}{r} 
\bigl| 2f^\nu(s^\nu)\partial_sv_\ell^\nu(\phi^\nu(s,t)) + 4tf^\nu(s^\nu)(f^\nu)'(s^\nu+2f^\nu(s^\nu)s)\partial_tv_\ell^\nu(\phi^\nu(s,t))\bigr|^2  \nonumber\\
+ \phantom{\Big|} \bigl|2f^\nu(s^\nu + 2f^\nu(s^\nu)s)\partial_tv_\ell^\nu(\phi^\nu(s,t))\bigr|^2
\end{array} 
 \nonumber\\
&\qquad\leq 
\Bigl( \bigl( 2f^\nu(s^\nu) + 4 r^\nu f^\nu(s^\nu) \max_{s\in[-\rho,\rho]} |(f^\nu)'(s)| \bigr)^2
+ \bigl(3f^\nu(s^\nu)\bigr)^2 
\Bigr)
\sup_{(x,y) \in B_{\eps^\nu}(s^\nu,t^\nu)} |\d\ul v(x,y)|^2  \nonumber\\
&\qquad \leq 
18f^\nu(s^\nu)^2
\sup_{(x,y) \in B_{\eps^\nu}(s^\nu,t^\nu)} |\d\ul v(x,y)|^2 
\;\leq\;18 (\tau_+^\infty - \tau_-^\infty)^2 .  
\label{eq:wbound}
\end{align}
Using these gradient bounds (and the compact boundary conditions in the case of noncompact symplectic manifolds), standard Gromov compactness asserts that after passing to a subsequence, \(w_0^\nu\) resp.\ \(w_1^\nu\) resp.\ \(w_2^\nu\) converge \(\cC^\infty_\loc\) on the interior of \(\R\times(-\infty, -1/2]\) resp.\ \(\R \times [-1/2,1/2]\) resp.\ \(\R \times [1/2,\infty)\).

To obtain convergence up to the seams $\R\times\{\pm 1/2\}$, we will first prove convergence of somewhat differently rescaled maps.
More precisely, to prove convergence of \(w_0^\nu, w_1^\nu\) near the \(L_{01}\)-seam \(\R \times \{-1/2\}\) we consider the maps
$$
u_0^\nu\colon \; U_0^\nu \coloneqq (-r^\nu, r^\nu) \times (-1/2, 0]  \; \to \; M_0,
\qquad
u_1^\nu\colon \; U_1^\nu \coloneqq (-r^\nu, r^\nu) \times [0, 1/2) \; \to \; M_1
$$
given by rescaling $u_\ell^\nu(s,t) \coloneqq v_\ell^\nu\bigl(\psi^\nu(s+it)\bigr)$ with the holomorphic map
$$
\psi^\nu(z) \coloneqq s^\nu + 2f^\nu(s^\nu)z - i F^\nu\bigl( s^\nu + 2f^\nu(s^\nu)z \bigr),
$$
where we identify \((s,t)\in \R^2\) with \(z=s+it \in \C \), and \(F^\nu\) is the extension of \(f^\nu\) from Definition~\ref{def:obedience}.
To see that $\psi^\nu$ is well-defined on \(U_0^\nu \cup U_1^\nu\) for sufficiently large \(\nu\), despite $f^\nu$ resp.\ $F^\nu$ only being defined on \([-\rho,\rho]\) resp.\ \([-\rho,\rho]^2\), note that \(s^\nu \to s^\infty \in (-\rho,\rho)\) and \(r^\nu f^\nu(s^\nu) \to 0\) by \eqref{a}.
To ensure that \(u_\ell^\nu\) is well-defined for large \(\nu\) and $\ell=0,1$ we moreover need to verify that $\psi^\nu(U_\ell^\nu)$ lies in the domain of $v_\ell^\nu$.
Indeed, firstly we have \(\psi^\nu(U_0^\nu\cup U_1^\nu) \subset (-\rho,\rho)^2\) for large \(\nu\) by \(s^\nu \to s^\infty \in (-\rho,\rho)\), \eqref{a}, and \(F^\nu \sr{\cC^\infty}{\lra} 0\).
Secondly, the bounds required by the seams are 
\begin{align}
\im \psi^\nu(s+it) &\leq -f^\nu(\re \psi^\nu(s+it)) \;\quad\forall\: s+it \in U_0^\nu,
\label{eq:U0seambound} \\
\bigl| \im \psi^\nu(s+it) \bigr| &\leq f^\nu(\re \psi^\nu(s+it)) \qquad\forall\: s+it \in U_1^\nu
\label{eq:U1seambound}
\end{align}
for large \(\nu\).
For that purpose we rewrite
\begin{align*}
& \im \psi^\nu(s+it) \pm f^\nu(\re \psi^\nu(s+it)) \\
&=
2tf^\nu(s^\nu) - \re F^\nu\bigl(s^\nu+2f^\nu(s^\nu)(s+it)\bigr) \pm f^\nu\Bigl(s^\nu + 2 f^\nu(s^\nu) s
+ \im F^\nu \bigl(s^\nu+2f^\nu(s^\nu)(s+it)\bigr) \Bigr) \\
%
%
&= (2t-1\pm1)f^\nu\bigl(s^\nu+2f^\nu(s^\nu)s\bigr) + E_1^\nu(s,t) - \re E_2^\nu(s,t) \pm E_3^\nu(s,t).
\end{align*}
with
\begin{align*}
E_1^\nu(s,t) &= 2t \bigl( f^\nu(s^\nu) - f^\nu\bigl(s^\nu+2f^\nu(s^\nu)s\bigr) \bigr), \\
E_2^\nu(s,t) &= F^\nu\bigl(s^\nu+2f^\nu(s^\nu)(s+it)\bigr)-F^\nu\bigl(s^\nu+2f^\nu(s^\nu)s\bigr) , \\
E_3^\nu(s,t) &= f^\nu\Bigl(s^\nu+2f^\nu(s^\nu)s +\im F^\nu\bigl(s^\nu+2f^\nu(s^\nu)(s+it)\bigr)\Bigr) - f^\nu\bigl(s^\nu+2f^\nu(s^\nu)s\bigr).
\end{align*}
We can bound \(E_1^\nu,E_2^\nu,E_3^\nu\) for large \(\nu\) and \((s,t)\in U_0^\nu\cup U_1^\nu\subset B_{2r^\nu}(0)\), using the obedient convergence \(f^\nu \Rightarrow 0\) from Definition~\ref{def:obedience},
\begin{align}
|E_1^\nu(s,t)| &\leq 4t \|(f^\nu)'\|_{\cC^0}f^\nu(s^\nu)s \leq 8 t r^\nu\|(f^\nu)'\|_{\cC^0} C_0 f^\nu\bigl(s^\nu+2f^\nu(s^\nu)s\bigr),
\end{align}
\begin{align}
|E_2^\nu(s,t)| &\leq  \|DF^\nu\|_{\cC^0} 2t f^\nu(s^\nu) \leq 2t \|DF^\nu\|_{\cC^0} C_0 f^\nu\bigl(s^\nu+2f^\nu(s^\nu)s\bigr),
\end{align}
\begin{align}
|E_3^\nu(s,t)| 
&\leq \|(f^\nu)'\|_{\cC^0}\bigl|\im F^\nu(s^\nu+2f^\nu(s^\nu)(s+it))\bigr| \label{eq:E3bound} \\
&= \|(f^\nu)'\|_{\cC^0}\Bigl|\im \bigl( F^\nu(s^\nu+2f^\nu(s^\nu)(s+it)) - F^\nu(s^\nu+2f^\nu(s^\nu)s) \bigr)\Bigr| \nonumber \\
&\leq \|(f^\nu)'\|_{\cC^0} \|DF^\nu\|_{\cC^0} 2 t f^\nu(s^\nu) \nonumber \\
&\leq 2t \|(f^\nu)'\|_{\cC^0} \|DF^\nu\|_{\cC^0} C_0 f^\nu\bigl(s^\nu+2f^\nu(s^\nu)s\bigr). 
\nonumber
\end{align}
Then by \eqref{a} and  \(F^\nu \sr{\cC^\infty}{\lra} 0\) we obtain for sufficiently large $\nu$\begin{align*}
\bigl| \im \psi^\nu(s+it) \pm f^\nu(\re \psi^\nu(s+it)) 
- (2t-1\pm1)f^\nu\bigl(s^\nu+2f^\nu(s^\nu)s\bigr) \bigr| 
\leq tf^\nu\bigl(s^\nu+2f^\nu(s^\nu)s\bigr).
\end{align*}
To check \eqref{eq:U0seambound} from this, recall that  \(t \in (-1/2,0]\) on $U^\nu_0$ so that
$$
\im \psi^\nu(s+it) + f^\nu(\re \psi^\nu(s+it)) \leq 
t f^\nu\bigl(s^\nu+2f^\nu(s^\nu)s\bigr)  \leq 0 .
$$
Similarly, on $U^\nu_1$ we have  \(t \in [0,1/2)\) so that \eqref{eq:U1seambound} follows from
\begin{align*}
f^\nu(\re\psi^\nu(s+it)) \pm \im\psi^\nu(s+it)
\geq
\bigl(1 - t \pm (2t-1)\bigr) f^\nu\bigl(s^\nu+2f^\nu(s^\nu)s\bigr) \geq 0
\end{align*}
since $f^\nu> 0$, $1-t-2t+1 = 2 - 3t > 0$ and $1-t+2t-1 = t\geq 0$.

Now that \(u_0^\nu, u_1^\nu \) are well-defined, note that the advantage of this rescaling is that the resulting maps are pseudoholomorphic with respect to the standard complex structure \(i\) on their domains (viewed as subsets of $\C$).
On the other hand it straightens out only one seam,
\begin{align*}
\psi^\nu \bigl( (-r^\nu,r^\nu)\times\{0\} \bigr) 
& =\bigl\{ \bigl(
s^\nu + 2f^\nu(s^\nu)s , - f^\nu\bigl( s^\nu + 2f^\nu(s^\nu)s \bigr)
\bigr) \,\big|\, s\in(-r^\nu,r^\nu) \bigr\} \\
&\subset
\bigl\{(x,y)\in\R^2 \,|\, y=- f^\nu(x)\bigr\}
\;=\;\cS_- ,
\end{align*}
so that we obtain the Lagrangian seam condition 
\begin{align*}
(u_0^\nu(s, 0 ), u_1^\nu(s, 0 )) \in L_{01} \qquad \forall \: s \in (-r^\nu, r^\nu),
\end{align*}
but the $L_{12}$-condition would hold on the curved seam $(\psi^\nu)^{-1}(\cS_+)$.
However, we use this rescaling only to prove convergence near the \(L_{01}\)-seam, and to prove convergence of \(w_1^\nu, w_2^\nu\) near the \(L_{12}\)-seam would use the rescaling $z\mapsto s^\nu + 2f^\nu(s^\nu)z + i F^\nu\bigl( s^\nu + 2f^\nu(s^\nu)z \bigr)$.

More precisely, we will below prove $\cC^\infty_\loc$-convergence of $u_\ell^\nu$ near $\R\times\{0\}$ since this yields control of the maps of interest $w_\ell^\nu$ for $\ell=0,1$ near the seam $\R\times\{-\frac 12\}$. 
Indeed, $w_\ell^\nu = u_\ell^\nu \circ ((\psi^\nu)^{-1} \circ \phi^\nu)$ is obtained from $u_\ell^\nu$ by composition with the local diffeomorphisms $(\psi^\nu)^{-1} \circ \phi^\nu$ which by Lemma~\ref{lem:rescaled-width} converge $\cC^\infty_\loc$ to a shift map. 
On the other hand, to establish convergence of the $u_\ell^\nu$, we can start from local uniform gradient bounds given by \eqref{eq:wbound} via the reparametrization with $(\psi^\nu)^{-1} \circ \phi^\nu$.
Further, we can work with the ``folded'' maps \(u_{01}^\nu\colon (-r^\nu, r^\nu) \times [0, 1/2) \to M_0^- \times M_1\) given by \(u_{01}^\nu(s,t) \coloneqq (u_0^\nu(s,-t), u_1^\nu(s,t))\).
These satisfy the Lagrangian boundary condition \(u_{01}^\nu(s,0) \in L_{01}\) for \(s \in (-r^\nu,r^\nu)\) and are pseudoholomorphic with respect to \(K_{01}^\n(s,t) \coloneqq \bigl( - J_0^\nu \bigl( \psi^\nu(s-it) \bigr)\bigr) \times J_1^\nu \bigl( \psi^\nu(s+it) \bigr) \), which converges to \(K_{01}^\infty \coloneqq \bigl( -J_0^\infty(s^\infty,0)\bigr) \times J_1^\infty(s^\infty,0)\) in  \(\cC^\infty_\loc\).
Now standard compactness for pseudoholomorphic maps implies that after passing to a subsequence, \((u_{01}^\nu)\) converges \(\cC^\infty_\loc\) on \(\R\times[0,1/2)\), and as discussed above this implies \(\cC^\infty_\loc\)-convergence for the corresponding subsequence of \(w_0^\nu, w_1^\nu\) near the \(L_{01}\)-seam \(\R \times \{-1/2\}\).  

An analogous argument shows that \(w_1^\nu, w_2^\nu\) converge \(\cC^\infty_\loc\) near \(\R \times \{1/2\}\), so we have now shown that \(w_0^\nu, w_1^\nu, w_2^\nu\) converge \(\cC^\infty_\loc\) everywhere to a \((J_0^\infty(s^\infty,0),J_1^\infty(s^\infty,0),J_2^\infty(s^\infty,0))\)-holomorphic figure eight bubble \(\ul w^\infty\) between \(L_{01}\) and \(L_{12}\).  
The lower gradient bound in \eqref{eq:E012bounds} implies that \(\ul w^\infty\) is nonconstant and hence has nonzero energy, hence by Lemma~\ref{lem:hbar} has energy at least \(\hbar_8>0\). 
Finally, rescaling invariance and \(\cC^\infty_\loc\)-convergence imply energy concentration of at least \(\hbar_8\) at \((s_{N+1},0)\):
\begin{align*}
\liminf_{\nu\to\infty} \int_{B_{\eps^\nu}(s^\nu,t^\nu)} \tfrac 1 2\bigl|\d\ul v^\nu\bigr|^2 \geq \liminf_{\nu\to\infty} \sum_{\ell\in\{0,1,2\}} \int_{U_\ell^\nu} {w_\ell^\nu}^*\omega_\ell \geq \int_{\R^2} \bigl|\d\ul w^\nu\bigr|^2 \geq \hbar_8>0.
\end{align*}

\item[\bf (D02):]
We will obtain a {\bf squashed eight bubble in \(\mathbf{ M_{02}}\) with boundary on \(\mathbf{L_{01}\circ L_{12}}\)} by rescaling
\begin{align}
\label{eq:D02_rescale}
w_\ell^\nu(s, t) \coloneqq v_\ell^\nu \Bigl( s^\nu + \tfrac s {R^\nu}, \tfrac {\wt f^\nu(s)} {\wt f^\nu(0)} \tfrac t {R^\nu} \Bigr),
\end{align}
where we have set \(\wt f^\nu(s) \coloneqq R^\nu f^\nu(s^\nu + s/R^\nu)\), to obtain maps 
\begin{align*}
w_0^\nu&\colon 
U_0^\nu\coloneqq
\bigl\{ (s,t) \: | \: -R^\nu(\rho + s^\nu) \leq s \leq R^\nu(\rho - s^\nu), \: -\tfrac {\rho R^\nu \wt f^\nu(0)} {\wt f^\nu(s)} \leq t \leq -\wt f^\nu(0) \bigr\} \to M_0, \\
w_1^\nu&\colon 
U_1^\nu\coloneqq\bigl\{ (s,t) \: | \: -R^\nu(\rho + s^\nu) \leq s \leq R^\nu(\rho - s^\nu), \: -\wt f^\nu(0) \leq t \leq \wt f^\nu(0) \bigr\} \to M_1, \\
w_2^\nu&\colon 
U_2^\nu\coloneqq\bigl\{ (s,t) \: | \: -R^\nu(\rho + s^\nu) \leq s \leq R^\nu(\rho - s^\nu), \: \wt f^\nu(0) \leq t \leq \tfrac {\rho R^\nu \wt f^\nu(0)} {\wt f^\nu(s)} \bigr\} \to M_2.
\end{align*}
Each \(w_\ell^\nu\) is pseudoholomorphic with respect to \((\wt J_\ell^\nu(s,t), j^\nu)\), where the almost complex structure \(\wt J_\ell^\nu(s,t) \coloneqq J_\ell^\nu(s^\nu + s/R^\nu, \wt f^\nu(s) t / \wt f^\nu(0) R^\nu )\) converges \(\cC^\infty_\loc\) to \(\wt J_\ell^\infty \coloneqq J_\ell^\infty(s_\infty,0)\) and the symmetric\footnote{As explained in the beginning of \S\ref{sec:squiggly}, a complex structure is said to be {\bf symmetric} if it is equivariant with respect to conjugation.
This is a necessary hypothesis for Thm.~\ref{thm:nonfoldedstripshrink}.} complex structure \(j^\nu\) on \(U_0^\nu \cup U_1^\nu \cup U_2^\nu \subset \R^2\),
\begin{align*}
j^\nu(s,t) \coloneqq \left( \begin{array}{cc}
-\tfrac {t(\wt f^\nu)'(s)} {\wt f^\nu(0)} & -\tfrac {\wt f^\nu(s)} {\wt f^\nu(0)} \\
\tfrac {t^2(\wt f^\nu)'(s)^2 + \wt f^\nu(0)^2} {\wt f^\nu(0)\wt f^\nu(s)} & \tfrac {t(\wt f^\nu)'(s)} {\wt f^\nu(0)}
\end{array} \right)
\end{align*}
converges \(\cC^\infty_\loc\) to the standard complex structure by Lemma~\ref{lem:rescaled-width} 
with \(\alpha^\nu \coloneqq R^\nu\).
The maps \(w_\ell^\nu\) also satisfy the Lagrangian seam conditions
\begin{gather*}
(w_0^\nu(s, -\wt f^\nu(0)), w_1^\nu(s, -\wt f^\nu(0))) \in L_{01}, \qquad (w_1^\nu(s, \wt f^\nu(0)), w_2^\nu(s, \wt f^\nu(0))) \in L_{12}
\end{gather*}
for all \(s\) in \((-R^\nu(\rho + s^\nu), R^\nu(\rho - s^\nu))\).
By the \(\cC^\infty_\loc\)-convergence \(\wt f^\nu(s) / \wt f^\nu(0) \to 1\) proven in Lemma~\ref{lem:rescaled-width}, we may choose a subsequence and \(r^\nu \to \infty\) with \(r^\nu \leq R^\nu \eps^\nu / 4\) such that we have
\begin{align} \label{eq:D02-tilde-f-bound}
\|\wt f^\nu(s) / \wt f^\nu(0)\|_{\cC^1((-r^\nu,r^\nu))} \leq 2
\end{align}
for \(\nu\) sufficiently large.
This allows us to translate \eqref{eq:vbounds} into upper and lower bounds on the energy density \(| \d \ul w^\nu |\) for sufficiently large $\nu$,
\begin{gather*}
| \d\ul w^\nu(0, R^\nu t^\nu) | 
\geq
\tfrac 1 {2R^\nu} | \d\ul v^\nu(s^\nu, t^\nu) | \geq \tfrac 1 2, \\
\sup_{(s,t) \in B_{r^\nu}(0)} | \d \ul w^\nu(s,t) | \leq 
\sup_{(s,t) \in B_{\eps^\nu}(s^\nu,t^\nu)} 
\tfrac 4 {R^\nu}| \d \ul v^\nu(s,t) | \leq 8.
\end{gather*}
Here we estimated 
$\tfrac1{R^\nu} | \partial_s v_\ell^\nu|- \tfrac{|(\wt f^\nu)'|}{R^\nu\wt f^\nu(0)}| \partial_t v_\ell^\nu| \leq \bigl| \partial_s w_\ell^\nu \bigr| \leq \left(\tfrac1{R^\nu} + \tfrac{|(\wt f^\nu)'|}{R^\nu\wt f^\nu(0)}\right) | \d v_\ell^\nu|$, used the identity $\bigl| \partial_t w_\ell^\nu \bigr| = \tfrac{\wt f^\nu}{R^\nu\wt f^\nu(0)} | \partial_t v_\ell^\nu|$, and need to check that $(s,t)\in B_{r^\nu}(0)$ implies $\bigl( s^\nu + \tfrac s {R^\nu}, \tfrac {\wt f^\nu(s)} {\wt f^\nu(0)} \tfrac t {R^\nu} \bigr) \in B_{\eps^\nu}(s^\nu,t^\nu)$ for sufficiently large $\nu$.
Indeed, \eqref{eq:D02-tilde-f-bound} yields:
\begin{align} \label{eq:D02-containment}
\bigl| \bigl( \tfrac s {R^\nu} \,,\, \tfrac {\wt f^\nu(s)} {\wt f^\nu(0)} \tfrac t {R^\nu} - t^\nu \bigr)\bigr|  /  \eps^\nu 
\;\leq\;
\tfrac {2r^\nu}{R^\nu\eps^\nu} + \tfrac{|t^\nu|}{\eps^\nu}
\;\leq\;
\tfrac 12 + \tfrac{\tau_+^\nu + \tau_-^\nu}{2 R^\nu \eps^\nu} ,
\end{align}
where \(R^\nu \eps^\nu \to \infty\) and $\tau_\pm^\nu\to\tau_\pm^\infty \in \R$.
Next we consider the limiting behaviour of the domain \(U_0^\nu \cup U_1^\nu \cup U_2^\nu\).
Its straight boundaries diverge \(-R^\nu(\rho + s^\nu) \to -\infty\) resp.\  \(R^\nu(\rho - s^\nu)\to \infty\) since $s^\nu\to s^\infty \in (-\rho,\rho)$.
The functions  \(\pm \rho R^\nu \wt f^\nu(0) / \wt f^\nu(s)\) of the upper/lower boundary converge \(\cC^\infty_\loc\) to \(\infty\) resp.\ \(-\infty\) by Lemma~\ref{lem:rescaled-width} with \(\alpha^\nu \coloneqq R^\nu\) and $R^\nu\to\infty$.
Finally, the straight seams \(\{ t = \pm \wt f^\nu(0)\} \) shrink to a single seam $\{t=0\}$ since we have \(\wt f^\nu(0)= \tfrac 1 2(\tau_+^\nu - \tau_-^\nu) \to 0 \).

Now we may apply Theorem~\ref{thm:nonfoldedstripshrink} to this strip shrinking situation to deduce that after passing to a subsequence, \((w_0^\nu(s, t - \wt f^\nu(0))\) resp.\ \((w_2^\nu(s, t + \wt f^\nu(0))\) converge in \(\cC^\infty_\loc(-\H)\) resp.\ \(\cC^\infty_\loc(\H)\) to \(\wt J_0^\infty\)- resp.\ \(\wt J_2^\infty\)-holomorphic maps \(w_0^\infty\) resp.\ \(w_2^\infty\), and that \((w_1^\nu|_{t=0})\) converges in \(\cC^\infty_\loc(\R)\) to a smooth map \(w_1^\infty\).
Furthermore, at least one of \(w_0^\infty, w_2^\infty\) is nonconstant, and the generalized seam condition \((w_0^\infty(s,0), w_1^\infty(s), w_1^\infty(s), w_2^\infty(s,0)) \in L_{01} \times_{M_1} L_{12}\) is satisfied for \(s \in \R\), so that \((w_0^\infty, w_2^\infty)\) is a nonconstant squashed eight bubble with boundary on \(L_{01} \circ L_{12}\), with energy bounded below by \(\hbar_{L_{01} \circ L_{12}}\).
Finally, rescaling invariance, \(\cC^\infty_\loc\)-convergence, and the containment proven in \eqref{eq:D02-containment} imply energy concentration: 
\begin{align*}
\liminf_{\nu \to \infty} \int_{B_{\eps^\nu}(s^\nu, t^\nu)} \tfrac 1 2| \d\ul v^\nu|^2 \geq \liminf_{\nu \to \infty} \sum_{\ell \in \{0,2\}} \int_{B_{\eps^\nu}(s^\nu, t^\nu)} v_\ell^{\nu\;*}\om_\ell &\geq \liminf_{\nu \to \infty} \sum_{\ell \in \{0,2\}}
\int_{B_{r^\nu}(0)} 
w_\ell^{\nu\;*}\omega_\ell \\
&\geq \sum_{\ell \in \{0,2\}} \int w_\ell^{\infty\;*}\omega_\ell \geq \hbar_{L_{01} \circ L_{12}}.
\end{align*}
\end{itemlist}

\medskip
\noindent
This ends the construction of a nontrivial bubble in the this last case, (D02), and thus finishes the iterative construction of a subsequence and blow-up points so that (0), (2), and (3) hold.
To establish the $\cC^\infty_\loc$-convergence on the complement of the blow-up points claimed in (1) we will apply Theorem~\ref{thm:nonfoldedstripshrink} to quilted domains that make up rectangles in \((-\rho,\rho)^2 \less \{z_1, \ldots, z_N\}\).

Standard elliptic regularity implies that \(v_0^\nu(s, t - f^\nu(s))\) resp.\ \(v_2^\nu(s, t + f^\nu(s))\) converge \(\cC^\infty_\loc\) on the interior of their domains \((-\rho,\rho) \times (-\rho,0] \less Z\) resp.\ \((-\rho,\rho) \times [0, \rho) \less Z\).
To extend this convergence to the boundary and to establish convergence of \(v_1^\nu(s,0)\), fix a point \((\sigma,0)\) in \((-\rho,\rho) \times \{0\} \less Z\), and define three maps by rescaling \(v_\ell^\nu\) for $\ell=0,1,2$ and straightening the seams: 
\begin{align*}
w_\ell^\nu(s,t) \coloneqq v_\ell^\nu \bigl( \sigma + s, \tfrac {f^\nu(s + \sigma)} {f^\nu(\sigma)}t \bigr).
\end{align*}
For \(r > 0\) sufficiently small, these maps form a squiggly strip quilt of size \((f^\nu(\sigma), r)\), which is \((\wt J_0^\nu, \wt J_1^\nu, \wt J_2^\nu, j^\nu)\)-holomorphic for \(\wt J_\ell^\nu\) and \(j^\nu\) the pulled-back almost complex and complex structures 
\begin{align*}
\wt J_\ell^\nu(s,t) \coloneqq J_\ell^\nu( \sigma + s, \tfrac {f^\nu(s + \sigma)} {f^\nu(\sigma)}t), \qquad j^\nu(s,t) \coloneqq \left(\begin{array}{cc}
-\tfrac {(f^\nu)'(s + \sigma)} {f^\nu(\sigma)}t & -\tfrac {f^\nu(s + \sigma)} {f^\nu(\sigma)} \\
\tfrac { (f^\nu)'(s + \sigma)^2 t^2 + f^\nu(\sigma)^2} {f^\nu(\sigma)f^\nu(s + \sigma)} & \tfrac {(f^\nu)'(s + \sigma)} {f^\nu(\sigma)} t
\end{array}\right).
\end{align*}
The obedient shrinking \(f^\nu \Rightarrow 0\) and the Arz\`{e}la--Ascoli theorem guarantee that after passing to a subsequence\footnote{For those choices of \((f^\nu)\) that arise from natural geometric situations -- e.g.\ from the figure eight bubble -- we expect convergence directly, i.e.\ without passing to a subsequence.},
\(f^\nu(s + \sigma) / f^\nu(\sigma)\) converges in \(\cC^\infty_\loc\); therefore \(\wt J^\nu_\ell\) and \(j^\nu\) converge in \(\cC^\infty_\loc\) to almost complex and complex structures \(\wt J_\ell^\infty\) and \(j^\infty\).
As long as \(r\) was chosen to be small enough, the bound \(\|j^\infty - i \|_{\cC^0} \leq \eps\) holds (where \(\eps\) is the constant appearing in Thm~\ref{thm:nonfoldedstripshrink}), so Theorem~\ref{thm:nonfoldedstripshrink} implies that \(w_0^\nu(s, t - f^\nu(\sigma))\), \(w_1^\nu(s, 0)\), \(w_2^\nu(s, t + f^\nu(\sigma))\) converge \(\cC^\infty_\loc\) to smooth maps \(w_0^\infty, w_1^\infty, w_2^\infty\) that satisfy a generalized seam condition in \(L_{01} \times_{M_1} L_{12}\).
Since \(f^\nu( s + \sigma) / f^\nu(\sigma)\) converges \(\cC^\infty_\loc\), we may conclude that \(v_0^\nu(s, t - f^\nu(\sigma))\), \(v_1^\nu(s,0)\), \(v_2^\nu(s, t + f^\nu(\sigma))\) converge \(\cC^\infty_\loc\) on a neighborhood of \((\sigma,0)\), and the limit maps satisfy a generalized seam condition in \(L_{01} \times_{M_1} L_{12}\).
We established convergence away from \((-\rho,\rho) \times \{0\}\) earlier, so we have now proven (1).
This finishes the proof of Theorem~\ref{thm:rescale}.

\begin{remark} \label{rmk:noncompact}
The purpose of this remark is to discuss the minimal assumptions which allow one to apply Theorem~\ref{thm:rescale} to symplectic manifolds $M_0, M_1, M_2$ that are not compact.
If the Lagrangian correspondences are compact, then -- unlike the ``bounded geometry'' assumptions in \cite{isom} -- we do not explicitly require uniform bounds on metrics and almost complex structures (which were used in \cite{isom} to show energy concentration in a sequence of pseudoholomorphic maps with unbounded gradient).
Instead we need to ensure convergence of maps which result from rescaling near a blow-up point of the gradients of a sequence of pseudoholomorphic maps.
If the rescaled domains contain boundary or seam conditions, then compactness of the Lagrangians implies \(\cC^0\)-bounds so that the rest of our arguments applies in a precompact neighborhood of the Lagrangians. 
If the rescaled domains do not contain boundary or seam conditions, or if the Lagrangians in the boundary or seam conditions are noncompact, then \(\cC^0\)-bounds must be obtained a priori from some special properties of the symplectic manifold or Lagrangians.

Note that the \(\cC^0\)-bounds are not merely technical complications -- in general, nontrivial parts of sequences of pseudoholomorphic curves can and will escape to infinity, at best yielding punctures and SFT-type buildings in the limit.
One way to achieve \(\cC^0\)-bounds would be to work with completed Liouville domains and Lagrangians which are cylindrical at infinity, as in Abouzaid--Seidel's definition of the wrapped Fukaya category in \cite{as:wrapped}.

Footnotes \ref{foot:bubblesinstrata} and \ref{foot:bounds} point out the main instances where the specific geometry would have to be considered when dealing with noncompact manifolds.
When working with noncompact Lagrangians, one would have to make additional assumptions -- such as ``bounded geometry'' for symplectic manifolds and Lagrangians -- to guarantee uniformity of the elliptic estimates in \cite{b:singularity}.
\end{remark}

\section{Boundary strata and algebraic consequences of strip-shrinking moduli spaces} \label{s:propaganda}

The purpose of this section is to analyze the expected boundary stratification of strip-shrinking moduli spaces and from this predict the algebraic consequences of figure eight bubbling.
While we make an effort to provide convincing arguments for the more surprising features, this part of our exposition will be rather cavalier -- aiming only to explain the rough form of what we expect to be able to make rigorous.
In particular, all Floer cohomology groups will be considered as ungraded and with coefficients in the Novikov field defined over \(\Z_2\), which should be valid as long as sphere bubbling can be avoided.
We ultimately expect a fully-fledged graded theory with Novikov coefficients defined over \(\Q\) (resp.\ over \(\Z\) in the absence of sphere bubbling).

\subsection{Boundary stratifications and their algebraic consequences} \label{ss:boundary}
One of the intuitions in the treatment of pseudoholomorphic curve moduli spaces is that sphere bubbling is ``codimension 2'' and disk bubbling is ``codimension 1''.
We give a more rigorous statement of this intuition in the polyfold framework and explain its algebraic consequences in Remark~\ref{rmk:codim} below, and will argue that, in a similarly imprecise sense, figure eight bubbling is ``codimension 0'' within the ``zero-width boundary components'' of quilt moduli spaces involving a strip or annulus of varying width.

\begin{remark}[\bf Codimension and algebraic contribution of sphere and disk bubbles] \label{rmk:codim}
In the polyfold setup for pseudoholomorphic curve moduli spaces (whose blueprint is given in \cite{hwz:gw} at the example of Gromov--Witten moduli spaces), the compactified moduli space is cut out of the ambient polyfold by a (polyfold notion of) Fredholm section, which arises from the Cauchy--Riemann operator.
Transversality while preserving compactness can then be achieved by adding a small, compact (possibly multivalued) perturbation, which is supported near the unperturbed moduli space.
This equips the perturbed moduli space with the structure of a compact (possibly weighted branched) manifold.
For expositions of this theory see e.g.\ \cite{Hofer,hwz:fred2,theguide}.)

An important feature of the ambient space is that there is a sensible notion of ``corner index'' -- a nonnegative integer associated to each point in the polyfold, so that the points of corner index 0 resp.\ 1 resp.\ $\ge 2$ should be thought of as the interior resp.\ smooth part of boundary resp.\ corner stratification.
The transverse perturbation can be chosen compatibly with corner index, so that a Fredholm index 0 section gives rise to a perturbed moduli space lying in the interior of the polyfold, and a Fredholm index \(1\) section gives rise to a perturbed moduli space whose boundary is given by the intersection of the zero set with the smooth (corner index \(1\)) part of the polyfold boundary.
The index 0 components of the perturbed moduli space are then typically used to define an algebraic structure, whose algebraic relations arise from the fact that the Fredholm index \(1\) component has nullhomologous boundary corresponding to algebraic compositions of Fredholm index \(0\) contributions.
More precisely, the sum over the algebraic contributions of the boundary points
is zero, and since these boundary points are given by the zero set of the section restricted to the smooth part of the boundary of the polyfold, the algebraic relations are given by a sum over these boundary strata.

It turns out that interior nodes do not contribute to the corner index, and in particular that curves with sphere bubbles and no other nodes are smooth interior points of the polyfold.
This can be understood from the gluing parameters $(R_0,\infty)\times S^1$ used to describe a neighborhood of the node.
The corresponding pre-gluing construction provides a local chart for the polyfold, in which the gluing parameters get completed by $\{\infty\}$ to an open disk, which contributes no boundary. 
On the other hand, gluing at a boundary node or a breaking is described by a parameter in $(R_0,\infty)$, which gets completed by $\{\infty\}$ to a half-open interval, so that pre-gluing in these cases provides local charts in which parameter $\infty$ indicates a contribution of \(+1\) to the corner index.
Hence each boundary node (e.g.\ from disk bubbling), each trajectory breaking (as in Floer theory), and each extra level of buildings (in SFT) contribute \(1\) to the corner index. 
This explains why sphere bubbling does not contribute to algebraic relations of the type discussed here, and instead it is the curves with exactly one boundary node (e.g.\ one disk bubble) or one breaking which contribute to the algebraic relations.
\end{remark}

Following the above remark, we need to analyze the boundary stratification of the polyfolds from which the strip-shrinking moduli spaces are cut out in order to predict the algebraic consequences of figure eight bubbling.
For that purpose we describe in the following the {\bf pre-gluing constructions that provide the local charts near figure eight, squashed eight, and disk bubbles}:

\begin{itemlist}
\item Gluing a {\bf figure eight} into a pseudoholomorphic quilt has to go along with introducing an extra strip of width $\delta>0$.
\item Since figure eights do not have an $S^1$ symmetry, it then remains to fix the length of neck between the bubble and the quilted Floer trajectory.
However, this gluing parameter in $(R_0,\infty)$ is in fact fixed by the choice of width \(\delta > 0\), as illustrated in Figure~\ref{fig:8gluing}.
Hence, while figure eight bubbles can only appear on the \(\delta=0\) boundary, they do not contribute to the corner index.
This means that a \(\delta=0\) quilted Floer trajectory with any number of figure eight bubbles will still just have corner index $1$.
Indeed, the figure eights can only be pre-glued simultaneously since their neck-lengths must all be given by the same strip width $\delta>0$.
\end{itemlist}

\begin{figure}
\centering
\def\svgwidth{\columnwidth}
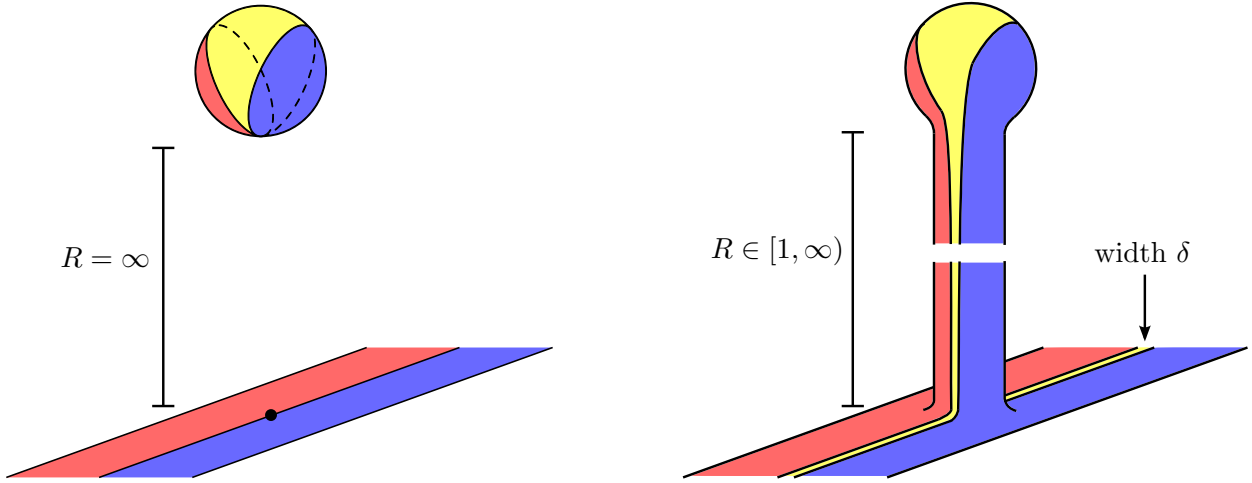
\caption{A figure eight bubble can be glued to a marked point on a double strip indicated in the left figure.
This is done by first fattening the seam in the double strip to a new middle strip of width \(\delta\) and a
centered puncture at the old marked point, then overlaying a neighborhood of this puncture with a neighborhood of the figure eight singularity in cylindrical coordinates, and finally interpolating between the maps on the new domain.
In this construction the neck-length parameter \(R\)  is determined by the strip width \(\delta\), as illustrated in the right figure: 
Since the seams are not straight in cylindrical coordinates, the relative shift between bubble and triple strip is determined by the positions of the seams having to match.  
Note that these figures illustrate the domains of the respective maps, not their images.
\label{fig:8gluing}
}
\end{figure}

\begin{itemlist}
\item The configuration of a {\bf squashed eight bubble} with seam in \(L_{01} \times_{M_1} L_{12}\) attached to a \(\delta=0\) quilted Floer trajectory has corner index $2$ since it can be glued with two independent parameters.
Indeed, the pre-gluing construction is to first widen the seam in both the base and the bubble to strips of independent widths $\delta \geq 0$ and $\eps \geq 0$
(turning the squashed eight into a figure eight in case \(\eps>0\)), and to then pre-glue the resulting bubble into the quilt with a gluing parameter \(R\in (R_0,\infty]\).
Here the strip width \(\delta\) is determined by \((R,\eps)\) as follows:
Pre-gluing with \(R=\infty, \eps>0\) yields a (only approximately holomorphic) figure eight attached to a middle strip of width \(\delta(\infty,\eps)=0\) whereas \(R<\infty, \eps=0\) produces a (approximately holomorphic) quilted Floer trajectory with middle strip width \(\delta(R,0)=0\).
Positive strip width \(\delta(R,\eps)>0\) is achieved only with \(R<\infty, \eps>0\), providing the interior of the chart.

\item
Similarly, a {\bf disk bubble}\footnote{Here we identify quilted spheres with two patches in \(M_k\) and \(M_{k+1}\) with disks in \(M_k^- \times M_{k+1}\);  see footnote~\ref{foot:fold}.
}
with boundary in \(L_{01}\) or \(L_{12}\) attached to a \(\delta=0\) quilted Floer trajectory has corner index 2 since the length of the gluing neck is independent of the width \(\delta\ge 0\).
In fact, in our tree setup there would be a constant figure eight between the disk and Floer trajectory, so that the gluing parameter is used to pre-glue the disk into the figure eight, and the width parameter pre-glues the resulting figure eight into the Floer trajectory.
\end{itemlist}

\begin{remark}[\bf Boundary stratification of strip-shrinking moduli spaces] \label{rmk:wtc}
The gluing construction for squashed eights above indicates that the closures of the two top boundary strata given by \(\delta = 0\) quilts with one figure eight bubble (i.e.\ \(R = \infty, \eps > 0\)) and by \(\delta = 0\) quilts with no bubbles (i.e.\ \(R < \infty, \eps = 0\)) intersect in a corner index 2 stratum consisting of \(\delta=0\) quilts with one squashed eight bubble (i.e.\ \(R = \infty, \eps = 0\)).  
To see how a sequence of \(\delta=0\) quilts with one figure eight bubble can converge to a \(\delta = 0\) quilt with one squashed eight bubble, note that the moduli space of figure eight bubbles has a boundary stratum in which all energy concentrates at the singularity, so that rescaling yields a
squashed eight bubble.

A corollary of this reasoning is that in the moduli space of figure eight bubbles, squashed eights form a codimension-1 boundary stratum.
\end{remark}

\subsection{Strip shrinking in quilted Floer theory for cleanly-immersed geometric composition}  \label{ss:HF}

With this framework in place, we now analyze the boundary stratification of a specific strip-shrinking moduli space, from which we will obtain specific algebraic predictions in Section~\ref{ss:algebra1}, 
The isomorphism between quilted Floer homologies \eqref{eq:HFiso} under monotone, embedded composition is proven in \cite{isom} by applying the cobordism argument in Remark~\ref{rmk:codim} to a moduli space of quilted Floer trajectories with varying width $\delta\in[0,1]$ of the strip mapping to $M_1$,
\begin{align*}
\mathcal{M} := \left\{\underline u = \left( \begin{aligned}
 u_0&\colon\R \times [-1, 0] \to M_0 \\
 u_1&\colon\R \times [0, \delta] \to M_1 \\
 u_2&\colon\R \times [\delta, 1+\delta] \to M_2
 \end{aligned} \right)
\left| 
\begin{aligned}
 \delta\in[0,1], \quad \partial_s u_\ell + & J_\ell(u_\ell) \partial_t u_\ell = 0  \quad \text{for}\; \ell=0,1,2, \\
\quad u_0(\R\times\{-1\})& \subset L_0,  \quad  u_2(\R\times\{1+\delta\})\subset L_2, \\
(u_0,u_1)(\R\times\{0\}) &\subset L_{01},  \;
(u_1,u_2)(\R\times\{\delta\}) \subset L_{12} , \\
\tint u_0^* \om_0 + \tint u_1^* \om_1 & + \tint u_2^* \om_2 <\infty , \quad {\rm ind}\,{\rm D}_{\underline u} = 1
 \end{aligned}
 \right. \right\} \!/ \R . 
\end{align*}
Its boundary arises from the strip widths $\delta=0$ and $\delta=1$, since other bubbling or breaking is excluded by the monotonicity assumption and restriction to Fredholm index $1$. 
Recall however that this bubble exclusion fails even in monotone cases as soon as the geometric composition is a multiple cover of a smooth Lagrangian (as in many examples of interest, e.g.\ \cite{w:chekanov}).
In order to obtain a result that allows for general symplectic manifolds and Lagrangians and cleanly-immersed composition \(L_{01} \circ L_{12}\), we need to study the boundary strata of the polyfold which provides an ambient space for a general compactified moduli space $\ol\M$ of quilted Floer trajectories with varying width.
In addition to breaking and bubbling (of spheres, disks, squashed eights, and figure eights), the ends of the interval $[0,1]$ contribute to its corner index.
Based on the previous analysis of gluing parameters, we predict that the 
{\bf top boundary strata of the Gromov-compactified strip-shrinking moduli space} $\ol\M$ (i.e.\ the strata of $\partial \overline{\mathcal M}$ with corner index $1$) consist of
\label{Bs}
\begin{itemize}
\item[(B1)]
quilted Floer trajectories for $\delta=1$;
\item[(B2)]
once-broken quilted Floer trajectories for $\delta \in (0,1)$;
\item[(B3)]
quilted Floer trajectories with one disk bubble on a seam for $\delta \in (0,1)$;
\item[(B4)]
quilted Floer trajectories for $\delta=0$ with generalized seam condition;\footnote{
Since the Lagrangian \(L_{01} \circ L_{12}\) is in general just a clean immersion, we require not only that the corresponding seam gets mapped to the composed Lagrangian, but we require this map to lift continuously to \(L_{01} \times_{M_1} L_{12}\), and include the lift as data of the Floer trajectory. 
}
\item[(B5)] quilted Floer trajectories for $\delta=0$ with any positive number of figure eight bubbles.\footnote{
In this case, the generalized seam condition requires a lift to \(L_{01} \times_{M_1} L_{12}\) that is continuous (and hence smooth) on the complement of the bubbling points, and at each bubbling point is possibly discontinuous in a way that matches with the limits $\lim_{s\to\pm\infty}w_1(s, \cdot)$ of the figure eight.
}
\end{itemize} 
Within each such boundary component we may also find curves that include trees of sphere bubbles.
Furthermore, contributions from (B2) and (B3) necessarily involve curves that for fixed $\delta$ are not cut out transversely, i.e.\ these contributions come from a finite set of singular values of $\delta\in (0,1)$.

To argue for our prediction, in particular the necessity of allowing figure eight bubbles in (B5), from a more geometric perspective, let us go through the rather silly example of shrinking the strip in standard Floer theory for a pair of Lagrangians $L_{01}\subset \pt \times M_1$ and $L_{12}\subset M_1^- \times \pt$. 
In this case, the boundary component (B1) represents the Floer differential. The boundary components (B2) and (B3) will be empty, since the Cauchy--Riemann operator for each $\delta>0$ is just a rescaling of that for $\delta=1$, and hence all can be made regular simultaneously.
Hence the Floer differential must coincide with the algebraic contributions from (B4--5). Indeed, each Floer trajectory can be viewed as a figure eight bubble by Example~\ref{ex:mickeymouse}, and in this case is attached to the constant Floer trajectory in $\pt\times\pt$.
Broken Floer trajectories are excluded for index reasons.
More evidence for the necessity of (B5) are the Floer homology calculations in \cite{w:chekanov} between Clifford tori and $\RP^n\subset \CP^n$ resp.\ the Chekanov torus in $S^2\times S^2$ using strip shrinking for multiply covered geometric composition, where bubbling can only be excluded for classes of Floer trajectories whose limits are not self-connecting.
Nonvanishing results for the corresponding entries of the differential from other calculation methods then indirectly show nontrivial figure eight contributions.

\medskip 

More generally, the Gromov-compactification $\ol\M$ of a moduli space of pseudoholomorphic quilts involving strip-shrinking, such as $\M$ above, arises by including all combinations of

\begin{itemize}
\item 
breaking at striplike ends (each adding 1 to the corner index), 
\item
sphere bubbles attached at interior points (not contributing to the corner index),
\item
disk bubbles attached at boundaries and seams (each adding 1 to the corner index),
\item
squashed eight bubbles developing at the seam resulting from strip-shrinking (each adding 1 to the corner index, in addition to the 1 corner index arising from stripwidth $0$),
\item
figure eight bubbles developing at the seam resulting from strip-shrinking (not contributing to the corner index, apart for the 1 corner index arising from stripwidth $0$). 
\end{itemize}

Combinations of these effects yield -- as described in e.g.\ \cite{ms:jh} and \cite{frauenfelder} -- trees of sphere bubbles attached at interior points, and trees of disk and sphere bubbles attached at boundaries and seams. At the new seam arising from strip-shrinking, we moreover have to include squashed eights and figure eights into the bubble trees. The resulting Gromov compactification is summarized in Remark~\ref{rmk:8bubbletree}. Since sphere bubbles do not contribute to the boundary stratification and thus the algebra, we will not include them in the following. 
Then a bubble tree involving squashed eights, figure eights, and disk bubbles can be described in terms of the following combinatorial structure, as illustrated in Figure~\ref{fig:tree-pair}:

\begin{itemize}
\item 
Every component $C$ of the bubble tree is represented by a vertex $v_C$, and the main component of the bubble tree corresponds to the root.
These vertices form the set of ``component vertices'' $V_{\compnt}=V_{\compnt,1}\cup V_{\compnt,2}$, which decomposes into the set $V_{\compnt,1}$ of vertices (representing squashed eights and disk bubbles) whose components have one seam, and the $V_{\compnt,2}$ of vertices (representing figure eights) whose components have two seams. 

\item 
Every seam $S$ on a component $C$ is represented by a vertex $v_S$, which is connected to $v_C$ by a solid edge. These vertices form the set of ``seam vertices'' $V_\seam$. 

\item
If a component $C'$ is a bubble which is attached to a seam $S$ on another component $C$, then $v_{C'}$ is connected to $v_S$ by a dashed edge.

\item 
Every marked point $P$ on a seam $S$ is represented by a leaf vertex $v_P$ and is connected to $v_S$ by a dashed edge.  These vertices form the set of ``marked point vertices'' $V_{\text{mark}}$. 

\item
The resulting tree structure forms an $r=2$ tree-pair as in Definition~\ref{def:tree-pair}. 
\end{itemize}

\begin{figure}
\centering
\def\svgwidth{0.6\columnwidth}
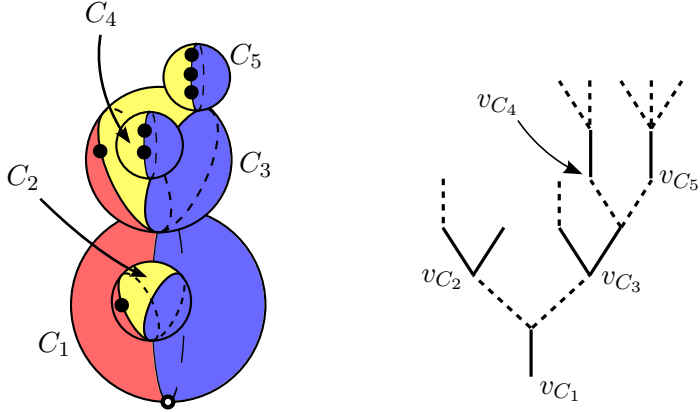
\caption{A figure eight bubble tree and the corresponding $r=2$ tree-pair.
\label{fig:tree-pair}}
\end{figure}

This combinatorial structure, called $r=2$ tree-pair, is developed and discussed for general $r\geq 1$ (indicating the number of seams on a generalized figure eight) in \cite{b:2ass}. Since the combinatorics is quite involved, we only define the $r=2$ case that is relevant to figure eights (though it does not serve to explain the origin of the term ``tree-pair'').
We also define a ``metric'' version of this notion, which we will use later to introduce Morse trajectories into our moduli spaces.

\begin{definition}
\label{def:tree-pair}
An {\bf{$\mathbf{r=2}$ tree-pair}} is a rooted ribbon tree $T$ (oriented toward the root) whose edges are either solid or dashed, which satisfies the following properties:
\begin{itemize}
\item The vertices of $T$ can be partitioned as $V(T) = V_{\compnt,1} \sqcup V_{\compnt,2} \sqcup V_\seam \sqcup V_{\text{mark}}$, where:
\begin{itemize}
\item every $\alpha \in V_{\compnt,k}$ has $k$ solid incoming edges, no dashed incoming edges, and either a dashed or -- when $\alpha$ is the root -- no outgoing edge;
\item every $\alpha \in V_\seam$ has $\geq 0$ dashed incoming edges, no solid incoming edges, and a solid outgoing edge;
\item every $\alpha \in V_{\text{mark}}$ is a leaf (i.e.\ has 0 incoming edges) with a dashed outgoing edge.
\end{itemize}
\item
The path from a leaf $\alpha\in V_\seam$ of $T$ to the root contains at most one element of $V_{\compnt,2}$.
For a leaf $\alpha \in V_{\text{mark}}$, the path to the root contains exactly one element of $V_{\compnt,2}$.

\end{itemize}
An {\bf $\mathbf{r=2}$ metric tree-pair} is an $r=2$ tree-pair together with a length function $\lambda\colon E^\node \to [0,\infty]$ such that:
\begin{itemize}
\item the sum
$h$ of lengths of edges between a vertex in $V_{\compnt,2}$ and the root vertex is the same for all vertices in $V_{\compnt,2}$; and

\item 
if $V_{\compnt,2}$ is nonempty, and if $v \in V_{\compnt,1}$ can be connected to the root by a path that does not pass through any element of $V_{\compnt,2}$, then the sum of lengths of edges between $v$ and the root is $\leq h$.
\end{itemize}
Here the {\bf nodal edges} $E^\node$ are those dashed edges connecting a vertex in $V_\seam$ to a vertex in $V_{\compnt,1}\sqcup V_{\compnt,2}$.
The {\bf figure eight height $\mathbf h$} of such a metric tree-pair is the sum of lengths of edges between each vertex in $V_{\compnt,2}$ and the root, or in case $V_{\compnt,2}=\emptyset$ it is the maximal sum of lengths of edges between a vertex in $V_{\compnt,1}$ and the root.\footnote{Equivalently, the figure eight height of an $r=2$ metric tree-pair is the maximal length of a chain of edges attached to the root which does not cross a $V_{\compnt,2}$ vertex.}
\end{definition}

\begin{remark}
\label{rmk:gen_tree_pairs}
More general tree-pairs of type $(n_1,\ldots,n_r)$ arise from a compactification of the space of generalized marked figure eight domains with $r\geq 1$ seams and $n_k\geq 0$ marked points on each seam.
The first author introduced tree-pairs in \cite{b:2ass} and showed that they give rise to abstract polytopes called 2-associahedra.
He constructed topological realizations of the 2-associahedra in \cite{b:realize} in terms of witch balls -- a notion of quilted marked surfaces that generalizes the domain of figure eight bubbles by allowing for further seams and marked points.
He will moreover use tree-pairs in \cite{b:2cat} to motivate the definition of a notion of $A_\infty$ 2-categories (see Remark~\ref{rmk:2cat}).
This will generalize the case $r=1$ in which moduli spaces of marked points on the boundary of a disk (viewed as a sphere with equatorial seam) realize the classical Stasheff associahedra, which encode the algebraic structures and relations in $A_\infty$ categories.

\begin{figure}
\centering
\def\svgwidth{0.8\columnwidth}
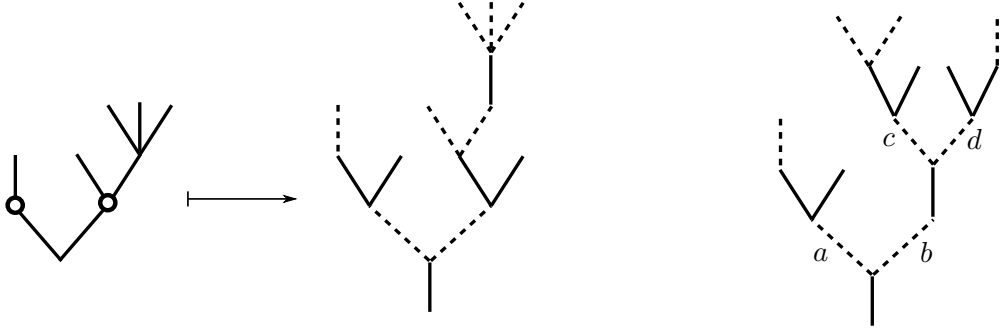
\caption{On the left is a colored rooted ribbon tree and the corresponding $r=2$ tree-pair.
On the right is an $r=2$ metric tree-pair, with nodal edges labeled by edge-length; the requirement that each vertex in $V_{\compnt,2}$ be the same distance from the root translates into the equations $a=b+c=b+d$.
\label{fig:colored_tree}}
\end{figure}

For example, the Stasheff multiplihedra, which encode $A_\infty$ functors, are shown in \cite{mw} to be realized by colored rooted ribbon trees, which in turn are used in \cite{mww} to associate $A_\infty$ functors to Lagrangian correspondences.
To see that this should become a natural (small) part of a symplectic $A_\infty$ 2-category, note that colored rooted ribbon trees can be identified with the $r=2$ tree-pairs of type $(n,0)$, corresponding to figure eight bubble trees in which marked points occur only on the 01-seam.
Indeed, given any $r=2$ tree pair we obtain a colored rooted ribbon tree by marking the vertices in $V_{\compnt,2}$ as colored and then collapsing all solid edges. 
Conversely, given a colored rooted ribbon tree, one obtains a specific type of tree-pair as illustrated in Figure~\ref{fig:colored_tree} by (1) converting edges to dashed edges; (2) converting noncolored leaves to vertices in $V_{\text{mark}}$; (3) replacing every noncolored interior vertex by a solid edge from a new vertex in $V_\seam$ to a new vertex in $V_{\compnt,1}$; (4) replacing every colored vertex $v$ by two solid edges from new vertices in $V_\seam$ to a new vertex in $V_{\compnt,2}$, and attaching the incoming subtree of $v$ to the left\footnote{The alternative convention of attaching the incoming subtree to the right ``seam vertex'' yields an analogous specialization of $r=2$ tree-pairs of type $(0,n)$, corresponding to bubble trees with marked points only on the 12-seam.} vertex in $V_\seam$.
\end{remark}

At this point, a reader eager to see algebraic consequences could skip to Section~\ref{ss:algebra1}, under the assumption that evaluation maps from the strip-shrinking moduli spaces satisfy sufficient transversality conditions to allow for pullbacks and pushforwards on spaces of chains on the Lagrangian correspondences. 
However, to obtain a well defined theory in general settings, we propose to construct these algebraic structures from the following enlarged moduli spaces that will simplify both our analytic and algebraic work, and also serve to further solidify our prediction of boundary stratifications.

\subsection{Morse bubble trees arising from strip-shrinking moduli spaces}
\label{ss:Morse}
To capture the algebraic effect of figure eight bubbling in terms of Morse chains on $L_{01}\times_{M_1} L_{12}$, we will extend strip-shrinking moduli spaces such as $\ol\M$ discussed in \S\ref{ss:HF} by allowing Morse flow lines between the quilted strip and figure eight, squashed eight, and disk bubbles.

For that purpose we start from the description of bubble trees appearing in any codimension, given in \S\ref{ss:HF} in terms of $r=2$ tree-pairs.
Apart from this more complicated combinatorial tree structure, the extension of strip-shrinking moduli spaces by allowing Morse flow lines is directly analogous to constructing the $A_\infty$-algebra of a single Lagrangian submanifold via trees of disk vertices and Morse edges in \cite{lw:morsetrees}, which we summarize here.\footnote{
This approach has been a partially realized vision in the field for a while.
Formally, it follows the $A_\infty$ perturbation lemma \cite[Prop.\ 1.12]{se:bo} for transferring an $A_\infty$-structure on a space of differential chains to the space of Morse chains.
A related moduli space setup was proposed in \cite{CL} but with a different algebraic goal.}

\begin{remark}[\bf Polyfold setup and boundary stratification for trees of disks with Morse edges]  \label{rmk:disks}
Consider a single disk bubble attached by a marked point to e.g.\ a Floer trajectory.
If we enlarge the moduli space by Morse flow lines between nodal pairs of boundary marked points, then boundary strata with length $0$ flow lines cancel strata with boundary nodes.
A point in this extended compactified moduli space consists of a metric rooted ribbon tree, together with a pseudoholomorphic disk (modulo appropriate reparametrizations) for every vertex and a length-$\ell$ generalized Morse trajectory (including broken trajectories that compactify the space of finite length Morse trajectories) for every length-$\ell$ edge.

In \cite{lw:morsetrees}, assuming the absence of sphere bubbling, this moduli space is described as the zero set of a Fredholm section in an M-polyfold bundle.
This section is given by the Cauchy--Riemann operators on each vertex together with the matching conditions for each edge between the endpoints of the Morse trajectory and corresponding marked point evaluations of the disk maps.
The ambient M-polyfold is the space of trees in which vertices are represented by (reparametrization equivalence classes of) not necessarily pseudoholomorphic maps and edges are represented by generalized Morse trajectories.
Nodal configurations with edge length $0$ are interior points of this space by pre-gluing the nodal disks at a single vertex (this is made rigorous in terms of M-polyfold charts arising from the pre-gluing construction).
Hence the boundary stratification of this space is induced by the compactified space of Morse trajectories -- which was given a smooth structure in \cite{w:morse}, with corner index equal to the number of critical points at which a trajectory breaks.

Now the arguments of Remark~\ref{rmk:codim} yield an algebraic structure from counting isolated solutions whose relations are given by summing over the top boundary strata (those with corner index 1). In this case, adding incoming Morse edges from input critical points yields a curved $A_\infty$-algebra because the top boundary strata -- configurations with exactly one broken trajectory, i.e.\ edge of length $\infty$ -- correspond to the top boundary strata of a space of metric rooted ribbon trees. 
The stable trees in the latter realize Stasheff's associahedra, so that the boundary strata yield the $A_\infty$-relations with the exception of terms involving $\mu^1$ or $\mu^0$.
These additional terms arise from breaking into two subtrees of which one is unstable with zero or one incoming Morse edge. 
Similarly, considering Floer trajectories for pairs of Lagrangians (or quilted Floer trajectories) with several bubble trees (on each boundary component resp.\ seam) yields the relations for the Floer differential coupled with the $A_\infty$-algebras of the Lagrangians.

Sphere bubbling can be included here by extending the ambient space of disk maps and Morse edges to allow for trees of spheres attached to the maps.
This introduces the additional complication of isotropy, turning the ambient space into a polyfold (M-polyfolds are a special case with trivial isotropy) and forcing the use of multivalued perturbations, thus yielding rational counts.
However, as discussed in Remark~\ref{rmk:codim}, this does not affect the boundary stratification and algebraic consequences.
\end{remark}

Next, we apply the same philosophy of enlarging moduli spaces by incorporating Morse edges to the figure eight bubble trees introduced at the beginning of this section. 

\begin{definition}
A {\bf Morsified figure eight bubble tree} for Lagrangian correspondences $L_{01}\subset M_0^-\times M_1, L_{12}\subset M_1^-\times M_2$ and fixed almost complex structures consists of the following data:
\begin{itemize}
\item an $r=2$ metric tree-pair $T$;

\item quilted pseudoholomophic spheres $\underline{w}_v$ for every component vertex $v \in V_{\compnt,1}\cup V_{\compnt,2}$:
\begin{itemize}
\item if $v \in V_{\compnt,2}$, then $\underline{w}_v$ is a figure eight bubble between \(L_{01}\) and \(L_{12}\) as in Definition~\ref{def:8};
\item if the path from $v \in V_{\compnt,1}$ to the root runs through a ``figure eight vertex'' $v'\in V_{\compnt,2}$, then $\underline{w}_v$ is a disk bubble (which can be viewed as quilted sphere as in Remark~\ref{rmk:stereographic}) with boundary on $L_{01}$ resp.\ $L_{12}$ according to whether the path from $v$ to $v'$ passes through the left resp.\ the right incoming edge of $v'$;
\item if $v\in V_{\compnt,1}$ is connected to the root without crossing a figure eight vertex in $V_{\compnt,2}$, then $\underline{w}_v$ is a squashed eight bubble with seam in \(L_{01} \times_{M_1} L_{12}\) as in Definition~\ref{def:8};
\end{itemize}

\item marked points $z^\pm_e$ for every dashed edge $e$ from a vertex $v\in V_{\compnt,1}\cup V_{\compnt,2}\cup V_\marked$ to a vertex
$v_S\in V_\seam$ as follows:
\begin{itemize}
\item
$z^+_e\in S$ is an incoming marked point on the domain of the quilt $\ul w_{v'}$
associated to the vertex $v' \in V_{\compnt,1}\cup V_{\compnt,2}$ to which $v_S$ is connected by a solid edge, and it lies on the seam $S\subset S^2$ corresponding to $v_S$;
\item
if $v\in V_{\compnt,1}\cup V_{\compnt,2}$ then $z^-_e\in S^2$ is the ``outgoing marked point'' on the domain of $\underline{w}_v$, which for a disk bubble lies on the boundary (seam) and for a figure eight or squashed eight bubble lies at the singularity of the quilted domain;
\item
for $v\in V_\marked$ there is no marked point $z^-_e$;
\end{itemize}

\item generalized Morse trajectories $\underline{\gamma}_e$ of length $\lambda(e)$ for every dashed edge $e$ from a vertex $v\in V_{\compnt,1}\cup V_{\compnt,2}\cup V_\marked$ to a vertex in $v_S\in V_\seam$ as follows:
\begin{itemize}
\item
$\underline{\gamma}_e$ is a trajectory on the Lagrangian by which the seam $S$ is labeled, 
and ends at $\underline{w}_{v'}(z^+_e)$ where $v'\in V_{\compnt,1}\cup V_{\compnt,2}$ is the vertex connected to $v_S$ by a solid edge; 
\item
if $v\in V_{\compnt,1}\cup V_{\compnt,2}$, then $\underline{\gamma}_e$ starts at $\underline{w}_v(z^-_e)$;
\item
if $v\in V_\marked$, then $\underline{\gamma}_e$ starts at a Morse critical point.
\end{itemize}
\end{itemize}
In addition, the component corresponding to the root of $T$ is equipped with an ``output'' marked point -- for a figure eight or squashed eight bubble this is the singularity of the quilted domain; for a disk bubble it can be any point on the boundary. 
All these marked points are required to be disjoint. 
Finally, we impose the {\bf stability condition} that constant squashed eights and disk bubbles must have at least 3 marked points (with interior marked points counting for 2, in situations where these are incorporated), while constant figure eight bubbles must have at least 2 marked points (including the singularity where the two seams intersect). 
\end{definition}

With this background, we can now introduce Morse flow lines into strip-shrinking moduli spaces. 
We will illustrate this in the case of the compactified moduli space of triple strips $\ol\M$ from \S\ref{ss:HF}, which is a representative example.
For $\delta \in [0,1]$, we will denote by $\ol\M_\delta$ those elements of $\ol\M$ whose middle strip has width $\delta$; hence $\ol\M = \bigcup_{\delta\in[0,1]} \ol\M_\delta$.
Recall that $\ol\M$ has five types of codimension-1 boundary strata, (B1--5), where
(B3) and (B5) are the two types involving bubbles: (B3) consists of trajectories with one disk bubble on a seam for $\delta \in (0,1)$, while (B5) consists of $\delta=0$ trajectories with any positive number of figure eight bubbles attached to the middle seam.
Our goal is to enlarge $\ol\M$ to another compact moduli space $\ol\M^\ext$ by attaching codimension-0 strata at the boundary strata of fiber product type (B3), (B5). We will do so by introducing Morse flow lines, and thus obtain new boundary strata (B3'), (B3''), (B5') from Morse breaking.
However, the other boundary strata (B1), (B2), (B4) will also be modified in the process, as will be specified in (B1'--5') below. 
More precisely, we construct the extended moduli space $\ol\M^\ext:=\bigcup_{\delta \in [-\infty,1]} \ol\M^\ext_\delta$ as follows:

\begin{itemize}
\item 
An element of the extended moduli space $\ol\M^\ext_\delta$ for $\delta \in (0,1]$ consists of a triple strip $\ul u \in \ol\M_\delta$ together with a finite number of trees of disk bubbles with (possibly infinite) Morse edges between them, attached -- via Morse edges and marked points -- to the seams and boundary components of $\ul u$.
In addition, trees of sphere bubbles can be attached at finitely many interior points of these domains.
The development of a sphere tree at a boundary or seam would be described by attaching it -- via interior marked points -- to a constant disk bubble.
This yields the stability requirement that disk components are constant only if they have at least 3 marked boundary points -- with interior marked points counting for 2 marked boundary points. 
As in Remark~\ref{rmk:disks} this captures disk bubbling on seams/boundaries in an algebraic coupling with the curved $A_\infty$-algebras generated by Morse chains on the Lagrangians. 

\item
An element of $\ol\M^\ext_\delta$ for $\delta \in [-\infty,0]$ consists of a double strip $\ul u \in \ol M_0$ together with Morsified figure eight bubble trees attached via Morse edges to the middle seam, such that for each one of these bubble trees, the figure eight height and the length of the Morse edge attaching the tree to the seam sum to $-\delta$.
Moreover, trees of sphere bubbles can be attached at finitely many interior point of these domains.
The development of a sphere tree at a boundary component or the middle sseam would be described by attaching it to a constant disk, squashed eight, or figure eight bubble.
\end{itemize}

Later on we further generalize these moduli spaces to obtain algebraic coupling with the Morse chain complexes on \(L_{01}\) and \(L_{12}\).
For that purpose we allow leaf vertices $v\in V_\text{mark}$ that are represented by Morse critical points $x^-_v$, with the associated edges represented by generalized Morse trajectories in the compactifications $\overline{\M}(x_v^-, L_{ij})$ of the unstable manifolds of $x_v^-$.

\begin{remark}[\bf Motivation and boundary strata of extended moduli space]  
The construction of the extended moduli space $\ol\M^\ext$ is guided by the goal of removing boundary strata involving nodes (connecting bubbles to the main strip or to each other).
Since a zero width strip with any number of figure eight bubbles is a corner index $1$ boundary point, we need to remove such boundary strata by extending with a single normal parameter, which is the reason for requiring the Morse length between any figure eight and the double strip to have length $-\delta$.
This extends the strip width parameter from $\delta\in[0,1]$ to \(\delta<0\).
Similarly, we cancel the boundary components involving squashed eights forming on the middle seam by introducing Morse flow lines on $L_{01}\times_{M_1} L_{12}$, whose length can now vary individually.
However, since squashed eights also lie on the boundary of moduli spaces of figure eight bubbles (when all energy concentrates at the seam tangency), the length of these Morse edges has to be bounded by the figure eight height.

From here we can note that any breaking of trajectories between squashed eights or figure eights and the main component corresponds to Morse height $-\delta\to\infty$ and thus forces at least one breaking between each figure eight vertex and the root.
If there is just one breaking for each figure eight, then the result lies in the top boundary stratum, but any additional breaking of an edge adds to the corner index individually.
Indeed, if there are $N$ figure eight vertices, then this effect is analogous to the breaking of finite length Morse trajectories in an $N$-fold Cartesian product: 
The first breaking has to happen simultaneously in all components since their lengths are coupled; further breakings are independent since trajectories can be constant in various components.
\end{remark}

As in Remark~\ref{rmk:disks}, we expect to obtain an ambient polyfold (or M-polyfold if sphere bubbling can be a priori excluded) for the compactified extended moduli space $\ol\M^\ext$ by replacing the pseudoholomorphic curves and quilts with appropriate spaces of not necessarily pseudoholomorphic maps modulo reparametrization.
Making this rigorous will require a precise setup of pre-gluing constructions for squashed eights and figure eights as M-polyfold charts from \cite{b:thesis}, which will also make the predicted boundary stratification rigorous.
Further steps in the program are the Fredholm property of the Cauchy--Riemann operator, and formal setups for the construction of gluing-coherent perturbations and orientations.
However, we can already see that the boundary stratification of this polyfold is -- apart from strata of types (B1), (B2), and (B4)
on page~\pageref{Bs} -- induced by the boundary structure of the compactified Morse trajectory spaces from \cite{w:morse}.
These boundary components correspond to the boundary strata of the space of metric tree-pairs.
So we expect to obtain algebraic relations from the {\bf top boundary strata of the compactified strip-shrinking moduli space with Morsified figure eight bubble trees
} $\overline{\mathcal M}^{\rm ext}$, which replace our previous list (B1--5) as follows (ignoring sphere bubbling, which is of codimension 2 in each boundary stratum):

\begin{figure}
\centering
\def\svgwidth{\columnwidth}
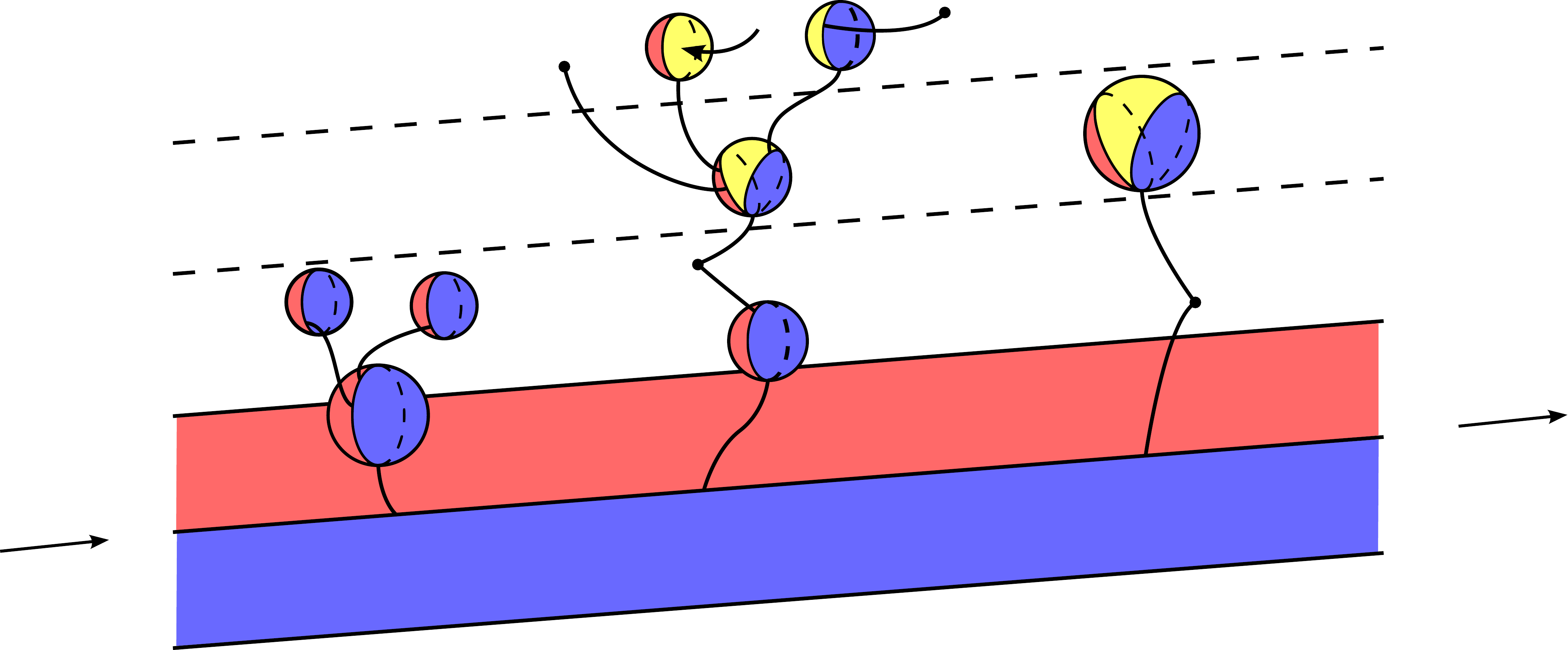
\caption{A contribution to the differential on \(CF( L_0, (L_{01}\circ L_{12}, \sum_{k,l\geq0} b_{02}^{k|l}), L_2 )\).
The two subtrees above the Morse critical points contribute to \(b_{02}^{1|1}\) and \(b_{02}^{0|0}\), respectively.
The dashed lines indicate the level structure of the metric tree-pair.
Additional marked leaf vertices are labeled with Morse chains $b_{01}$ or $b_{12}$, indicating a formal sum of trees whose edges from leaves are represented by half-infinite Morse trajectories starting at the Morse critical points that represent the chains.
\label{fig:prediction}
}
\end{figure}

\label{BBs}
\begin{enumerate}
\item[(B1')]
quilted Floer trajectories for $\delta=1$ may now also include trees of disk bubbles with finite-length Morse edges;
\item[(B2')]
once-broken quilted Floer trajectories for $\delta \in (0,1)$ may now also include trees of disk bubbles with finite-length Morse edges;
\item[(B2'')]
once-broken quilted Floer trajectories can also appear for $\delta \in (-\infty,0)$, and consist of Floer trajectories of width $0$ with a positive number of Morsified figure eight bubble trees of figure eight height $-\delta$, and possibly further trees of disk bubbles with finite-length Morse edges;  
%
\item[(B3')]
quilted Floer trajectories for $\delta \in (0,1)$ with one disk bubble on a seam or boundary are canceled as boundary, but new boundary components contain quilted Floer trajectories of width $\delta \in (0,1)$ with trees of disk bubbles, in which one Morse edge is broken once;
\item[(B3'')]
quilted Floer trajectories of width $0$ with a single broken Morse edge can also appear for $\delta \in (-\infty,0)$, i.e.\ in the presence of one or several Morsified figure eight bubble trees of figure eight height $-\delta$, with the broken edge occurring either above the figure eight height or in a disk bubble tree that is attached to a boundary component;
\item[(B4')]
quilted Floer trajectories for $\delta=0$ may now include trees of disk bubbles with finite length Morse edges attached to the boundary components; (however,when the width goes to zero in the presence of a tree of disk bubbles on $L_{01}$ or $L_{12}$, this is viewed as constant figure eight to which this tree is attached, and is canceled like other strata of type (B5));
\item[(B5')]
quilted Floer trajectories for $\delta=0$ with figure eight bubbles are canceled as boundary components, but new boundary components contain quilted Floer trajectories of width $0$ with one or several Morsified figure eight bubble trees of figure eight height $-\delta=\infty$ attached to the middle seam; that is between each figure eight -- of which there is at least one, otherwise this component is of codimension $\geq 2$ as a consequence of Remark~\ref{rmk:wtc}
-- and main quilt there is exactly one once-broken Morse trajectory, and all other edges have finite length.
\end{enumerate}

An example of a boundary point of type (B5') is given in Figure~\ref{fig:prediction}, where the dashed lines indicate the level structure of the metric tree-pair resulting from strip shrinking:
In the first level above the root quilt, all vertices are represented by squashed eights, whereas in the third level the vertices are represented by disk bubbles with boundary on $L_{01}$ or $L_{12}$.
The second level provides the division since between each marked point
and the root there is exactly one vertex represented by a figure eight.
The figure eight height of the pictured tree is $\infty$, reflected by at least one broken trajectory below each figure eight.
The corner index is $1$ since there is exactly one breaking for each figure eight.

\begin{remark}[\bf Analytic description of extended moduli space]  \label{rmk:compactness-for-Morse}
Introducing Morse trajectories into the strip-shrinking moduli spaces will not complicate the compactness analysis.
Since the new moduli spaces can be regarded as a fiber product of spaces of finite/semi-infinite/infinite Morse trajectories (whose structure have been thoroughly worked out, e.g.\ in \cite{w:morse}) with moduli spaces of pseudoholomorphic quilts (with Morse trajectories replaced by marked points), Gromov compactness is no more difficult than Gromov compactness for the strip-shrinking moduli spaces without Morse trajectories (see \S\ref{ss:overview}).
Furthermore, this fiber product structure immediately suggests how to define the ambient polyfold of maps, from which these moduli spaces will be cut out by a polyfold Fredholm section.
\end{remark}

\subsection{Floer homology isomorphism for general cleanly-immersed geometric composition} \label{ss:algebra1}
To generalize the isomorphism between quilted Floer homologies \eqref{eq:HFiso} under monotone, embedded composition to general symplectic manifolds and Lagrangians and cleanly-immersed composition \(L_{01} \circ L_{12}\), we analyzed in Section~\ref{ss:boundary} the boundary strata of the polyfold which provides an ambient space for a general compactified moduli space of quilted Floer trajectories with varying width, such as $\overline{\mathcal M}$ in \S\ref{ss:HF}.
The cobordism argument outlined in Remark~\ref{rmk:codim} then predicts an algebraic identity from summing over the boundary strata (B1--5) of $\overline{\mathcal M}$ on page~\pageref{Bs} resp.\ the refined boundary strata (B1'--B5') on page~\pageref{BBs} of the compactified extended moduli space $\overline{\mathcal M}^{\rm ext}$.

We expect the strata of types (B2), (B3) resp.\ (B2'), (B2''), (B3'), (B3'') to appear only at finitely many singular values of strip width $\delta\in (0,1)$ or figure eight height $\delta\in (-\infty,0)$ and to provide a chain homotopy equivalence between two Floer complexes:
The first is in both frameworks defined from the regular strip width $\delta = 1$, with the differential given by counting solutions of type (B1) resp.\ (B1'). In the Morse framework, the second Floer complex is defined from counting regular solutions of types (B4') and (B5').
In the framework of page~\pageref{Bs}, the second complex should arise from solutions of types (B4) and (B5) at $\delta=0$. 
Up to such a chain homotopy equivalence, or assuming there are no singular values in $(-\infty,1)$, we obtain the following identity relating the Floer differential \(\mu^1_{(\delta = 1)}\) arising from strip width $\delta=1$ and the Floer differential \(\mu^{1|0}_{(\delta=0)}\) arising from strip width $\delta=0$ with generalized seam condition in \(L_{01} \times_{M_1} L_{12}\):\footnote{
Note that the moduli spaces with generalized seam condition involve a choice of lift of the seam values to $L_{01}\times_{M_1}L_{12}$. So in case $L_{01}\circ L_{12}$ is a smooth Lagrangian correspondence albeit multiply covered by $L_{01}\times_{M_1}L_{12}$, this Floer complex is generated by lifts of intersection points.
The differential  \(\mu^{1|0}_{(\delta=0)}\) only counts Floer trajectories with smooth seam lift, whereas the terms \(\mu^{1|k}_{(\delta=0)}\) for $k\ge1$ will allow for jumps in the seam lift at $k$ marked points.
}
\begin{align} \label{eq:mu1relation}
\mu^1_{(\delta = 1)}(-) \; = \; \textstyle \sum_{k\ge 0}  \mu^{1 | k}_{(\delta = 0)} (\, - \, |\,  b_{02},\ldots, b_{02}).
\end{align}
In the Morse framework, the moduli spaces defining the differentials \(\mu^1_{(\delta = 1)}\) and \(\mu^{1|0}_{(\delta=0)}\) both allow for trees of disks with finite Morse edges (including trees of squashed eights attached to the seam obtained from strip shrinking).
Figure eight bubbling is in both frameworks encoded in the higher operations \(\mu^{1|k}_{(\delta=0)}\) for \(k \geq 1\). 
In the framework of (B5), this operation should be defined from quilted Floer trajectories of middle strip width $0$ with $k$ incoming marked points on the seam labeled by the immersion $L_{01}\circ L_{12}$, and \(b_{02}\) should be a chain obtained from a moduli space of figure eight bubbles by evaluation at the singularity.
In order to avoid the challenge of arranging transversality of evaluation maps to an infinitely-generated (co)chain complex, we will construct the operations
$$
\mu^{1|k}_{(\delta=0)} : \;
CF(L_0, L_{01}\circ L_{12}, L_2) \otimes CM(L_{01}\times_{M_1} L_{12})^{\otimes k} \;\longrightarrow\; CF(L_0, L_{01}\circ L_{12}, L_2)
$$
on (finitely-generated) Morse complexes.
These operations will be defined by counting quilted Floer trajectories of strip width $0$ with \(k\) incoming Morse edges (marked point leaves whose edges are represented by half-infinite trajectories in $L_{01}\times_{M_1} L_{12}$ starting at a Morse critical point) attached (possibly via trees of squashed eights and finite Morse edges) to the middle seam.
Then (B5') indicates that the Morse chain
\begin{equation}\label{b02}
b_{02} \in CF(L_{01}\circ L_{12}, L_{01}\circ L_{12}) \coloneqq CM(L_{01}\times_{M_1} L_{12})
\end{equation}
should be defined by counting regular isolated Morsified figure eight bubble trees
as in \S\ref{ss:Morse}.

Once such operations are defined, \eqref{eq:mu1relation} identifies (up to chain homotopy equivalence) the quilted Floer chain complexes
\begin{align*}
 CF\bigl(L_0, L_{01}, L_{12}, L_2 \bigr)  
\;\simeq\;
CF\bigl(L_0, (L_{01} \circ L_{12}, b_{02}), L_2 \bigr),
\end{align*}
where the differential on the left hand side is $\mu^1_{(\delta = 1)}$ and the differential on the right hand side is the twisted differential
$\partial_{b_{02}}\coloneqq \sum_{k\ge 0}  \mu^{1 | k}_{(\delta = 0)} (\, - \, |\,  b_{02},\ldots, b_{02})$.
Here the right hand side treats $L_{01} \circ L_{12}$ as an immersion.
If this is an embedding then the right hand side is the Floer chain complex of the Lagrangian $L_{01} \circ L_{12}\subset M_0^-\times M_2$ twisted by the Morse chain $b_{02}$.
An example of a contribution to the twisted Floer differential is Figure~\ref{fig:prediction} without the middle tree.
This result is meaningful if the cyclic Lagrangian correspondence \(L_0, L_{01}, L_{12}, L_2\) is naturally unobstructed in the sense that the differential $\partial =\mu^1_{(\delta = 1)}$ satisfies $\partial^2= 0$.
In particular, it asserts that the twisted differential on the right hand side satisfies \(\partial_{b_{02}}^2 = 0\).
To understand more intrinsically why the twisted differential squares to zero, we need to go into the $A_\infty$ algebra.

\medskip
\noindent
{\bf Remark on $\mathbf{A_\infty}$ 
terminology:} {\it  In the upcoming sections, we will denote by $\bigl(\mu^d_{01}\bigr)_{d\ge 0}$, $\bigl(\mu^d_{12}\bigr)_{d\ge 0}$, resp.\ $\bigl(\mu^d_{02}\bigr)_{d\ge 0}$ the curved $A_\infty$-algebras associated to $L_{01}$, $L_{12}$, resp.\  $L_{01}\circ L_{12}$, constructed on Morse chain complexes as outlined in Remark~\ref{rmk:disks}.
If working with the latter, we will usually assume that $L_{01}\circ L_{12}$ is embedded, though there are extensions to multiply covered and even cleanly immersed cases, as outlined in Remark~\ref{rem:immfuk}.

Moreover, we will call \(b \in CF(L,L)\) a \emph{bounding cochain for the Lagrangian \(L\)} if it satisfies the Maurer--Cartan equation \(\sum_{d\geq 0} \mu^d(b,\ldots,b) = 0\).
When a quilted Floer differential is twisted by bounding cochains for each Lagrangian correspondence, it will square to zero.
It is however also possible that twisting with more general cochains yields a chain complex.
}
\medskip

\begin{figure}
\centering
\includegraphics[width=\columnwidth]{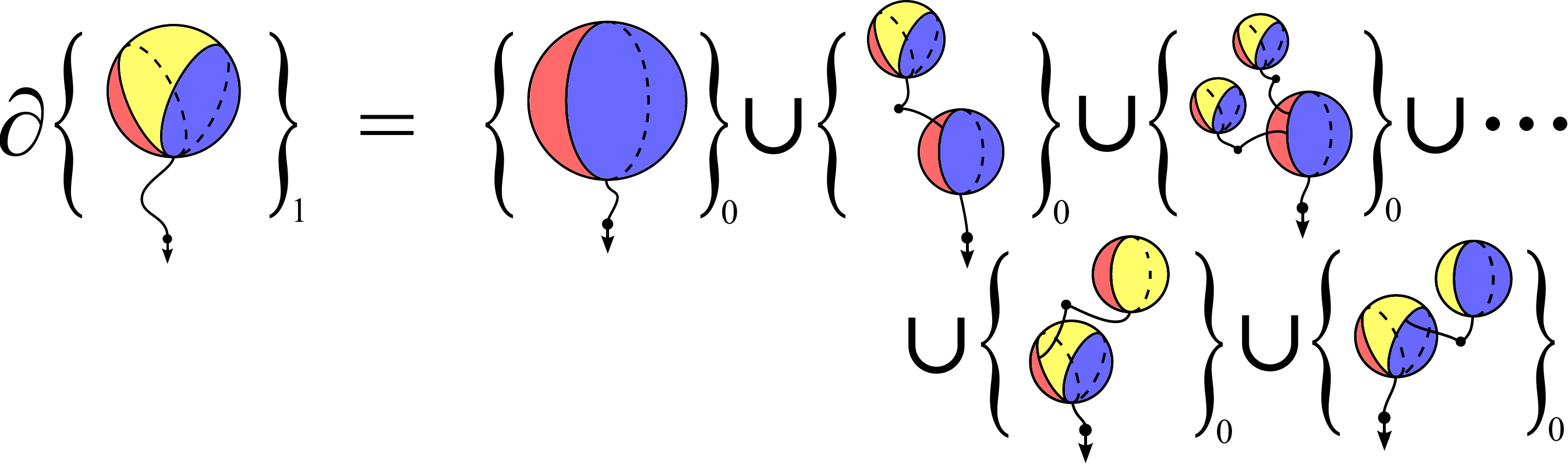}
\caption{
The expected boundary strata of a Morsified figure eight bubble tree
moduli space explain various algebraic identities in Remark~\ref{rmk:b02}.
If contributions from disk bubbles on $L_{01},L_{12}$ can be excluded (i.e.\ \(\mu_{01}^0 = \mu_{12}^0 = 0\)) then the element \(b_{02} \in CM(L_{01}\times_{M_1} L_{12})\) should solve the Maurer--Cartan equation for \(L_{01} \circ L_{12}\).
For monotone, embedded composition, the figure eight bubbles can be excluded to explain the identity $n_{L_{01}} + n_{L_{12}} = n_{L_{01}\circ L_{12}}$ between disk counts.
\label{fig:MC}
}
\end{figure}

\begin{remark} \label{rmk:b02}
The vanishing \(\partial_{b_{02}}^2 = 0\) generally follows from the identification of differentials $\partial=\partial_{b_{02}}$ in \eqref{eq:mu1relation} together with the assumption $\partial^2=0$. 
In more special cases we expect this to be a consequence of \(b_{02} \in CF(L_{01}\circ L_{12}, L_{01}\circ L_{12})\) being a bounding cochain, i.e.\ satisfying the Maurer--Cartan equation \(\sum_{d=0}^\infty \mu_{02}^d(b_{02},\ldots,b_{02}) = 0\).
(Here we assume \(L_{01} \circ L_{12}\) to be an embedded composition, though this remark should extend to the cleanly-immersed setting.)
This should follow from a cobordism argument illustrated in Figure~\ref{fig:MC}: Consider the 1-dimensional moduli space of Morsified figure eight bubble trees
between \(L_{01}\) and \(L_{12}\), with a half-infinite outgoing Morse trajectory on \(L_{01}\times_{M_1}L_{12}\) attached to its singularity.
Extrapolating from the boundary analysis in \S\S\ref{ss:HF}--\ref{ss:Morse}, we expect the 0-dimensional boundary strata to come in two types (where we simplify somewhat for the sake of clarity -- for instance, we do not mention trees of disk bubbles in which all Morse edges are unbroken): 
\begin{itemize}
\item 
A sequence of strata for each $d\geq 0$ is given by squashed eights with seam in \(L_{01}\times_{M_1} L_{12}\), with one outgoing half-infinite Morse trajectory on \(L_{01}\times_{M_1}L_{12}\) attached to the singularity, and \(d \geq 0\) figure eight bubbles attached to the seam via once-broken Morse trajectories on \(L_{01} \circ L_{12}\).
The formal sum over the limiting critical points of the outgoing trajectories of such isolated solutions yields \(\mu_{02}^d(b_{02},\ldots,b_{02})\).

\item The remaining strata are given by figure eights with one outgoing half-infinite Morse trajectory attached to the singularity, and a disk bubble mapping to \((M_k^-\times M_{k+1}, L_{k(k+1)})\) for either $k=0$ or $k=1$ attached to the \(L_{k(k+1)}\)-seam via a once-broken Morse trajectory on \(L_{k(k+1)}\).
The formal sum over the limiting critical points of such isolated solutions yields \(C^2(\:|\: \mu_{01}^0)\) resp.\ \(C^2(\mu_{12}^0 \:|\:)\) when \(k = 0\) resp.\ \(k=1\), where \(C^2\) is the curved \(A_\infty\)-bifunctor whose blueprint we sketch in \S\ref{ss:algebra2}.
\end{itemize}

As boundaries of a $1$-dimensional moduli space, the algebraic contributions of these boundary strata should sum to zero.
(In fact, this equation is one of the curved $A_\infty$-bifunctor relations satisfied by $C^2$, and the cobordism argument we have just made is a special case of the argument made in \S\ref{ss:algebra2}.)
In the special case \(\mu_{01}^0 = \mu_{12}^0 = 0\) (i.e.\ when there are no disk bubbles on $L_{01}$ or $L_{12}$, or their contributions cancel) this yields the expected Maurer--Cartan equation \(\sum_{d\geq 0} \mu_{02}^d(b_{02},\ldots,b_{02}) = 0\).  

This also illuminates an identity between disk counts noted in \cite[Remark~2.2.3]{isom}: Working with monotone orientable Lagrangians and embedded composition, one expects both differentials $\partial_{\delta}\coloneqq\mu^1_{(\delta>0)}$ and $\partial_{0}\coloneqq\mu^1_{(\delta=0)}$ to square to multiples of the identity, $\partial_{\delta}^2 = w_\delta {\rm id}$ resp.\ $\partial_{0}^2 = w_0 {\rm id}$, with 
$w_\delta= n_{L_0} + n_{L_{01}} + n_{L_{12}} + n_{L_2}$ resp.\ $w_0= n_{L_0} + n_{L_{01}\circ L_{12}} + n_{L_2}$ given by sums of counts $n_L$ of Maslov index $2$ disks through a generic point on the Lagrangian $L$.
Arguing by strip shrinking identifying the differentials, \cite{isom} concluded $w_\delta=w_0$ and hence $n_{L_{01}} + n_{L_{12}} = n_{L_{01}\circ L_{12}}$.
This identity can now also be seen directly from the above cobordism argument: Monotonicity excludes nonconstant figure eight bubbles, which reduces the boundary strata on the right hand side of Figure~\ref{fig:MC} to the first and last two types, corresponding to $n_{L_{01}\circ L_{12}}$ and $n_{L_{01}}$, $n_{L_{12}}$, respectively.
\end{remark}

Next, we relax the unobstructedness to the assumption that the Lagrangians \(L_0, L_{01}, L_{12}, L_2\) are equipped with bounding cochains \(\underline{b}=(b_0, b_{01}, b_{12}, b_2)\) so that the twisted differential $\partial_{\underline{b}}$ (which arises from adding marked points labeled with $b_0$, $b_{01}$, $b_{12}$, resp.\ $b_2$  to the $\delta=1$ quilted Floer trajectories) satisfies $\partial_{\underline{b}}^2=0$.
Then we may add these bounding cochains to the previous strip-shrinking moduli space as incoming Morse edges whose starting points represent $b_0, b_{01}, b_{12}$, resp.\ $b_2$ and whose endpoints correspond to marked points on the seams labeled $L_0, L_{01}, L_{12}$, resp.\ $L_2$ anywhere on the quilted Floer trajectory or the attached bubble trees.
Then an analogous cobordism argument yields a chain homotopy equivalence
\begin{align*}
CF\bigl( (L_0,b_0), (L_{01},b_{01}), (L_{12},b_{12}), (L_2,b_2) \bigr)
\; \simeq\; 
CF\bigl( (L_0,b_0), \bigl(L_{01}\circ L_{12}, {\textstyle \sum_{k,\ell\geq0}} b_{02}^{k|\ell}\bigr), (L_2,b_2) \bigr),
\end{align*}
where the Morse chains \(b_{02}^{k|\ell} \in CF(L_{01}\circ L_{12},L_{01}\circ L_{12})\) are obtained by adding incoming Morse edges to the figure eight bubble trees that define \(b_{02}\eqqcolon b_{02}^{0|0}\). 
More precisely, we allow more general underlying metric tree-pairs, which may now have leaves with dashed outgoing edges; the Morse bubble trees are modified by attaching
\(k\) incoming Morse edges representing $b_{01} \in CF(L_{01},L_{01})$ to the $L_{01}$ seams
anywhere on the bubble tree, and we attach \(\ell\) incoming Morse edges representing $b_{12}$ to the $L_{12}$ seams.
Figure~\ref{fig:prediction} provides an example of Morsified figure eight contributions to the twisted Floer differential for the composed Lagrangian correspondences on the right-hand side of the equivalence.  

This demonstrates, as advertised in the introduction, that the isomorphism of quilted Floer homologies \eqref{eq:HFiso} should generalize in a straightforward fashion to the nonmonotone, cleanly-immersed case as isomorphism of quilted Floer homologies with twisted differentials,
\begin{align*}
HF\bigl( \ldots , (L_{01},b_{01}), (L_{12},b_{12}), \ldots \bigr)
\; \simeq\; 
HF\bigl( \ldots , (L_{01}\circ L_{12}, 8(b_{01},b_{12}) ), \ldots \bigr),
\end{align*}
in which the chain $8(b_{01},b_{12})$ for the composed Lagrangian is obtained from moduli spaces of
Morse bubble trees with inputs $b_{01}$ and $b_{12}$, as described above.
In particular, the chain $8(0,0)=b_{02}$ in \eqref{b02} for vanishing inputs is a generally nonzero count of Morsified figure eight bubble trees.

\begin{remark}
The cobordism argument in Remark~\ref{rmk:b02} that $b_{02}$ satisfies the Maurer--Cartan equation
can be adapated to the situation that \(L_{01}, L_{12}\) are equipped with bounding cochains \(b_{01} \in CF(L_{01}, L_{01})\), \(b_{12} \in CF(L_{12},L_{12})\):
This time, for every \(k, \ell \geq 0\) we consider 1-dimensional moduli spaces of Morsified figure eight bubble trees with one outgoing half-infinite Morse trajectory attached to the singularity, and \(k\) resp.\ \(\ell\) incoming half-infinite Morse trajectories attached to the \(L_{01}\)-seam resp.\ the \(L_{12}\)-seam.
The algebraic contributions of the boundary strata with incoming Morse critical points representing \(b_{01}\) and \(b_{12}\) should sum to zero. 
Summing over all \(k, \ell \geq 0\), we obtain the expected equation \begin{align*}
&\sum_{k, \ell \geq 0}\sum_{k_1+\cdots+k_d=k,\atop
\ell_1+\cdots+\ell_d=\ell} \mu_{02}^d\bigl(b_{02}^{k_d|\ell_d}, \ldots, b_{02}^{k_1|\ell_1}\bigr) \\
&\hspace{1in} = \sum_{k,\ell\geq 0}\sum_{a+d\leq k} C^2\Bigl(\underbrace{b_{12}, \ldots, b_{12}}_\ell \:|\: \underbrace{b_{01}, \ldots, b_{01}}_{k-a-d}, \mu_{01}^d(b_{01}, \ldots, b_{01}), \underbrace{b_{01}, \ldots, b_{01}}_a\Bigr) \\
&\hspace{1.5in} + \sum_{k,\ell\geq 0}\sum_{a+d\leq \ell} C^2\Bigl(\underbrace{b_{12}, \ldots, b_{12}}_{\ell-a-d},\mu_{12}^d(b_{12},\ldots,b_{12}),\underbrace{b_{12},\ldots,b_{12}}_a \:|\: \underbrace{b_{01}, \ldots, b_{01}}_k\Bigr).
\end{align*}
The right side vanishes, after a reorganization, by the Maurer--Cartan equations for $b_{01}$ and $b_{12}$.
Then a reorganization of the left side yields the expected Maurer--Cartan equation for \(L_{01}\circ L_{12}\),
\begin{align*}
\textstyle\sum_{d\geq 0} \mu_{02}^d\Bigl( {\textstyle\sum_{k,\ell\geq0}} b_{02}^{k|\ell}, \ldots, {\textstyle\sum_{k,\ell\geq0}} b_{02}^{k|\ell}\Bigr) = 0.
\end{align*}
\end{remark}

\subsection{Conjectural categorical framework for geometric composition} \label{ss:algebra2}

\begin{figure}
\centering
\def\svgwidth{6in}
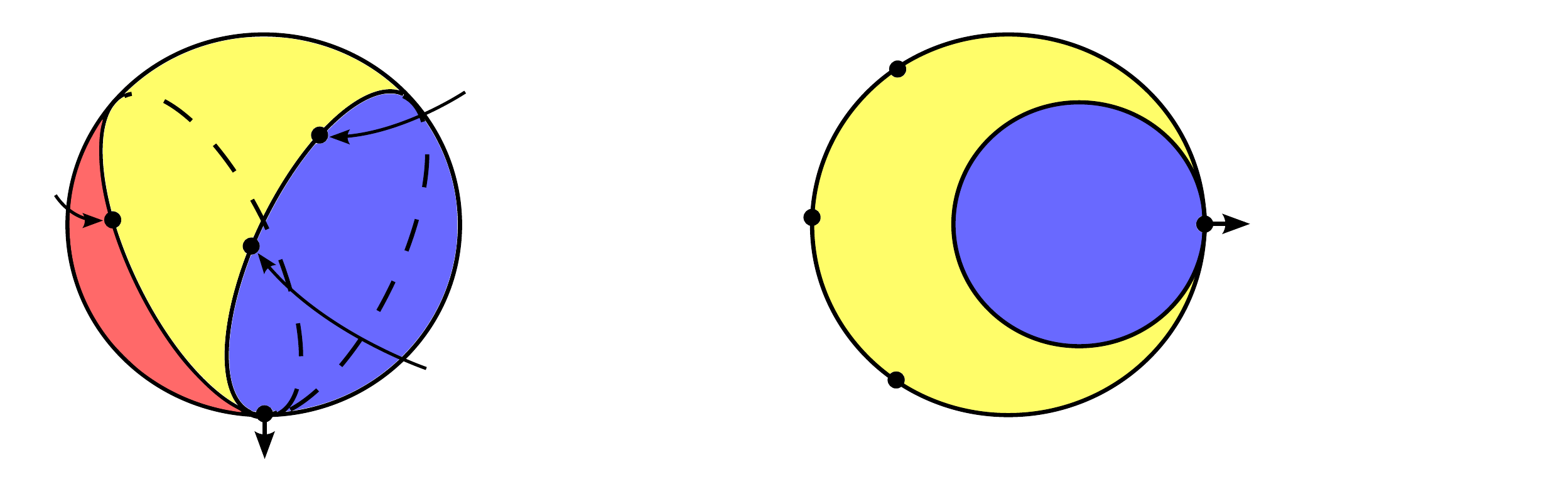
\caption{The \(A_\infty\)-bifunctor \(\comp\) will be defined by counting figure eight quilts with input marked points on the seams as on the left of this figure. 
The \(A_\infty\)-functor \(\comp_{L_{01}}\) is then obtained by setting \(M_2\coloneqq\pt\), which amounts to eliminating the \(M_2\)-patch and replacing the Lagrangian correspondences \(L_{12}^i\) with Lagrangians
\(L_1^i \subset M_1^-\).
So \(\comp_{L_{01}}\) is defined by counting quilted disks as on the right of this figure.
\label{fig:functordef}
}
\end{figure}

The bounding cochains \(b_{02}^{k|\ell}\) in the previous section can be viewed as the result of a map $CF(L_{12},L_{12})^{\otimes k} \otimes CF(L_{01},L_{01})^{\otimes\ell} \to CF(L_{01}\circ L_{12},L_{01}\circ L_{12})$ applied to $ b_{12}^{\otimes k} \otimes b_{01}^{\otimes \ell}$.
More generally, we propose that the figure eight should not be viewed as an undesirable obstruction, but as a geometric object whose raison d'\^{e}tre is to give rise to maps between Floer chain groups involving geometric composition.  

For that purpose note that the moduli spaces of quilted disks constructed in \cite{mw} to represent Stasheff's multiplihedra can be viewed as configuration spaces of marked points on one of the two seams of the domain of figure eight bubbles.
Allowing for marked points on both seams is analogous to the construction of ``biassociahedra'' in \cite{mww}, and will again yield families of marked quilted surfaces parametrized by polyhedra.
We can then build moduli spaces of pseudoholomorphic quilts whose domains are given by points in a biassociahedron. 
More precisely, we modify the figure eight bubble by placing \(d \geq 0\) input marked points on the \(12\)-seam (between the \(M_1\) patch and the \(M_2\) patch), \(e \geq 0\) input marked points on the \(01\)-seam, and one output marked point at the singular point of the quilt (the left half of Figure~\ref{fig:functordef} is the \(d=1, e=2\) case).
The \(12\)-seam is now divided into \(d+1\) segments which we label by Lagrangians \(L_{12}^0, \ldots, L_{12}^d \subset M_1^- \times M_2\).
Similarly, we label the segments of the \(01\)-seam by \(L_{01}^0, \ldots, L_{01}^e \subset M_0^- \times M_1\).
Given a finite-energy pseudoholomorphic quilt with this domain, and assuming that \(L_{01}^0 \circ L_{12}^0\) and \(L_{01}^e \circ L_{12}^d\) are cleanly immersed and intersect each other cleanly, it follows\footnote{
\cite{b:singularity} considers figure eight bubbles with no input marked points (i.e.\ with seams mapping to single Lagrangians \(L_{01}, L_{12}\)).
However, the proof is local at the singularity, so it also applies to figure eights with marked points.} from 
\cite[Removal of Singularity Thm.~2.2]{b:singularity} that the limit of the three maps at the output marked point is a generator of \(CF(L_{01}^0 \circ L_{12}^0, L_{01}^e \circ L_{12}^d)\).
The 0-dimensional moduli space of such quilts should therefore define a map
\begin{figure}
\centering
\def\svgwidth{6in}
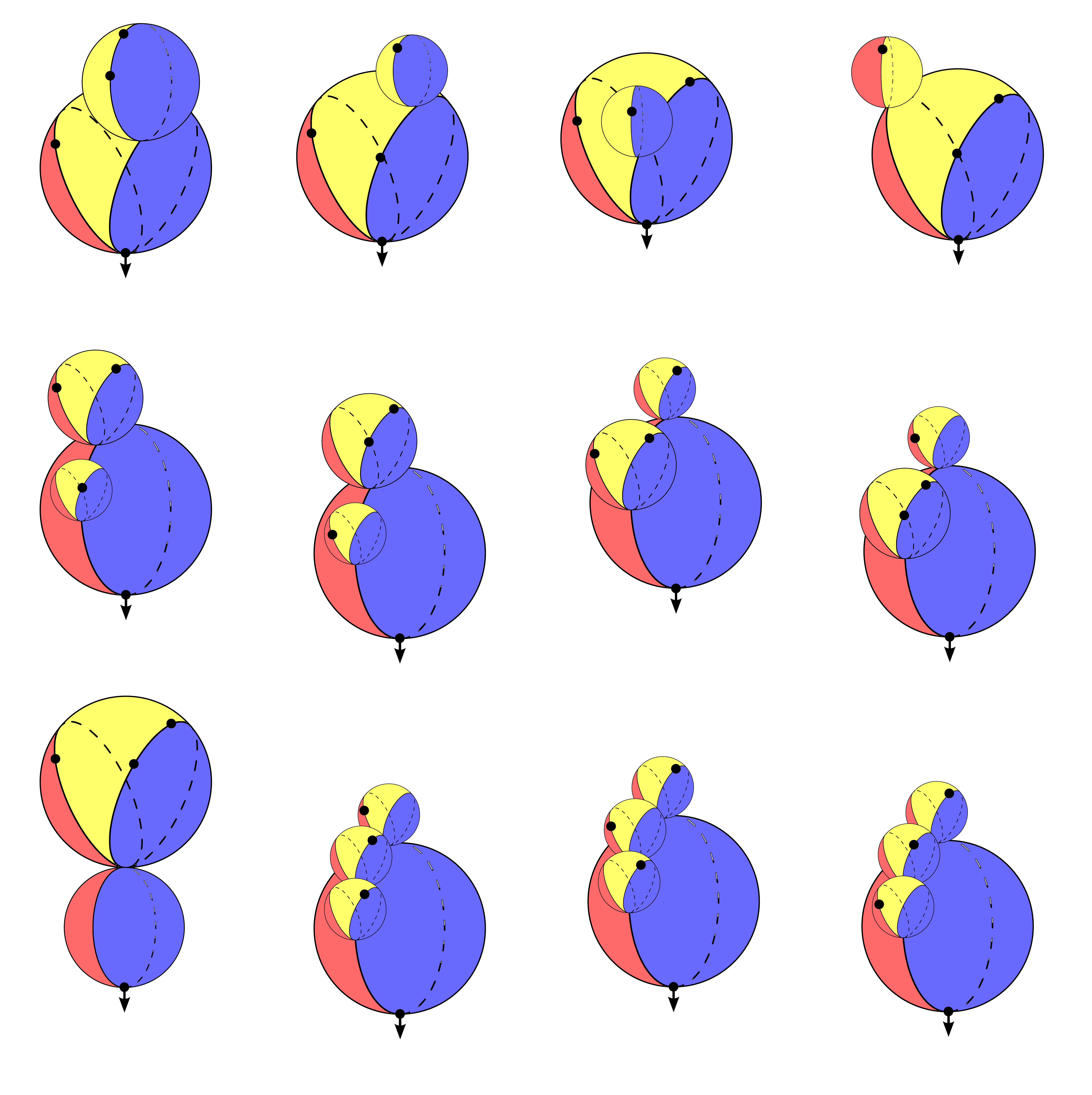
\caption{
These quilted surfaces represent on the one hand the algebraic expressions in the bifunctor relation of Conjecture~\ref{conj:bifunctor} (with the exception of curvature terms), and on the other hand the expected boundary strata of the 1-dimensional moduli space of figure eight bubbles with one marked point on the \(01\)-seam and two on the \(12\)-seam (with the exception of bubbling that does not involve marked points).
\label{fig:bifunctorrel}
}
\end{figure}
\begin{align}
& \comp(- \:|\: -) \; \colon \;  
CF(L_{12}^{d-1}, L_{12}^d) \otimes \cdots \otimes CF(L_{12}^0, L_{12}^1)  \nonumber \\
&\hspace{1in} 
\otimes CF(L_{01}^{e-1}, L_{01}^e) \otimes \cdots \otimes CF(L_{01}^0, L_{01}^1)
\;\;\lra\;\; CF(L_{01}^0 \circ L_{12}^0, L_{01}^e \circ L_{12}^d). \label{eq:bifunctor}
\end{align}
The boundaries of 1-dimensional moduli spaces of such quilts will then give rise to a collection of relations among these maps. These boundary components will arise from several effects.
Firstly, the underlying biassociahedron of quilted surfaces has boundary.
In the example of Figure~\ref{fig:bifunctorrel}, its boundary strata correspond to the eight algebraic terms not involving $\mu^1$-terms.
Secondly, Floer trajectories can break off at any input marked point.
In the example of Figure~\ref{fig:bifunctorrel}, this corresponds to the three algebraic terms in the first row involving pre-composition with $\mu^1_{01}$ or $\mu^1_{12}$.
Moreover, energy concentrating at the outgoing marked point (where in cylindrical coordinates two pairs of $01$- and $12$-seams approach each other asymptotically) can be captured geometrically as a Floer trajectory for the composed Lagrangians breaking off.
In the example of Figure~\ref{fig:bifunctorrel}, this corresponds to the algebraic term in the bottom left corner involving post-composition with $\mu^1_{02}$.
Together, these algebraic terms capture the relations describing an \(A_\infty\)-bifunctor.

Finally, energy can concentrate without marked points being involved, yielding sphere, disk, or figure eight bubbles.
Spheres will be interior points of the ambient polyfold, hence do not contribute to the algebraic relation.
Disk bubbling can appear on a $01$- or $12$-seam, yielding algebraic terms involving pre-composition with $\mu_{01}^0$ or $\mu_{12}^0$, which reflect curvature of an \(A_\infty\)-algebra associated to a Lagrangian $L_{01}^i$ or $L_{12}^i$.
Since figure eight bubbling does not add to the corner index, we expect additional boundary faces arising from adding any number of figure eight bubbles without marked points to the $02$-seams of these configurations.
Algebraically, this will be reflected by \(\comp( \: | \: )\) terms (with $d=e=0$) in any number of entries of \(\mu_{02}^k\), meaning that the \(A_\infty\)-bifunctor itself is curved.
More precisely, we expect for each \(d, e \geq 0\) and fixed $(y_d,\ldots,y_1), (x_e,\ldots,x_1)$ the following relation:
\begin{align} \label{eq:Aoobifun}
& \sum
\comp\bigl(y_d, \ldots, y_{k+\ell+1}, \mu_{12}^\ell(y_{k+\ell},\ldots, y_{k+1}), y_k, \ldots, y_1 \:|\: x_e, \ldots, x_1\bigr) \:\:\: + \nonumber \\
&\hspace{0.15in} + \: \sum
\comp\bigl(y_d, \ldots, y_1 \:|\: x_e, \ldots, x_{k+\ell+1}, \mu_{01}^\ell(x_{k+\ell}, \ldots, x_{k+1}), x_k, \ldots, x_1\bigr) \nonumber \\
&\hspace{0.3in} = \sum \mu_{02}^n\bigl(
\comp(y_{j_n}, \ldots, y_{j_{n-1}+1} \:|\: x_{i_n}, \ldots, x_{i_{n-1}+ 1}), \ldots,
\comp(y_{j_1}, \ldots, y_{j_0+1} \:|\: x_{i_1}, \ldots, x_{i_0+1}) \bigr).
\end{align}
Here the first sum is over $0\le k \le d, 0\le \ell\le d-k$, the second sum is over $0\le k \le e, 0\le \ell\le e-k$, and the right hand sum is over $n\ge 1$ and order respecting partitions into $n$ (possibly empty) tuples $(x_e,\ldots,x_1)=(x_{i_n},\ldots,x_{i_{n-1}+1}) \cup \ldots \cup (x_{i_1},\ldots,x_{i_0+1})$, and
$(y_d,\ldots,y_1)=(y_{j_n},\ldots,y_{j_{n-1}+1}) \cup \ldots \cup (y_{j_1},\ldots,y_{j_0+1})$.
For example, the $(d,e)=(0,0)$ relation is
\begin{align*}
\comp(\mu_{12}^0 \:|\: )  + \comp( \:|\: \mu_{01}^0)  =
{\textstyle \sum_{n=1}^\infty} \; \mu_{02}^n\bigl(\comp( \:|\: ), \ldots, \comp( \:|\: )\bigr).
\end{align*}
The twelve summands not involving curvature terms in the \(d = 1, e = 2\) case, together with their corresponding boundary strata, are shown in Figure~\ref{fig:bifunctorrel}.
These relations are exactly what is required of a curved \(A_\infty\)-bifunctor \cite[Def.~8.8]{pretriangulated}, so we are led to the following conjecture.

\begin{conjecture} \label{conj:bifunctor}
For compact symplectic manifolds \(M_0, M_1, M_2\) there is a curved \(A_\infty\)-bifunctor \begin{align*}
\comp\colon (\Fuk(M_1^- \times M_2), \Fuk(M_0^- \times M_1)) \to \Fuk(M_0^- \times M_2)
\end{align*}
between Fukaya categories of cleanly-immersed Lagrangians that associates to a pair of Lagrangians \((L_{12}, L_{01})\) its geometric composition \(L_{01} \circ L_{12}\) -- if cleanly-immersed -- and is defined on the morphism level by the maps \(\comp(- \:|\: -)\) in \eqref{eq:bifunctor} if all composed Lagrangians have clean intersections.
\end{conjecture}

Morse edges can be incorporated into the above discussion of $C^2$ to overcome the differential-topological challenges of constructing these structures, but they would not have contributed to the clarity of the algebraic exposition.

\medskip

\begin{remark}[Immersed Fukaya categories] \label{rem:immfuk}
Conjecture~\ref{conj:bifunctor} is naturally stated in the immersed setting: Even if the source Fukaya categories were chosen to contain only embedded Lagrangians as objects, the target Fukaya category would still need to contain immersed Lagrangians, since Hamiltonian perturbation of Lagrangian correspondences \(L_{01}, L_{12}\) can always achieve cleanly-immersed but generally not embedded composition.
The Fukaya categories in Conjecture~\ref{conj:bifunctor} will have as objects immersed Lagrangians \(\varphi\colon L \to M\) with clean self-intersections; boundary conditions in \(L\) for a map \(u\colon \Sigma \to M\) on \(\gamma \subset \partial \Sigma\) will require the data of a continuous lift of \(u|_\gamma\) to \(L\).
Cleanly-immersed Lagrangian boundary conditions have been discussed in various settings before, see e.g.\ \cite{a:immersed, aj, chan, cel:switch, se:genus2}, but \cite{bw:bigkahuna} will develop both analysis and algebra from scratch -- the first since we need an abstract regularization framework (such as polyfold theory) to work in general symplectic manifolds, 
and the second since our version of the immersed Fukaya category requires control of sheet-switching at self-intersection points in terms of chain labels which encode contributions from sheet-switching figure eight bubbles.
\end{remark}

\begin{remark}
Special cases of Conjecture~\ref{conj:bifunctor} will yield \(A_\infty\)-functors similar to the ones constructed in \cite{mww}.
The main differences are that \cite{mww} works with \emph{extended} Fukaya categories (whose objects are composable sequences of embedded Lagrangian correspondences) and is necessarily limited to settings (such as monotonicity) in which figure eight bubbling is excluded.  

\begin{enumlist}
\item[(i)]
For \(M_2=\pt\) the restriction of \(\comp\) to a fixed unobstructed object \(L_{01}\in \Fuk(M_0^-\times M_1)\) yields a curved \(A_\infty\)-functor \(\comp_{L_{01}}\colon \Fuk(M_1^-) \to \Fuk(M_0^-)\).
On the object level, this functor sends \(L_1 \subset M_1^-\) to \(L_{01} \circ L_1 \subset M_0^-\) if this composition is cleanly immersed; on the morphism level, this functor is defined by the maps \(\comp( - \:|\: )\) arising from the moduli spaces of quilted disks with a nonnegative number of marked points on the boundary circle, as on the right of Figure~\ref{fig:functordef}.
\item[(ii)]
For \(M_0=\pt\) the restriction of \(\comp\) to a fixed unobstructed object \(L_{12}\in \Fuk(M_1^-\times M_2)\) yields a curved \(A_\infty\)-functor \(\phantom{.}_{L_{12}}{\comp}\colon \Fuk(M_1) \to \Fuk(M_2)\) that sends \(L_1 \subset M_1\) to \(L_1 \circ L_{12} \subset M_2\) if this composition is cleanly immersed, and is defined by the maps \(\comp( \:|\: - )\) on morphism level.
 \item[(iii)] 
We expect the special case $\comp_{\Delta_M}$ resp.\ $\!\!\phantom{.}_{\Delta_M}{\comp}$ of both functors to be the identity functor on $\Fuk(M\times \pt)\cong \Fuk(M)\cong\Fuk(\pt^-\times M)$.
More generally, we will show in \cite{bw:bigkahuna} that \((\comp_{L_{12}^T}, \!\phantom{.}_{L_{12}} \comp)\) and \((\!\phantom{.}_{L_{12}}{\comp}, \comp_{L_{12}^T})\) form adjoint pairs, where \(L_{12}^T\subset M_2^-\times M_1\) is obtained from \(L_{12}\subset M_1^-\times M_2\) by exchange of factors.
\end{enumlist}
\end{remark}

The constructions we have described should also have algebraically flat (i.e.\ uncurved) incarnations, by a general construction:
Given a curved \(A_\infty\)-category \(\cA\), one can form its associated flat \(A_\infty\)-category \(\ol\cA\), an object of which is an object \(X\) of \(\cA\) together with a bounding cochain \(b \in \hom_{\cA}(X,X)\).
A curved \(A_\infty\)-multifunctor \(F\colon (\cA_k,\ldots,\cA_1) \to \cB\) induces a flat \(A_\infty\)-multifunctor \(\ol F\colon (\ol \cA_k, \ldots, \ol \cA_1) \to \ol \cB\), where \(\ol F\) is defined on the object level by \begin{align*}
\ol F\bigl((X_k, b_k), \ldots, (X_1, b_1)\bigr) \coloneqq \Bigl( F(X_k,\ldots,X_1), \sum_{\ell_1, \ldots, \ell_k \geq 0} F\Bigl(\underbrace{b_k, \ldots, b_k}_{\ell_k} \:|\: \cdots \:|\: \underbrace{b_1, \ldots, b_1}_{\ell_1}\Bigr) \Bigr).
\end{align*}
Hence \(\comp\), once defined, descends to a flat \(A_\infty\)-bifunctor \begin{align*}
\ol \comp\colon (\ol\Fuk(M_1^- \times M_2), \ol\Fuk(M_0^- \times M_1)) \to \ol\Fuk(M_0^- \times M_2).
\end{align*} 
In particular, given \((L_{01}, b_{01}) \in \ol\Fuk(M_0^-\times M_1)\), we can define \(\ol \comp_{(L_{01},b_{01})}\colon\ol\Fuk(M_0) \to \ol\Fuk(M_1)\).
Observe that even if \(L_0\) and \(L_{01}\) are unobstructed, \(\comp_{L_{01}}(L_0)=L_0 \circ L_{01}\) may be obstructed, but \(\ol \comp_{(L_{01},0)}\) will specify a bounding cochain for \(L_0\circ L_{01}\) given by a count of figure eight bubbles in this context: quilted disks as on the right of Figure~\ref{fig:functordef} with no input marked points.

\medskip

Thus we hope to have demonstrated that figure eight bubbles are not only analytically necessary and manageable, but also in fact natural algebraic contributions when geometric composition of Lagrangian correspondences is involved.
We end with a brief outlook on how the full embrace of figure eight quilts, and their natural generalization, should yield a symplectic \(A_\infty\) 2-category.

\begin{remark}[Outlook on a proposed \(A_\infty\) 2-category]
\label{rmk:2cat}
A natural way to generalize the notion of figure eight quilt is to allow the seams in the domain to be any positive number \(r\) of circles that intersect tangentially at the south pole of the total domain \(S^2\), to allow any number of input marks on the seams, and to regard the south pole as an output mark.
We call such a quilt an ``\(r\)-seam witch ball''.

For $r = 1, 2, 3$ we expect that counts of the rigid $r$-seam witch balls will coincide with existing algebraic constructions involving Fukaya categories: A 1-seam witch ball with labels \(M_0\) and \(M_1\) is the same as a pseudoholomorphic disk with boundary marked points mapping to \(M_0^- \times M_1\), so counting rigid 1-seam witch balls should yield the $A_\infty$ operations $\mu^d$ in $\Fuk(M_0^-\times M_1)$.
2-seam witch balls are expected to yield the $A_\infty$ bifunctor $C^2$ discussed above, which specializes to yield e.g.\ the $A_\infty$ functor $\Phi_{L_{12}} \colon \Fuk(M_1) \to \Fuk(M_2)$ of \cite{mww}.
Finally, counts of certain rigid 3-seam witch balls should produce the homotopy $\Phi_{L_{01} \circ L_{12}} \simeq \Phi_{L_{12}} \circ \Phi_{L_{01}}$ constructed in \cite{mww}.

The structure that unifies these specializations is the collection of posets indexing the underlying domain moduli spaces.
These posets -- which are, in fact, abstract polytopes -- are called {\bf 2-associahedra} and are defined and explored in \cite{b:2ass}.
There is a 2-associahedron $W_\bn$ for every $r \geq 1$ and $\bn \in \Z_{\geq0}^r \setminus \{\mathbf 0\}$, and the elements of $W_\bn$ are the type-$\bn$ tree pairs (as described in Remark~\ref{rmk:gen_tree_pairs}).
The expected specializations to existing algebraic constructions described in the previous paragraph correspond to specializations of the 2-associahedra: For $r=1$, $W_n$ is the $(n-2)$-dimensional associahedron; for $r=2$, $W_{(n,0)}$ is the $(n-1)$-dimensional multiplihedron; and for $r=3$, $W_{n,0,0}$ is the $n$-dimensional bimultiplihedron (as in \cite{mww}).
Going beyond \cite{bw:bigkahuna}, we plan to show that the moduli spaces of pseudoholomorphic witch balls, as the number of seams ranges over all nonnegative integers, give rise to a symplectic $A_\infty$ 2-category (a notion which will be defined in \cite{b:2cat}) which comprises all compact symplectic manifolds as objects, and -- when restricted to monotone objects and morphisms -- contains the monotone symplectic 2-category in \cite{ww:cat} on homology level. 

\end{remark}

\appendix

\section{Proofs of Lemma~\ref{lem:hbar}, Proposition~\ref{prop:weak-remsing}, and Lemma~\ref{lem:rescaled-width}.}
\label{sec:squiggly}

In this appendix we give the proofs of three results from \S\S\ref{sec:bubbles}--\ref{sec:rescale}: an energy quantization for figure eight and squashed eight bubbles (Lemma~\ref{lem:hbar}), a weak removal of singularity (Proposition~\ref{prop:weak-remsing}), and some properties of rescaled width functions that were used in the proof of Theorem~\ref{thm:rescale} (Lemma~\ref{lem:rescaled-width}).

We will prove Lemma~\ref{lem:hbar} by contradiction: Given a sequence of figure eight or squashed eight bubbles with positive energy tending to zero, we rescale to produce a nonconstant tuple of maps, which is a contradiction to the scale-invariance of energy.
Here the convergence of the rescaled maps will be deduced from the following result of \cite{b:singularity}, which establishes \(\cC^\infty\)-compactness given a uniform gradient bound.
It uses the notion of a {\bf symmetric complex structure} on \([-\rho,\rho]^2\), which is a complex structure \(j\) such that the equality \begin{align*}
j(s,t) = -\sigma \circ j(s,-t) \circ \sigma
\end{align*}
holds for any \((s,t) \in [-\rho,\rho]^2\), where \(\sigma\) is the conjugation \(\alpha\partial_s + \beta\partial_t \mapsto \alpha\partial_s - \beta\partial_t\).
(The standard complex structure, for instance, is symmetric.)

\begin{theorem}[Thm.~3.1, \cite{b:singularity}] \label{thm:nonfoldedstripshrink}
There exists \(\eps > 0\) such that the following holds: Fix \(k \in \N_{\geq 1}\), positive reals \(\delta^\nu \to 0\) and \(\rho > 0\), symmetric complex structures \(j^\nu\) on \([-\rho,\rho]^2\) that converge \(\cC^\infty\) to \(j^\infty\) with \(\| j^\infty - i \|_{\cC^0} \leq \eps\), and \(\cC^{k+2}\)-bounded sequences of \(\cC^{k+2}\) almost complex structures \(J_\ell^\nu\) on \([-\rho,\rho]^2\) as in \eqref{eq:J} such that the \(\cC^{k+1}\)-limit of each \((J_\ell^\nu)\) is a \(\cC^\infty\) almost complex structure.

Then if \((v_0^\nu,v_1^\nu,v_2^\nu)\) is a sequence of size-\((\delta^\nu,\rho)\) \((J_0^\nu,J_1^\nu,J_2^\nu, j^\nu)\)-holomorphic squiggly strip quilts for \((L_{01}, L_{12})\) with uniformly bounded gradients, \begin{align*}
\sup_{\nu \in \N, \: (s,t) \in [-\rho,\rho]^2}  |\d v^\nu|(s,t) < \infty,
\end{align*}
then there is a subsequence in which \((v_0^\nu(t - \delta^\nu))\), \((v_1^\nu|_{t=0})\), \((v_2^\nu(t + \delta^\nu))\) converge \(\cC^k_\loc\) to a \((J_0^\infty, J_2^\infty,i)\)-holomorphic size-\(\rho\) degenerate strip quilt \((v_0^\infty, v_1^\infty, v_2^\infty)\) for \(L_{01} \times_{M_1} L_{12}\).

If the inequality \(\liminf_{\nu\to\infty, (s,t) \in [-\rho,\rho]^2} | \d v^\nu |(s,t) > 0\) holds, then \(v_0^\infty, v_2^\infty\) are not both constant.
\end{theorem}

\begin{proof}[Proof of Lemma~\ref{lem:hbar}]\label{proof:hbar}
We begin by proving energy quantization for the figure eight bubble.
Suppose by contradiction that there is a sequence $\ul{w}^\nu=(w_0^\nu,w_1^\nu,w_2^\nu)$ of \((J_0^\nu(\sigma^\nu,0), J_1^\nu(\sigma^\nu,0), J_2^\nu(\sigma^\nu,0))\)-holomorphic nonconstant figure eight bubbles for some \((\sigma^\nu) \subset [-\rho,\rho]\), with energy ${E(\ul{w}^\nu)\to 0}$.
Then, despite dealing with a quilted domain, we can deduce \(\lim_{\nu \to \infty} \sup_{\ell \in \{0,1,2\}} \sup | \d w^\nu_\ell | = 0\) from the mean value inequality $|\rd u(z)|^2 \leq C r^{-2} \int_{B_r(z)} |\rd u|^2$ for pseudoholomorphic maps (see e.g.\ \cite[Lemma~4.3.1]{ms:jh} or \cite[Theorem~1.3, Lemma A.1]{w:quant}).\footnote{
Here we use the metric on \(M_\ell\) that is induced by \(\omega_\ell\) and \(J_\ell^\infty\) by \eqref{eq:metrics}.
Note that for any fixed \(\nu\), the convergence \(J_\ell^\nu \to J_\ell^\infty\) implies that the norm induced by this metric is equivalent to the norm induced by \(\om_\ell\) and \(J_\ell^\nu\); furthermore, the constant of equivalence can be chosen to be independent of \(\nu\).
This in particular yields uniform constants in the mean value inequalities.
}
Indeed, it applies to each of the maps \(w_0, w_1, w_2\) on balls of radius \(\tfrac 1 2\) that do not intersect seams, and it applies to the folded maps \((w_0(s, -\frac 12 -t), w_1(s,-\frac 12 + t))\) and \((w_1(s,\frac 12 -t), w_2(s,\frac 12 + t))\) on partial balls of radius \(\tfrac 1 2\) that intersect the boundary of the domain $\R\times[0,1]$, where these maps are defined, only in $\R\times\{0\}$, where we have Lagrangian boundary conditions in $L_{01}$ resp.\ $L_{12}$.
Together, these balls cover the entire domain of the figure eight, and thus prove the uniform gradient convergence.

Next, since each triple is nonconstant we can find a subsequence (still denoted $(\ul{w}^\nu)_{\nu\in\N}$), an index $\ell_0 \in\{0,1,2\}$, and points $(s^\nu,t^\nu)$ in the domain of $w^\nu_{\ell_0}$ such that $\delta^\nu\coloneqq|\d w^\nu_{\ell_0}(s^\nu,t^\nu)|\geq \half \sup_{\ell\in\{0,1,2\}} \sup| \d w^\nu_\ell |$.
We just showed \(\delta^\nu\to 0\), and we claim that in fact \(\delta^\nu t^\nu \to 0\).
Indeed, this only requires a proof in the case \(|t^\nu| \to \infty\).
In that case we may apply the mean value inequality on balls of radius \(t^\nu - 1\) to obtain \(\delta^\nu t^\nu \to 0\).
By shifting each triple of maps in the \(s\)-direction, we may moreover assume \(s^\nu = 0\) for all $\nu\in\N$.

Now rescale \(v_\ell^\nu(s, t) \coloneqq w_\ell^\nu(s / \delta^\nu, t/\delta^\nu)\) to obtain maps \begin{align*}
v_0^\nu : \R \times (-\infty, \tfrac 1 2\delta^\nu] \to M_0, \qquad v_1^\nu: \R \times [-\tfrac 1 2\delta^\nu, \tfrac 1 2\delta^\nu] \to M_1, \qquad v_2^\nu: \R \times [\tfrac 1 2\delta^\nu, \infty) \to M_2.
\end{align*}
These maps are \(J_\ell(\sigma^\nu,0)\)-holomorphic and satisfy the following seam conditions: \begin{align*}
(v_0^\nu(s,-\tfrac 1 2\delta^\nu), v_1^\nu(s, -\tfrac 1 2\delta^\nu)) \in L_{01}, \qquad (v_1^\nu(s, \tfrac 1 2\delta^\nu), v_2^\nu(s, \tfrac 1 2\delta^\nu)) \in L_{12} \qquad \forall \: s \in \R.
\end{align*}
The rescaling was chosen to ensure an upper bound on the gradient, \(\sup_{\ell \in \{0,1,2\}} | \d v_\ell^\nu | \leq 2\), as well as a lower bound \(| \d v_{\ell_0}^\nu(0, \tau^\nu) | \geq 1\) for \(\tau^\nu: = \delta^\nu t^\nu \to 0\).
Theorem~\ref{thm:nonfoldedstripshrink} implies that the restrictions of \(v_0^\nu(s, t - \tfrac 1 2\delta^\nu)\) resp.\ \(v_2^\nu(s, t + \tfrac 1 2\delta^\nu)\) to \((-1,1) \times (-1,0]\) resp.\ \((-1,1) \times [0,1)\) converge \(\cC^1_\loc\) to maps \(v_0^\infty\) resp.\ \(v_2^\infty\), and that at least one of the limit maps is nonconstant.
This is in contradiction to the scale-invariant energy converging to 0:
\begin{align*}
0 &< \int_{(-1,1) \times (-1,0]} (v_0^\infty)^*\om_0 + \int_{(-1,1) \times [0,1)} (v_2^\infty)^*\om_2 \\
&= \lim_{\nu \to \infty} \left( \int_{(-1,1)\times[0,1)} (v_0^\nu)^*\om_0 + \int_{(-1,1)\times[0,1)} (v_2^\nu)^*\om_2 \right) \\
&= \lim_{\nu \to \infty} \left(\int_{B_{(-\delta^\nu,\delta^\nu)\times[0,\delta^\nu)}} (w_0^\nu)^*\om_0 + \int_{B_{(-\delta^\nu,\delta^\nu)\times[0,\delta^\nu)}} (w_2^\nu)^*\om_2\right) \;\leq\; \liminf_{\nu\to\infty} E(\ul w^\nu) = 0.
\end{align*}
Hence we have proven the existence of a positive lower bound \(\hbar_8 > 0\).

A similar argument establishes energy quantization for squashed eights.
One difference between the two arguments is that the mean value inequality as stated in the literature requires the boundary to map to an \emph{embedded} Lagrangian, so we cannot deduce uniform gradient convergence to zero.
Hence we consider two cases, depending on whether the limit \(L \coloneqq \lim_{\nu \to \infty} \sup_{\ell \in \{0,1,2\}} \sup | \d w_\ell^\nu| \) (which exists after passing to a subsequence) is finite or infinite. \begin{itemize}
\item \(\mathbf{L \in [0, \infty):}\) Center and rescale as in the proof of \(\hbar_8 > 0\).
To deal with the immersed boundary condition, choose a finite open cover \(L_{01} \times_{M_1} L_{12} = \bigcup_{i=1}^N U_i\) such that \(\pi_{02}: L_{01} \times_{M_1} L_{12} \to M_0^- \times M_2\) restricts to an embedding on each \(U_i\).
Since the rescaled maps have uniformly-bounded gradient, and since their boundary values have smooth lifts to \(L_{01} \times_{M_1} L_{12}\), we can pass to a subsequence and bounded domain to work with embedded boundary conditions in some \(\pi_{02}(U_i)\).
Depending on whether \(L\) is finite or infinite, we can then appeal to either standard bootstrapping techniques (e.g.\ \cite[Theorem 4.1.1]{ms:jh}) or Theorem~\ref{thm:nonfoldedstripshrink} to obtain convergence and hence a contradiction.

\item \(\mathbf{L = \infty:}\) Choose points \((s^\nu,t^\nu)\) and \(\ell_0\) so that \(|\d w_{\ell_0}^\nu(s^\nu,t^\nu)| \to \infty\).
\item As in the proof of Theorem~\ref{thm:rescale}, we can apply the Hofer trick to vary the points \((s^\nu,t^\nu)\) slightly and produce numbers \(R^\nu, \eps^\nu\); rescaling by \(v_\ell^\nu(s,t) \coloneqq w_\ell^\nu(s^\nu + s/R^\nu, t^\nu + t/R^\nu)\) produces a sequence of maps which has nonconstant limit.
\end{itemize}
Finally, $\hbar>0$ follows from the above since we have standard lower bounds $\hbar_{S^2}, \hbar_{D^2}>0$, which can be proven by a single mean value inequality applied to balls resp.\ half balls of large radius, see e.g.\ \cite[Proposition 4.1.4]{ms:jh}.
\end{proof}

\noindent For the proof of Proposition~\ref{prop:weak-remsing}, we will need an extension result.  
To state it we will use the following notation for partitioning the closed unit ball \(\ol B_1(0)\subset\R^2\) into four quadrants:\begin{gather} \label{eq:quadrants}
U_0 \coloneqq \{(x, y) \in \ol B(0,1) \: | \: y \leq x, \: y \leq -x\}, \qquad U_1 \coloneqq \{(x, y) \in \ol B(0,1) \: | \: x \geq y, \: x \geq -y\}, \\
U_2 \coloneqq \{(x, y) \in \ol B(0,1) \: | \: y \geq x, \: y \geq -x\}, \qquad U_3 \coloneqq \{(x, y) \in \ol B(0,1) \: | \: x \leq y, \: x \leq -y \}. \nonumber
\end{gather}
The resulting partition of the boundary circle $\partial \ol B_1(0)$ will be denoted by \(A_i \coloneqq U_i \cap \partial \ol B_1(0)\) for \(i = 0,1,2,3 \), and we denote the intersections of these arcs by  \(p_{i(i+1)} \coloneqq A_i\cap A_{i+1}\) for \(i \;{\rm mod}\; 4\).
We denote the length of a path $\sigma_i:A_i \to X_i$ with respect to $g_i$ by $\ell(\sigma_i)\coloneqq\int_{A_i} |\d \sigma_i| $.

\begin{lemma} \label{lem:extension}
Let \((X_i,g_i)\) be Riemannian manifolds equipped with closed $2$-forms $\omega_i$ for $i=0,1,2$, and let \(Y_{01} \subset X_0 \times X_1\), \(Y_{12} \subset X_1 \times X_2\) be closed submanifolds.
Then for every \(\eps > 0\) there exists \(\delta > 0\) such that the following extension property holds:  Suppose that \(\sigma_i:A_i\to X_i\) for $i=0,1,2,3$ are smooth arcs that satisfy
\begin{gather} \label{eq:conds-for-extension}
 \ell(\sigma_i) \leq \delta, \qquad (\sigma_i(p_{i(i+1)}), \sigma_{i+1}(p_{i(i+1)})) \in Y_{i(i+1)} \qquad \forall \: i \;{\rm mod}\;  4.
\end{gather}
Here we denote \(X_3 \coloneqq X_1\), \(Y_{23} \coloneqq Y_{12}^T\), and \(Y_{30} \coloneqq Y_{01}^T\), with $(\cdot)^T$ denoting the interchange of factors in \(X_i \times X_{i+1}\).
Then there exist smooth extensions \(\wt \sigma_i:U_i\to X\) of \(\wt\sigma_i|_{A_i}=\sigma_i\) such that 
\begin{gather*}
\int_{U_i} \wt\sigma_i^*\omega_i \leq \eps , \qquad
(\wt\sigma_i(p), \wt\sigma_{i+1}(p)) \in Y_{i(i+1)} \quad \forall \: p \in U_i \cap U_{i+1} \qquad \forall \: i \;{\rm mod}\; 4.
\end{gather*}
\end{lemma}

\begin{proof}[Proof of Lemma~\ref{lem:extension}]
Set \(a_i \coloneqq \sigma_i(p_{(i-1)i})\), \(b_i \coloneqq \sigma_i(p_{i(i+1)})\) for \(i \;{\rm mod}\; 4\).
For a constant \(\eps' > 0\) that we will fix later in the proof, let us show that if \(\delta\) is chosen small enough, there exist \(x^\pm = (x_0, x_1^\pm, x_1^\pm, x_2) \in Y_{01} \times_{X_1} Y_{12}\) (two lifts of the same point in $Y_{01}\circ Y_{12}\subset X_0\times X_2$) such that the following distances with respect to the metric \(g_0 \oplus g_1 \oplus g_1 \oplus g_2\) are bounded,
\begin{align} \label{eq:closepoint}
\max\bigl\{ d( (b_0, a_1, b_1, a_2), x^+ )\, ,\,  d( (a_0, b_3, a_3, b_2), x^-) \bigr\} \leq \eps' .
\end{align}
Suppose by contradiction that the sequences \((\sigma_0^\nu,\sigma_1^\nu,\sigma_2^\nu,\sigma_3^\nu)\) satisfy \eqref{eq:conds-for-extension} for a sequence \(\delta^\nu \to 0\) but
\begin{equation}\label{eq:contradicte}
\min_{x^\pm=(x_0, x_1^\pm, x_1^\pm, x_2) \in Y_{01} \times_{X_1} Y_{12}} 
\max\bigl\{ d( (b^\nu_0, a^\nu_1, b^\nu_1, a^\nu_2), x^+ )\, ,\,  d( (a^\nu_0, b^\nu_3, a^\nu_3, b^\nu_2), x^-) \bigr\} > \eps' 
\end{equation}
for all $\nu\in\N$ with \(a_i^\nu \coloneqq \sigma_i^\nu(p_{(i-1)i})\), \(b_i^\nu \coloneqq \sigma_i^\nu(p_{i(i+1)})\).
Since $(b_0^\nu,a_1^\nu)\in Y_{01}, (b_1^\nu,a_2^\nu)\in Y_{12}, (a_3^\nu,b_2^\nu)\in Y_{12}, (a_0^\nu,b_3^\nu)\in Y_{01}$ and \(Y_{01}, Y_{12}\) are compact, we may pass to a subsequence and assume that \(a_i^\nu\), \(b_i^\nu\) have limits \(a_i^\infty, b_i^\infty\) as \(\nu \to \infty\).  
These limits have to coincide \(a_i^\infty = b_i^\infty\) since they are the limits of endpoints of the paths $\sigma_i^\nu$ whose length $\ell(\sigma_i^\nu)\leq \delta^\nu$ goes to zero with $\nu\to\infty$.
This gives rise to two lifts $x^+ \coloneqq (a_0^\infty, a_1^\infty, a_1^\infty, a_2^\infty), x^- \coloneqq (a_0^\infty, a_3^\infty, a_3^\infty, a_2^\infty)\in Y_{01} \times_{X_1} Y_{12}$ since $(a_0^\infty, a_1^\infty) =\lim_{\nu\to\infty}(b_0^\nu, a_1^\nu)$ and $(a_0^\infty, a_3^\infty) =\lim_{\nu\to\infty}(a_0^\nu, b_3^\nu)$ are limits in the closed submanifold $Y_{01}$ and $(a_1^\infty, a_2^\infty) =\lim_{\nu\to\infty}(b_1^\nu, a_2^\nu)$ and $(a_3^\infty, a_2^\infty) =\lim_{\nu\to\infty}(a_3^\nu, b_2^\nu)$ are limits in the closed submanifold $Y_{12}$, and they contradict \eqref{eq:contradicte} since both distances converge to zero, e.g. 
$$
d( (b^\nu_0, a^\nu_1, b^\nu_1, a^\nu_2), x^+ )
\leq d(b^\nu_0,a_0^\infty=b_0^\infty) + d(a^\nu_1,a_1^\infty) + d(b^\nu_1,a_1^\infty=b_1^\infty) + d(a^\nu_2,a_2^\infty)
\;\underset{\nu\to\infty}\longrightarrow\; 0 .
$$
With that we may assume to have lifts \(x^\pm=(x_0, x_1^\pm, x_1^\pm, x_2) \in Y_{01} \times_{X_1} Y_{12} \) satisfying \eqref{eq:closepoint} and begin to construct the extensions \(\wt \sigma_i\) by 
\begin{align*}
\wt\sigma_0(0) \coloneqq x_0, \qquad \wt\sigma_1(0) \coloneqq x_1^+, \qquad \wt \sigma_2(0) \coloneqq x_2, \qquad \wt \sigma_3(0) \coloneqq x_1^-.
\end{align*}
To construct $\wt\sigma_i\times \wt\sigma_{i+1}: U_i \cap U_{i+1} \to Y_{i(i+1)}$, note that the given values on both ends of this line segment are at distance at most $\eps'$ in $Y_{i(i+1)}$.
Hence for sufficiently small \(\eps'\) we may use local charts of the submanifolds \(Y_{01}, Y_{12}\) to choose each extension $\wt\sigma_i :\partial U_i \to X_i$ of $\wt\sigma_i|_{A_i}=\sigma_i$ such that they satisfy the seam conditions $(\wt\sigma_i(p), \wt\sigma_{i+1}(p)) \in Y_{i(i+1)}$ and length bound $\ell(\wt\sigma_i|_{\partial U_i})\leq 2\eps' + \delta$. 
By choosing $\eps'$ and $\delta$ sufficiently small, we can moreover ensure that each of these loops lies in contractible charts of $X_i$.  
On the one hand, that allows us to extend the given $\wt\sigma_i:\partial U_i\to X$ to a smooth map $\wt\sigma_i: U_i \to X_i$.
On the other hand, in each such contractible chart $V\subset X_i$ the given $2$-form $\omega_i|_V = \rd \eta_{V}$ has a uniformly bounded primitive $\eta_V\in\Omega^1(V)$, which gives the desired bound 
$$
\int_{U_i} \wt\sigma_i^*\omega_i \;=\; \int_{\partial U_i} \wt\sigma_i^*\eta_V
\;\leq\; \|\eta_V\|_\infty \ell(\wt\sigma_i|_{\partial U_i})
\;\leq\; \|\eta_V\|_\infty (2\eps' + \delta)\;\leq\;\eps
$$
for sufficiently small $\delta>0$.
In fact, we can cover the projections of the compact Lagrangians to the factors $X_i$ with finitely many contractible charts $V$ so that $\|\eta_V\|_\infty$ is uniformly bounded.
This ensures that the choice of sufficiently small $\delta>0$ for given $\eps>0$ is independent of the arcs~$\sigma_i$.
\end{proof}

\begin{proof}[Proof of Proposition~\ref{prop:weak-remsing}]\label{proof:weak-remsing}
To simplify notation we shift domains so that $w_0$ and $w_2$ in case (i) as well as (ii) are parametrized by $\R\times (-\infty,0]$ and $\R\times[0,\infty)$, respectively.
On these domains we rewrite the figure eight energy integral in polar coordinates 
$\wt w_\ell (r,\theta) = w_\ell (r\cos\theta, r\sin\theta)$
to obtain
\begin{align*}
E& \coloneqq \tint w_0^* \omega_0 + \tint w_1^* \omega_1 + \tint w_2^* \omega_2  \\
& =\int_{\R\times (-\infty,0]} r^{-2} |\partial_\theta \wt w_0 |^2 \d s \d t
\; + \int_{\R\times[-\frac12,\frac12]} |\partial_t w_1 |^2 \d s \d t
\; + \int_{\R\times[0,\infty)} r^{-2} |\partial_\theta \wt w_2|^2 \d s \d t \\
& = \lim_{R\to\infty} \int_0^R r^{-1} A(r) \, \d r
\end{align*}
with integrand
\begin{align*}
A(r) \coloneqq
 \int_\pi^{2\pi}  |\partial_\theta \wt w_0 (r,\theta)|^2 \d\theta
+\int_0^\pi |\partial_\theta \wt w_2(r,\theta)|^2 \d\theta  
+ \int_{-\frac12}^{\frac 12} r |\partial_t w_1 (-r,t) |^2 \d t 
+ \int_{-\frac12}^{\frac 12} r |\partial_t w_1 (r,t) |^2 \d t .
\end{align*}
The same holds for squashed eights if we drop the terms involving $w_1$.
By assumption $\int_0^R r^{-1} A(r) \d r$ converges as $R\to\infty$ although $A(r)\geq 0$ and $\int_0^R r^{-1} \d r \to\infty$ as $R\to\infty$.
Hence there exists a sequence $r_i \to \infty$ such that $A(r_i)\to 0$. 
Depending on a $\delta>0$ to be determined and the $\eps>0$ given, we now choose $r_0>1$ sufficiently large such that $A(r_0)\leq\delta$ and $\bigl| E - \int_0^{r_0} r^{-1} A(r) \d r \bigr| \leq \half\eps$.
Denoting by $B_{r_0}^\pm$ the ball of radius $r_0$ around the origin in the halfplanes $\H^+=\R\times[0,\infty)$ resp.\ $\H^-=\R\times(-\infty,0]$, we now have approximated the energy
\begin{align*}
\biggl| E -  \biggl( \int_{B_{r_0}^-} w_0^* \omega_0 + \int_{[-r_0,r_0]\times[-\frac12, \frac 12]} w_1^* \omega_1 + \int_{B_{r_0}^+} w_2^* \omega_2  \biggr) \biggr| \leq \half\eps 
\end{align*}
and bounded lengths of arcs
\begin{align*}
\ell\bigl( w_1|_{\{\pm r_0\} \times [-\frac12, \frac12]} \bigr) \leq \sqrt{\delta/{r_0}} ,
\qquad
\ell\bigl( \wt w_\ell |_{|z|= r_0} \bigr) \leq \sqrt{\pi \delta} \quad\text{for}\; \ell\in\{0,2\}.
\end{align*}
Here the latter for \(\wt w_0\) (and analogously for \(\wt w_2\) and \(w_1\)) follows from the estimate
\begin{align*}
\ell\bigl( \wt w_0|_{|z|= r_0}\bigr)
\leq \int_\pi^{2\pi}  |\partial_\theta \wt w_0 (r_0,\theta)| \d\theta
\leq \sqrt{\pi A(r_0)} .
\end{align*}
Then for sufficiently large \(r_0>1\) and small \(\delta>0\), the maps \((u_0,\wh u_1,u_2)\) will be constructed as extensions of \((w_0|_{B_{r_0}^-},w_1|_{[-r_0,r_0]\times[-\frac12, \frac 12]},w_2|_{B_{r_0}^+})\).
We first pull them back to the quilted sphere by stereographic projection as in Remark~\ref{rmk:stereographic}, to define a quilted map \((v_0,\wh v_1, v_2)\) on the complement of a neighborhood \(N\subset S^2\) of the puncture \((0,0,1)\). 
Thus $N$ is a slightly-deformed ball with diameter of order \(r_0^{-1}\), which we identify with $\overline{B}_1(0)$ in Lemma~\ref{lem:extension} so that the arcs $\sigma_0=v_0|_{\partial N}$ and $\sigma_2=v_0|_{\partial N}$ are reparametrizations of the short paths \((w_0|_{B_{r_0}^-},w_2|_{B_{r_0}^+})\), and $\sigma_1,\sigma_3$ are the two connected components of $v_1|_{\partial N}$, given by reparametrizations of the short paths $w_1|_{[-r_0,r_0]\times[-\frac12, \frac 12]}$.
For sufficiently small $\delta>0$, Lemma~\ref{lem:extension} then provides a smooth extension of \((v_0,\wh v_1,v_2)|_{\partial N}\) by a quilted map \((\wt\sigma_0,\wt\sigma_1,\wt\sigma_3,\wt\sigma_2)\) on $N$ with total symplectic area bounded by \(\frac12 \epsilon\).
After smoothing these extensions near $\partial N$, we finally construct \((u_0,\wh u_1,u_2)\) by pullback of these extended maps.
More precisely, we construct $u_0$ (and similarly $u_2$) by precomposition of $v_0,\wt\sigma_0$ with a smooth bijection from $D^2$ to ${\rm dom}(v_0)\cup U_0$ which maps $1\in\partial D^2$ to the corner of $U_0$.
(This is not possible by a diffeomorphism, but there is a smooth map with vanishing derivatives at $1$.)
To construct \(\wh u_1:[0,2\pi]\times [-\frac12,\frac12]\to M_1\) we pull back $v_1,\wt\sigma_1,\wt\sigma_3$ by a smooth map from $[0,2\pi]\times [-\frac12,\frac12]$ to ${\rm dom}(u_1)\cup U_1\cup U_3$ which on the boundary components $[0,2\pi]\times \{\pm\frac12\}$ (in polar coordinates) coincides with the bijections from $\partial D^2$ to $U_0\cap (U_1\cup U_3)$ resp.\ $U_2\cap (U_1\cup U_3)$ used in the construction of $u_0,u_2$, thus guaranteeing the seam conditions. 
It can moreover be chosen as bijection with the exception of mapping the two edges $\{0,2\pi\}\times [-\frac12,\frac 12]$ to the common corner point $U_1\cap U_3$. 
Smoothness of these maps guarantees smoothness of the pullbacks \((u_0,\wh u_1,u_2)\), and bijectivity on the complement of a zero set guarantees that they have the same symplectic area as the extension of \((v_0,\wh v_1,v_2)|_{\partial N}\).
Finally, $\wh u_1$ by construction is constant equal to $\wt\sigma_1(0)$ on $\{0\}\times [-\frac12,\frac 12]$ and equal to $\wt\sigma_3(0)$  on $\{2\pi\}\times [-\frac12,\frac 12]$, and extends smoothly to an annulus if $\wt\sigma_1(0)=\wt\sigma_3(0)$. 
The latter is guaranteed by the seam conditions on the extensions $\wt\sigma_i$ if \(\pi_{02} : L_{01} \times_{X_1} L_{12} \to L_{01}\circ L_{12}\) is injective.
\end{proof}

\begin{proof}[Proof of Lemma~\ref{lem:rescaled-width}]\label{proof:rescaled-width}
The functions \(\wt f^\nu(s) - \wt f^\nu(0)\) resp.\ \(\wt f^\nu(s) / \wt f^\nu(0)\) are equal to 0 resp.\ 1 at \(s = 0\), so it suffices to show that \((\wt f^\nu(s) - \wt f^\nu(0))^{(k)}\) and \((\wt f^\nu(s) / \wt f^\nu(0))^{(k)}\) converge \(\cC^0_\loc\) to $0$ for every \(k \geq 1\).
This convergence follows from the formulas
\begin{align*}
(\wt f^\nu - \wt f^\nu(0))^{(k)}(s) = (\alpha^\nu)^{1 - k} (f^\nu)^{(k)}(s^\nu + \tfrac s {\alpha^\nu}), \qquad 
(\wt f^\nu / \wt f^\nu(0))^{(k)}(s) = \frac {(\alpha^\nu)^{-k} (f^\nu)^{(k)}(s^\nu + \tfrac s {\alpha^\nu})} {f^\nu(s^\nu)},
\end{align*}
the convergence \(s^\nu \to s^\infty\) and \(\alpha^\nu \to \infty\), and the obedient shrinking \(f^\nu \Rightarrow 0\) in Definition~\ref{def:obedience}.

The domain of $\phi^\nu$ is \([-\alpha^\nu(s^\nu + \rho), -\alpha^\nu(s^\nu - \rho)]\times \R\) which contains $B_R(0)$ for sufficiently large $\nu$ since $s^\nu\to s^\infty\in (-\rho,\rho)$, $\alpha^\nu\to\infty$, and \(f^\nu \sr{\cC^\infty}{\longrightarrow} 0\).
The image concentrates at $(s^\nu,0)$, more precisely we have 
\begin{align*}
\phi^\nu(B_R(0)) \subset B_{R\delta^\nu}(s^\nu,0) , \qquad \delta^\nu\coloneqq\max\bigl\{ (\alpha^\nu)^{-1}, 2 \|f^\nu\|_{\cC^0} \bigr\} \to 0 ,
\end{align*}
since $\bigl| \phi^\nu(s,t) - (s^\nu, 0) \bigr|^2 \leq s^2 / (\alpha^\nu)^2 +  4 \|f^\nu\|_{\cC^0}^2 t^2
\leq  \bigl( \delta^\nu |(s,t)|\bigr)^2$.
Next, we claim that $B_{R\delta^\nu}(s^\nu,0)$ lies in the image of $\psi^\nu$ for sufficiently large $\nu$. 
Indeed, given $y\in B_{R\delta^\nu}(s^\nu,0)$, we can solve 
\begin{align} \label{eq:banach}
y= \psi^\nu(z)
\end{align}
iff there is a solution 
$z \in -\alpha^\nu s^\nu + [-\alpha^\nu\rho,\alpha^\nu\rho]^2$ of \begin{align*}
z = \alpha^\nu\Bigl( y - s^\nu +  i F^\nu\bigl( s^\nu + z/\alpha^\nu \bigr)\Bigr)\eqqcolon H(z).
\end{align*}
The existence of a such a solution follows from Banach's fixed point theorem applied to \(H\).
Indeed, \(H\) is a smooth map from \(-\alpha^\nu s^\nu + [-\alpha^\nu \rho, \alpha^\nu\rho]^2\) to itself since \(y \in B_{R\delta^\nu}(s^\nu,0)\) gives 
\begin{align*}
\bigl| H(z) + \alpha^\nu s^\nu \bigr| &= \bigl|\alpha^\nu(y + i F^\nu(s^\nu + \tfrac z {\alpha^\nu}))\bigr| 
\leq \alpha^\nu ( | s^\nu | + R\delta^\nu + \|F^\nu\|_{\cC^0})
< \alpha^\nu \rho
\end{align*}
for $\nu$ sufficiently large so that \(R\delta^\nu + \|F^\nu\|_{\cC^0} < \rho - |s^\nu| \).
The latter holds for large $\nu$ since the left hand side converges to $0$ while $\rho - |s^\nu| \to \rho - |s^\infty| >0$.
Furthermore, \(H\) is a contraction mapping once $\nu$ is large enough so that $\|F^\nu\|_{\cC^1}<1$,
\begin{align*}
\bigl|H(z) - H(w)\bigr| &= \alpha^\nu \bigl| F^\nu\bigl( s^\nu + z/\alpha^\nu \bigr) - F^\nu\bigl( s^\nu + w/\alpha^\nu \bigr) \bigr| \leq \|F^\nu\|_{\cC^1}|z - w|.
\end{align*}
Therefore Banach's fixed point theorem guarantees a (unique) solution \(z \in -\alpha^\nu s^\nu + [-\alpha^\nu\rho,\alpha^\nu\rho]^2\) of \eqref{eq:banach}, which shows that for \(\nu\gg1\), the image of \(\psi^\nu\) contains \(\phi^\nu(B_R(0))\).
To show that \((\psi^\nu)^{-1} \circ \phi^\nu\) is a well-defined element of \(\cC^\infty(B_R(0), \R^2)\), it remains to show that \(\psi^\nu\) is injective and has a Jacobian with nonvanishing determinant.  

Injectivity again holds once $\|F^\nu\|_{\cC^1}<1$ since 
\begin{align*}
\psi^\nu(z) = \psi^\nu(w) &\iff z - w = i\alpha^\nu\Bigl(F\bigl(s^\nu + \tfrac z {\alpha^\nu}\bigr) - F\bigl(s^\nu + \tfrac w {\alpha^\nu}\bigr)\Bigr) \\
&\implies |z - w| \leq \|F^\nu\|_{\cC^1}|z - w|.
\end{align*}
The Jacobian of \(\psi^\nu\) is given by \begin{align} \label{eq:psijacobian}
\fJ^\nu(s,t) &\coloneqq \Jac(\psi^\nu)(s+it) = (\alpha^\nu)^{-1}\left(\begin{array}{ll}
1 + \partial_s\im F^\nu(s^\nu + \tfrac{s+it}{\alpha^\nu}) & \partial_t\im F^\nu(s^\nu + \tfrac{s+it}{\alpha^\nu}) \\
-\partial_s\re F^\nu(s^\nu + \tfrac{s+it}{\alpha^\nu}) & 1 - \partial_t\re F^\nu(s^\nu + \tfrac{s+it}{\alpha^\nu})
\end{array}\right),
\end{align}
which for \(\nu\gg1\) has nonvanishing determinant since \(F^\nu \sr{\cC^\infty}{\lra} 0\). 
This proves that \((\psi^\nu)^{-1} \circ \phi^\nu\) is a well-defined element of \(\cC^\infty(B_R(0), \R^2)\) for \(\nu\gg1\).

We now restrict to the case \(\alpha^\nu \coloneqq (2f^\nu(s^\nu))^{-1}\).
To establish the \(\cC^\infty_\loc(\R^2,\R^2)\)-convergence of \((\psi^\nu)^{-1} \circ \phi^\nu\) to the map \((s,t) \mapsto (s,t+\tfrac 1 2)\), we begin by noting their equality at $(s,t)=(0,-\tfrac 1 2)$, 
\begin{align*}
\bigl((\psi^\nu)^{-1}\circ \phi^\nu\bigr)(0,-\tfrac 1 2) = (\psi^\nu)^{-1}(s^\nu,-f^\nu(s^\nu)) = (0,0).
\end{align*}
It remains to show \(\cC^\infty_\loc\)-convergence of the Jacobians
\(\Jac\bigl((\psi^\nu)^{-1} \circ \phi^\nu\bigr) \to \Id\).
Using the inverse of \eqref{eq:psijacobian} and abbreviating $Q^\nu(s,t) \coloneqq s^\nu + \bigl((\psi^\nu)^{-1}\circ \phi^\nu\bigr)(s,t) /\alpha^\nu$ we have
\begin{align}
& \Jac\bigl((\psi^\nu)^{-1}\circ \phi^\nu\bigr)(s,t) 
= 
\Bigl(\fJ^\nu\bigl((\psi^\nu)^{-1}(\phi^\nu(s,t))\bigr)\Bigr)^{-1} \cdot\Jac(\phi^\nu)(s,t) \nonumber \\
&\hspace{0.35in} = 
\frac {\left(\begin{array}{ll}
1 - \partial_t\re F^\nu\circ Q^\nu & -\partial_t\im F^\nu\circ Q^\nu \\
\partial_s\re F^\nu\circ Q^\nu & 1 + \partial_s\im F^\nu\circ Q^\nu
\end{array}\right)\left(\begin{array}{ll}
1 & 0 \\
2(f^\nu)'(s^\nu + s/\alpha^\nu)t & 2\alpha^\nu f^\nu(s^\nu + s/\alpha^\nu)
\end{array}\right)} {(1 + \partial_s\im F^\nu\circ Q^\nu)(1 - \partial_t\re F^\nu\circ Q^\nu) + \partial_t\im F^\nu\circ Q^\nu\partial_s\re F^\nu\circ Q^\nu}. 
\label{eq:compjacobian}
\end{align}
The \(\cC^\infty\)-convergence \(F^\nu \to 0\) implies that the first matrix divided by the denominator converges \(\cC^0_\loc\) to the identity.
In fact, this is \(\cC^k_\loc\)-convergence if the derivatives of $Q^\nu$ up to order $k$ are uniformly bounded on compact sets.
In the second matrix we have $2(f^\nu)'(s^\nu + s/\alpha^\nu)t \to 0$ in \(\cC^\infty_\loc\) by the \(\cC^\infty\)-convergence \(f^\nu\to0\) and $(\alpha^\nu)^{-1}\to 0$, and the bottom right entry $\tfrac {f^\nu(s^\nu + 2f^\nu(s^\nu)s)} {f^\nu(s^\nu)} =\tfrac {\wt f^\nu(s)} {\wt f^\nu(0)}\to 1$ converges already in \(\cC^\infty_\loc\) by the first statement of the current lemma. 

This proves \(\cC^0_\loc\) convergence of the Jacobians and thus \(\cC^1_\loc\)-convergence of the maps \((\psi^\nu)^{-1}\circ\phi^\nu\).
Since $Q^\nu$ is given in terms of these maps and $(\alpha^\nu)^{-1}\to 0$, we conclude that its derivatives are uniformly bounded on compact sets, thus the convergence of the Jacobians is in \(\cC^1_\loc\), which implies \(\cC^2_\loc\)-convergence of the maps \((\psi^\nu)^{-1}\circ\phi^\nu\).
Iterating this argument proves the claimed \(\cC^\infty_\loc\) convergence.
\end{proof}

\section{Examples of figure eight bubbles (in collaboration with Felix Schm\"{a}schke)} \label{app:ex}

In this section we provide some examples of figure eight bubbles.
Our first, previously known, examples show that classical holomorphic discs and holomorphic strips give rise to figure eight bubbles, which naturally appear in the strip shrinking limit of Theorem~\ref{thm:rescale} in case $M_1$ or both of $M_0, M_2$ are points.
Of course, in that case, strip shrinking is not needed to identify the respective moduli spaces.

\begin{example} \label{ex:mickeymouse}
Let $M_0$ and \(M_2\) each be a point, let $M_1$ be any symplectic manifold that is either compact or satisfies the boundedness assumptions of Remark~\ref{rmk:noncompact}, and let $L, L' \subset M_1$ be any two compact Lagrangian submanifolds.
Then a $J_1$-holomorphic strip $u_1\colon [-1,1] \times \R \to M_1$ with boundary conditions
\begin{equation*}
u_1(-1,t) \in L,\qquad  u_1(1,t) \in L'\; \qquad \forall t\in \R
\end{equation*}
gives rise to a figure eight bubble in the sense of Definition~\ref{def:8} by setting $u_0 \coloneqq \const \eqqcolon u_2$.
Such bubbles are generally sheet-switching, unless $u_1$ is a self-connecting Floer trajectory.

Here the correspondences \(\{\pt\} \times L\) and \(L' \times \{\pt\}\) have immersed composition if and only if the Lagrangians intersect transversely in $M_1$.
This bubble type in fact appears in the strip shrinking that relates the quilted Floer trajectories for $\bigl(\{\pt\} ,  \{\pt\} \times L ,  L' \times \{\pt\} , \{\pt\}\bigr)$ and $\bigl(\{\pt\} ,  (\{\pt\} \times L)\circ  (L' \times \{\pt\}) , \{\pt\}\bigr)$.
The first are easily identified with the Floer strips for $(L,L')$.
The latter are pairs of strips in $M_0$ and $M_2$, hence there only is one constant trajectory.
All nontrivial Floer trajectories result in a single figure eight bubble on this constant trajectory.
This demonstrates that figure eight bubbling must be reckoned with, even when only considering isolated Floer trajectories.
\end{example}

All following next figure eight bubbles will be constructed as tuples of maps from the following Riemann surfaces with boundary:
\begin{gather*}
  \Sigma_0 = \{ w \in \C \ | \ \nm{w+1}\leq 1\} \less \{0\}, \qquad \Sigma_2 = \{w \in \C \ | \ \nm{w-1} \leq 1\} \less \{0\}, \\
  \Sigma_1 = \{ w \in \C \ | \ \nm{w+1} \geq 1 , \nm{w-1}
  \geq 1\} \less \{0\}.
\end{gather*}
Each of these surfaces is equipped with the complex structure induced by the inclusion into $\C$, and we will ensure that the maps on $\Si_1$ extend smoothly to $\infty\in\CP^1 \cong \C\cup\{\infty\}$.
The coincidence of boundary components in $\C$ then induces seams between the surfaces and thus defines a quilted surface with total space $\CP^1 \less\{0\}$.
It can be identified, by a biholomorphism $\CP^1 \less\{0\}\cong\C\less\{\infty\}$, with the quilted surface underlying the figure eight bubbles in Definition~\ref{def:8}.

\begin{example} \label{ex:mickeymouse2}
Let $M_1$ be a point, let $M_0, M_2$ be any two symplectic manifolds that are either compact or satisfy our boundedness assumptions, and let $L \subset M_0$ and $L' \subset M_2$ be any two compact Lagrangian submanifolds.
Given two punctured holomorphic discs $u_0\colon \Sigma_0 \to M_0$, \(u_2\colon \Sigma_2 \to M_2$ with \(u_0(\partial \Sigma_0) \subset L\), \(u_2(\partial \Sigma_2) \subset L'\), we obtain a figure eight by setting $u_1 \coloneqq \const$.  

Here the correspondences \(L \times \{\pt\}\) and \(\{\pt\} \times L'\) always have embedded composition $L\times L'\subset M_0\times M_2$.
The singularity in these figure eight bubble is already removed by our construction, and for other bubbles of this type can be removed by the standard result for punctured disks, yielding a pair of disk bubbles on $L$ and $L'$.
These could occur in the strip shrinking that, for further Lagrangians $L_0\subset M_0$ and $L_2\subset M_2$, relates the quilted Floer trajectories for $\bigl(L_0, L \times \{\pt\} ,  \{\pt\} \times L' , L_2 \bigr)$ and $\bigl(L_0, L \times L' , L_2 \bigr)$. 
However, these moduli spaces also have an elementary identification, so this type of figure eight bubbling just is an expression of the fact that the moduli space has boundary components where disk bubbles appear at the same $\R$ coordinate on different seams.
Actual boundary, rather than corners, correspond to one of the disk bubbles being constant.
\end{example}

For the final, more nontrivial example, $\CP^n$ will denote the complex projective space equipped with its standard complex structure and with K\"ahler form $\om_{\CP^n}$ associated to the Fubini-Study metric.

\begin{example}
Consider the $S^1$-action on $\CP^3$ given by $u * [z_0:z_1:z_2:z_3]
\coloneqq [uz_0:uz_1:u^{-1}z_2:u^{-1}z_3]$ for any $u \in \{z\in\C : |z|=1\} \cong S^1 $.
This is a Hamiltonian $S^1$-action with Hamiltionian
\[
\mu([z_0:z_1:z_2:z_3]) \coloneqq \frac 12\frac{ \nm{z_0}^2 + \nm{z_1}^2 -
  \nm{z_2}^2 - \nm{z_3}^2}{\nm{z_0}^2 + \nm{z_1}^2 + \nm{z_2}^2 +
  \nm{z_3}^2}\;.
\]
Symplectic reduction at regular values generally gives rise to Lagrangian correspondences, see \cite[Example~2.0.2(e)]{quiltfloer}.
In this case, reduction at \(0\) yields a Lagrangian correspondence between $M_0=M_2\coloneqq\CP^3$ and $M_1\coloneqq\CP^1 \times \CP^1$, equipped with the symplectic structure $\om \coloneqq \om_{\CP^1} \oplus \om_{\CP^1}$. 
Indeed, the level set of the moment map is 
\begin{equation*}
\mu^{-1}(0) = \left\{ [z_0:z_1:z_2:z_3]\in \CP^3 \    \left| \ \nm{z_0}^2 + \nm{z_1}^2 = \nm{z_2}^2 +
      \nm{z_3}^2\right\}\right. 
\end{equation*}
and the quotient map $\pi\colon \mu^{-1}(0) \to M_0 \qu S^1 \cong \CP^1 \times \CP^1$ is given by \([z_0:z_1:z_2:z_3] \mapsto \left([z_0:z_1],[z_2:z_3]\right)\).
With the inclusion $\i \colon \mu^{-1}(0)\to \CP^3$, this gives rise to Lagrangian submanifolds $L_{01}\coloneqq(\i\times\pi)(\mu^{-1}(0))\subset M_0^- \times M_1$ and $L_{12}\coloneqq(\pi\times\i)(\mu^{-1}(0))\subset M_1^- \times M_2$.
Both are diffeomorphic to \(S^3 \times S^2\), hence simply connected; therefore \(L_{01}\) and \(L_{12}\) are monotone, with the same monotonicity constant as \(M_0^-\times M_1\) resp.\ \(M_1^-\times M_2\).
The geometric composition is 
$$
L_{01}\circ L_{12} \;=\; \left\{  ( Z_0 , Z_2 ) \in \mu^{-1}(0) \times \mu^{-1}(0)  \    \left| \  \pi(Z_0)=\pi(Z_2) \right\}\right. \; \subset\; \CP^3 \times \CP^3  ,
$$
and is embedded since $\pi$ is a surjection and determines $\pi(Z_0)=\pi(Z_2)\in M_1$ uniquely.

Now a general idea for constructing figure eight bubbles applies to this case.
The holomorphic map $\C\to \C^4, w\mapsto (w-1,w+1,w^2-1,1)$ induces holomorphic maps to both $\CP^1\times \CP^1$ and $\CP^3$, and on the seams $\{ |w\pm 1| = 1 \}\subset \C$ the latter takes values in $\mu^{-1}(0)$.
Hence the following triple $(u_0,u_1,u_2)$ defines a nonconstant figure eight bubble, 
\begin{align*}
  u_0&\colon\Sigma_0 \to \CP^3 , &&\!\!\!\!\!\!\!\!\!\!\!\!\!\!\!\!\!\!\!\!\!\!\!\!\!\!\!\!\!\!\!\!\!\!\!\!\!\!\!\!\!\!\!\!\!\!\!\!
  w\mapsto \; \, [w-1:w+1:w^2-1:1],\\
  u_1&\colon\Sigma_1 \to \CP^1\times \CP^1,&&\!\!\!\!\!\!\!\!\!\!\!\!\!\!\!\!\!\!\!\!\!\!\!\!\!\!\!\!\!\!\!\!\!\!\!\!\!\!\!\!\!\!\!\!\!\!\!\!
  w\mapsto   \big([w-1:w+1],[w^2-1:1]\big),\\
  u_2&\colon\Sigma_2 \to \CP^3,&&\!\!\!\!\!\!\!\!\!\!\!\!\!\!\!\!\!\!\!\!\!\!\!\!\!\!\!\!\!\!\!\!\!\!\!\!\!\!\!\!\!\!\!\!\!\!\!\!
  w \mapsto \;\, [w-1:w+1:w^2-1:1] .
\end{align*}
Note here that $u_1$ extends continuously to $\infty\in\CP^1$ since $[w-1:w+1]\to [1:1]$ and
$[w^2-1:1] \to [1:0]$ as $w\to\infty$.
Moreover, all maps extend smoothly to $0\in\C$ -- an example of a removable singularity.

This figure eight bubble could occur in the strip shrinking that, for further Lagrangians $L_0, L_2\subset \CP^3$ relates the quilted Floer trajectories for $\bigl(L_0, L_{01} , L_{12} , L_2 \bigr)$ and $\bigl(L_0, L_{01}\circ L_{12} , L_2 \bigr)$.
In particular, if both $L_0,L_2$ are monotone and so-called ``transverse lifts'' of Lagrangians $\ell_0,\ell_2 \subset \CP^1\times\CP^1$, i.e.\ $L_i \pitchfork \mu^{-1}(0)$ and $\pi\colon L_i \cap \mu^{-1}(0) \to \ell_i$ is bijective, then the above figure eight bubble could be an obstruction to the identification of the Floer homologies $HF(\ell_0,\ell_2)\cong HF(L_0,  L_{01} , L_{12} , L_2 )$ in $\CP^1\times\CP^1$ and $HF(L_0\times L_2, L_{01}\circ L_{12})$ in $(\CP^3)^-\times\CP^3$.
\end{example}

\end{document}